\documentclass[12pt]{amsart}

\usepackage{amsmath,amsfonts,amssymb,amsthm,amscd,comment}
\usepackage[colorlinks=true, pdfstartview=FitV, linkcolor=blue, citecolor=blue, urlcolor=blue, breaklinks=true]{hyperref}
\usepackage{latexsym}
\usepackage[all]{xy}
\usepackage{tikz}

\newcommand\C{\mathbb{C}}
\newcommand\Z{\mathbb{Z}}

\newcommand\R{\mathbb{R}}
\newcommand\F{\mathbb{F}}
\newcommand\N{\mathbb{N}}
\newcommand\kk{{\Bbbk}}

\newcommand\id{\mathrm{id}}

\newcommand\aff{\mathrm{aff}}

\newcommand\Proj{\mathrm{Proj}}
\newcommand\Sym{\mathrm{Sym}}

\newcommand\fH{\mathfrak{H}}
\newcommand\fh{\mathfrak{h}}
\newcommand\fA{\mathfrak{A}}
\newcommand\fB{\mathfrak{B}}

\newcommand\bF{\mathbf{F}}

\newcommand\bA{\mathbf{A}}
\newcommand\bB{\mathbf{B}}

\newcommand\cH{\mathcal{H}}

\newcommand\cP{\mathcal{P}}

\newcommand{\vt}{v_\theta}

\DeclareMathOperator{\Hom}{Hom}

\DeclareMathOperator{\End}{End}

\DeclareMathOperator{\Ob}{Ob}

\DeclareMathOperator{\Ind}{Ind}

\newcommand{\rcap}{\mathrm{RCap}}
\newcommand{\lcap}{\mathrm{LCap}}
\newcommand{\rcup}{\mathrm{RCup}}
\newcommand{\lcup}{\mathrm{LCup}}
\newcommand{\res}{\mathrm{Res}}
\newcommand{\ind}{\mathrm{Ind}}
\newcommand{\ucross}{\mathrm{Ucross}}
\newcommand{\dcross}{\mathrm{Dcross}}
\newcommand{\rcross}{\mathrm{Rcross}}
\newcommand{\lcross}{\mathrm{Lcross}}

\theoremstyle{plain}
\newtheorem{theorem}{Theorem}[section]
\newtheorem{proposition}[theorem]{Proposition}
\newtheorem{lemma}[theorem]{Lemma}
\newtheorem{corollary}[theorem]{Corollary}

\newtheorem{conjecture}[theorem]{Conjecture}

\theoremstyle{definition}
\newtheorem{definition}[theorem]{Definition}
\newtheorem*{remark*}{Remark}
\newtheorem{remark}[theorem]{Remark}

\numberwithin{equation}{section}

\allowdisplaybreaks

\newcommand{\comments}[1]{ \begin{center} \parbox{5 in}{{\bf {\footnotesize Comments:  }}{\footnotesize \textit{#1}}} \end{center}}
\begin{document}
%

\title[Heisenberg categorification]{Hecke algebras, finite general linear groups, and Heisenberg categorification}

\author{Anthony Licata}
\address{A.~Licata: Department of Mathematics, Stanford University}

\author{Alistair Savage}
\address{A.~Savage: Department of Mathematics and Statistics, University of Ottawa}
\thanks{The research of the second author was supported by a Discovery Grant from the Natural Sciences and Engineering Research Council of Canada.}



\begin{abstract}
We define a category of planar diagrams whose Grothendieck group contains an integral version of the infinite rank Heisenberg algebra, thus yielding a categorification of this algebra.  Our category, which is a $q$-deformation of one defined by Khovanov, acts naturally on the categories of modules for Hecke algebras of type $A$ and finite general linear groups.  In this way, we obtain a categorification of the bosonic Fock space.  We also develop the theory of parabolic induction and restriction functors for finite groups and prove general results on biadjointness and cyclicity in this setting.
\end{abstract}

\maketitle \thispagestyle{empty}

\tableofcontents

\section{Introduction}

In the 1970's, Geissinger gave a representation-theoretic realization of the bialgebra $\Sym$ of symmetric functions \cite{Gei77}.  He considered the Grothendieck groups of representations of all symmetric groups over a field $\kk$ of characteristic zero and constructed an isomorphism of bialgebras
\[ \textstyle
  \Sym \cong \bigoplus_{n=0}^\infty K_0(\kk[S_n]\text{-mod}).
\]
In Geissinger's construction, the algebra structure is the map on the Grothendieck group induced by the induction functor
\[
  [\ind]: K_0(\kk[S_n]\text{-mod}) \otimes K_0(\kk[S_m]\text{-mod})
  \to K_0(\kk[S_{n+m}]\text{-mod}),
\]
while the coalgebra structure is given by restriction.  Mackey theory for induction and restriction in symmetric groups implies that the coproduct is an algebra homomorphism.
Each class $[V] \in K_0(\kk[S_n]\text{-mod})$ defines an endomorphism of $\bigoplus_{n=0}^\infty K_0(\kk[S_n]\text{-mod})$ given by multiplication by $[V]$.  These endomorphisms, together with their adjoints, define a representation of an infinite rank Heisenberg algebra on the Grothendieck group.

Several generalizations of Geissinger's construction were subsequently given by Zelevinsky in \cite{Zel81}.  Two of these generalizations involve a kind of $q$-deformation: in one the group algebra of the symmetric group $\kk[S_n]$ is replaced by the Hecke algebra $H_n(q)$, and in the other by the group algebra $\kk[GL_n(\F_q)]$ of the general linear group over a finite field.  In these cases, too, endomorphisms of the Grothendieck group given by multiplication by classes $[V]$, together with their adjoints, generate a representation of the Heisenberg algebra.

In addition to the Heisenberg algebra action on the Grothendieck group, the categories
\[ \textstyle
  \bigoplus_{n}\kk[S_n]\text{-mod}, \quad \bigoplus_nH_n(q)\text{-mod}, \quad \text{and} \quad
  \bigoplus_n \kk[GL_n(\F_q)]\text{-mod}
\]
are interesting for another reason: they are all examples of braided monoidal categories \cite{JoyStr93}.  In each case, the monoidal operation is given by an induction functor, the very same functor that defines the algebra structure at the level of the Grothendieck group.  This algebra structure accounts for the action of the creation operators in the Heisenberg algebra,
which are given by multiplication, but not the annihilation operators which are their adjoints.
This suggests that in order to describe a categorical Heisenberg action, one should consider not just induction functors, but also the dual restriction functors.  These functors should descend to a Heisenberg algebra action on the Grothendieck group.  This categorical action should also involve natural transformations that use both the braiding amongst compositions of induction functors and the duality between induction and restriction, information that is lost when passing to the Grothendieck group.

To define such categorical Heisenberg actions is the goal of the current paper.  Precisely, we define a family of categories that act naturally on the above module categories.  This action is a kind of modification of the notion of a braided monoidal category which takes into account induction and restriction functors, their compositions, and natural transformations between these compositions.

In \cite{Kho10}, Khovanov takes a similar perspective in the study of induction and restriction functors between characteristic zero representation categories of symmetric groups.  He defines a monoidal category $\cH$ that acts naturally on the category of representations of all symmetric groups, categorifying Geissinger's construction.
The Grothendieck group of Khovanov's $\cH$ contains, and is conjecturally isomorphic to, an integral form of the Heisenberg algebra.

In the current paper, we define and study a category $\cH(q)$, which is a $q$-deformation of Khovanov's category.  When $q$ is not a nontrivial root of unity, the Grothendieck group of $\cH(q)$ contains, and is conjecturally isomorphic to, an integral form of the Heisenberg algebra.  When $q$ is a root of unity, we expect the category $\cH(q)$ to have an interesting structure, though we do not say much about this in the current paper.  The morphism spaces in $\cH(q)$ are also objects which we believe to be of independent interest.  In particular, a $q$-deformation of the degenerate affine Hecke algebra, which is a subalgebra of the affine Hecke algebra, arises naturally from the morphisms of $\cH(q)$.

Just as Khovanov's category $\cH$ is related to Geissinger's construction, the category $\cH(q)$ is related to both of Zelevinsky's constructions.  In Section ~\ref{sec:Hecke-action}, we show that $\cH(q)$ acts naturally on $\bigoplus_n H_n(q)$-mod, while in Section~\ref{sec:GL-action}, we show that $\cH(q)$ acts on $\bigoplus_n \kk[GL_n(\mathbb{F}_q)]$-mod.  In both cases, passing to the Grothendieck group recovers the Fock space representation of the Heisenberg algebra.  These two representations of the category $\cH(q)$ provide a new perspective on the relationship between Hecke algebras and finite general linear groups.

In \cite{CauLic10}, the first author, together with S.~Cautis, gave graphical categorifications of a family of Heisenberg algebras parameterized by the finite subgroups of $SL_2(\C)$ and related these categorifications to the geometry of Hilbert schemes on ALE spaces.  The constructions of the present paper suggest that a $q$-deformation of the categories in \cite{CauLic10} should also exist.  We expect these deformations will be related to the ``finite subgroups" of $U_q(\mathfrak{sl}_2)$.

Important in the constructions in the current paper is the fact that the Hecke algebras form a so-called \emph{tower of algebras}.  It is natural to expect that other towers of algebras (for further examples, see \cite{BerLi09,Kho10b} and the references therein) give rise to graphical categorifications.

The crucial observation that allows us to present our category $\cH(q)$ using planar topology is the cyclic biadjointness of the defining generating objects $Q_+,Q_-$ of $\cH(q)$.  In any representation of $\cH(q)$, these generators are mapped to biadjoint functors, that is, functors that are both left and right adjoint to each other.  The importance of such functors in low dimensional topology was emphasized in \cite{Kho02} and more recently in subsequent work
\cite{ChuRou08,  KhoLau10, Lau08, Rou09} on categorified quantum groups.
In the last section of this paper, we give some examples of cyclic biadjoint functors arising in the representation theory of finite groups.

The structure of this paper is as follows.  In Section~\ref{sec:hecke-heis-defs}, we recall the definitions of the Hecke algebras of type $A$ and the Heisenberg algebra.  We introduce the category $\cH(q)$ in Section~\ref{sec:category-def} and deduce from the definitions various useful relations.  In Section~\ref{sec:categorification} we present the main result giving a categorification of the Heisenberg algebra.  We then define an action of our category on modules for Hecke algebras and finite general linear groups in Sections~\ref{sec:Hecke-action} and~\ref{sec:GL-action} respectively.  In particular, in Section~\ref{sec:Hecke-action}, we use the action to prove our main theorem.  Finally, in Section~\ref{sec:parabolic}, we describe the cyclic biadjointness of parabolic induction and restriction functors for finite groups.

\subsection*{Acknowledgements}

We would like to thank D.~Bump, S.~Cautis, M.~Khovanov, and M.~Mackaay for useful discussions.  We also thank M.~Khovanov for making available to us the preprint \cite{Kho10}.

\section{Hecke algebras and the Heisenberg algebra} \label{sec:hecke-heis-defs}

In this section, we introduce our main algebraic objects of interest: the Hecke algebra of type $A$ and the Heisenberg algebra.

Let $\kk$ be a ring and let $q$ be either an indeterminate or an invertible element of $\kk$ (in which case $\kk[q,q^{-1}] = \kk$).

\begin{definition}[Hecke algebra] For $n \ge 2$, the \emph{Hecke algebra} $H_n(q)$ is the $\kk[q,q^{-1}]$-algebra generated by $t_1,\hdots t_{n-1}$ with relations
\begin{enumerate}
\item $t_i^2 = q + (q-1)t_i$,
\item $t_it_j = t_jt_i$ for $i,j =1,2,\dots,n-1$ such that $|j-i| > 1$,
\item $t_it_{i+1}t_i = t_{i+1}t_it_{i+1}$ for $i=1,2,\dots,n-2$.
\end{enumerate}
By convention, we set $H_0(q) = H_1(q) = \kk[q,q^{-1}]$.  We let $1_n$ denote the
identity element of $H_n(q)$.  To simplify notation, we will write $H_n$ for $H_n(q)$ in the sequel.
\end{definition}

The algebra $H_n$ has a basis $\{t_w\}_{w \in S_n}$, where for $w \in S_n$, $t_w = t_{i_1} \dots
t_{i_k}$, where $w = s_{i_1} \dots s_{i_k}$ is a reduced expression for $w$.  Note that the generator $t_i$ is invertible with inverse $t_i^{-1} = q^{-1}t_i + (q^{-1}-1)$.

\begin{definition}[Heisenberg algebra]
The (infinite rank) \emph{Heisenberg algebra} $\fh$ is the associative $\C$-algebra with generators $p_i,q_i$, $i \in \N_+ = \{1,2,\dots\}$, and relations
\[
  p_i q_j = q_j p_i + \delta_{i,j} 1,\quad p_i p_j = p_j p_i,\quad q_i q_j = q_j q_i, \quad i,j \in \N_+.
\]
\end{definition}

\begin{remark}
The Heisenberg algebra plays a fundamental role in quantum field theory and the theory of affine Lie algebras.  Is it isomorphic to the \emph{Weyl algebra}, which is the algebra of operators on $\C[x_1,x_2,\dots]$ generated by multiplication by $x_i$, $i \in \N_+$, and partial differentiation $\partial/\partial x_i$, $i \in \N_+$.
\end{remark}

\begin{definition}[Integral form of the Heisenberg algebra]
Let $\fh_\Z$ be the unital ring with generators $a_n,b_n$, $n \in \N_+$, and relations
\begin{equation} \label{eq:integral-heis-relations}
  a_n b_m = b_m a_n + b_{m-1} a_{n-1}, \quad a_n a_m = a_m a_n, \quad b_n b_m = b_m b_n, \quad n,m \in \N_+.
\end{equation}
Here we adopt the convention that $a_0=b_0=1$.
\end{definition}

The fact that $\fh_\Z$ is an integral form of the Heisenberg algebra can be seen as follows (we thank M.~Mackaay for explaining this to us).  Defining generating functions
\[
  A(t) = 1 + a_1 t + a_2 t^2 + \dots, \quad B(u) = 1 + b_1u + b_2u^2 + \dots,
\]
we can rewrite the relations~\eqref{eq:integral-heis-relations} as
\begin{equation} \label{eq:A-B-gen-func-relation}
  A(t)B(u) = B(u)A(t)(1+tu).
\end{equation}
Define
\begin{gather*}
  \tilde A(t) = 1 + tA'(-t)A(-t)^{-1},\quad \tilde B(u) = 1+uB'(-u)B(-u)^{-1}, \\
  \tilde A(t) = 1 + \tilde a_1 t + \tilde a_2 t^2 + \dots,\quad B(u) = 1 + \tilde b_1 u + \tilde b_2 u^2 + \dots.
\end{gather*}
Using~\eqref{eq:A-B-gen-func-relation}, one can show that
\[
  \tilde A(t)B(u) = B(u) \tilde A(t) + B(u) \frac{ut}{1-ut},
\]
from which it follows that
\[
  [\tilde a_n, b_m]=\delta_{m,n} \text{ for } m \le n.
\]
Now note that for each $n \in \N_+$, $\tilde a_n$ is equal to $(-1)^{n-1}na_n$ plus terms involving products of $a_m$ for $m<n$ (and similarly for $\tilde b_n$).  Thus, by symmetry, it follows that
\[
  [\tilde a_n, \tilde b_m] = (-1)^{n-1}\delta_{n,m}n \quad \forall\ m,n \in \N_+.
\]
Therefore the elements $\frac{1}{n}(-1)^{n-1}a_n, b_n$, $n \in \N_+$, satisfy the defining relations of the Heisenberg algebra $\fh$.  It follows that $\fh_\Z \otimes_\Z \C \cong \fh$ and so $\fh_\Z$ is an integral form of $\fh$.

\section{A graphical category} \label{sec:category-def}

\subsection{Definition}
\label{sec:Hq'-def}

We define an additive $\kk[q,q^{-1}]$-linear strict monoidal category $\cH'(q)$ as follows. The set of objects is generated by two objects $Q_+$ and $Q_-$.  Thus an arbitrary object of $\cH'(q)$ is a finite direct sum of tensor products $Q_\varepsilon := Q_{\varepsilon_1} \otimes \dots \otimes Q_{\varepsilon_n}$, where $\varepsilon = \varepsilon_1 \dots \varepsilon_n$ is a finite sequence of $+$ and $-$ signs.  The unit object $\mathbf{1}=Q_\varnothing$.

The space of morphisms $\Hom_{\cH'(q)}(Q_\varepsilon,
Q_{\varepsilon'})$ is the $\kk[q,q^{-1}]$-module generated by  planar diagrams
modulo local relations.  The diagrams are
oriented compact one-manifolds immersed in the strip $\R \times
[0,1]$, modulo rel boundary isotopies.  The endpoints of the
one-manifold are located at $\{1,\dots,m\} \times \{0\}$ and
$\{1,\dots,k\} \times \{1\}$, where $m$ and $k$ are the lengths of
the sequences $\varepsilon$ and $\varepsilon'$ respectively.  The
orientation of the one-manifold at the endpoints must agree with the
signs in the sequences $\epsilon$ and $\epsilon'$ and triple intersections are not allowed.  For example, the diagram
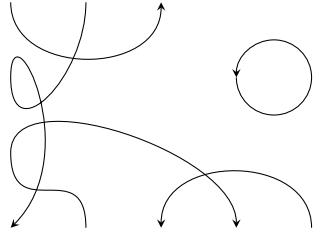
\[
  \begin{tikzpicture}[>=stealth]
    \draw[->] (0,3) .. controls (0,2) and (2,2) .. (2,3);
    \draw[->] (1,3) .. controls (1,2) and (0,1) .. (0,2) .. controls (0,3) and (1,1) .. (0,0);
    \draw[->] (1,0) .. controls (1,1) and (0,0) .. (0,1) .. controls (0,2) and (3,1) .. (3,0);
    \draw[->] (4,0) .. controls (4,1) and (2,1) .. (2,0);
    \draw[->] (3,2) arc(-180:180:.5);
  \end{tikzpicture}
\]
is a morphism from $Q_{-+--+}$ to $Q_{--+}$.  Composition of morphisms is given by the natural gluing of diagrams.  An endomorphism of $\mathbf{1}$ is a diagram without endpoints.
The local relations are as follows.

\begin{equation} \label{eq:local-relation-basic-Hecke}
\begin{tikzpicture}[>=stealth,baseline=25pt]
  \draw (0,0) .. controls (1,1) .. (0,2)[->];
  \draw (1,0) .. controls (0,1) .. (1,2)[->] ;
  \draw (1.5,1) node {=};
  \draw (2,1) node {$q$};
  \draw (2.5,0) --(2.5,2)[->];
  \draw (3.5,0) -- (3.5,2)[->];
  \draw (4.7,1) node {$+\ (q-1)$};
  \draw (5.7,0) --(6.7,2)[->];
  \draw (6.7,0) -- (5.7,2)[->];
\end{tikzpicture}
\end{equation}

\begin{equation} \label{eq:local-relation-braid}
\begin{tikzpicture}[>=stealth,baseline=25pt]
  \draw (0,0) -- (2,2)[->];
  \draw (2,0) -- (0,2)[->];
  \draw (1,0) .. controls (0,1) .. (1,2)[->];
  \draw (2.5,1) node {=};
  \draw (3,0) -- (5,2)[->];
  \draw (5,0) -- (3,2)[->];
  \draw (4,0) .. controls (5,1) .. (4,2)[->];
\end{tikzpicture}
\end{equation}

\begin{equation} \label{eq:local-relation-up-down-double-crossing}
\begin{tikzpicture}[>=stealth,baseline=25pt]
  \draw (0,0) .. controls (1,1) .. (0,2)[<-];
  \draw (1,0) .. controls (0,1) .. (1,2)[->] ;
  \draw (1.5,1) node {=};
  \draw (2,1) node {$q$};
  \draw (2.5,0) --(2.5,2)[<-];
  \draw (3.5,0) -- (3.5,2)[->];
  \draw (4,1) node {$- \ q$};
  \draw (4.5,1.75) arc (180:360:.5) ;
  \draw (4.5,2) -- (4.5,1.75) ;
  \draw (5.5,2) -- (5.5,1.75) [<-];
  \draw (5.5,.25) arc (0:180:.5) ;
  \draw (5.5,0) -- (5.5,.25) ;
  \draw (4.5,0) -- (4.5,.25) [<-];
\end{tikzpicture} \qquad \qquad
\begin{tikzpicture}[>=stealth,baseline=25pt]
  \draw (0,0) .. controls (1,1) .. (0,2)[->];
  \draw (1,0) .. controls (0,1) .. (1,2)[<-] ;
  \draw (1.5,1) node {$=\ q$};
  \draw (2.3,0) --(2.3,2)[->];
  \draw (3.3,0) -- (3.3,2)[<-];
\end{tikzpicture}
\end{equation}

\begin{equation} \label{eq:cc-circle-and-left-curl}
\begin{tikzpicture}[>=stealth,baseline=0pt]
  \draw [<-](0,0) arc(180:0:.5);
  \draw (0,0) arc(180:360:.5);
  \draw (1.5,0) node {$=$};
  \draw (2,0) node {$\id$};
  \draw (0,-1) node {};
  \draw (0,1) node {};
\end{tikzpicture} \qquad \qquad
\begin{tikzpicture}[>=stealth,baseline=0pt]
  \draw (-1,0) .. controls (-1,.5) and (-.3,.5) .. (-.1,0) ;
  \draw (-1,0) .. controls (-1,-.5) and (-.3,-.5) .. (-.1,0) ;
  \draw (0,-1) .. controls (0,-.5) .. (-.1,0) ;
  \draw (-.1,0) .. controls (0,.5) .. (0,1) [->] ;
  \draw (0.7,0) node {$=0$};
\end{tikzpicture}
\end{equation}

Note that the upward crossing is invertible with inverse
\begin{equation} \label{eq:inverse-crossing}
\begin{tikzpicture}[>=stealth,baseline=25pt]
  \draw[->] (0,0) to (1,2);
  \draw[->] (1,0) to (0,2);
  \filldraw[fill=white,draw=black] (0.5,1) circle(2pt);
  \draw (2,1) node {$:= \ q^{-1}$};
  \draw[->] (3,0) to (4,2);
  \draw[->] (4,0) to (3,2);
  \draw (5.5,1) node {$+\ (q^{-1}-1)$};
  \draw[->] (7,0) to (7,2);
  \draw[->] (8,0) to (8,2);
\end{tikzpicture}\quad  .
\end{equation}
The reader should note that the local relations on upward pointing strands are simply the relations of the Hecke algebra, where the generator $t_i$ corresponds to the crossing of the $i$-th and $(i+1)$-st strands (numbered from the right).  The definitions of the other local relations are motivated by the following result.

\begin{lemma} \label{lem:basic-heisenberg-relation}
In the category $\cH'(q)$, we have
\[
  Q_{-+} \cong Q_{+-} \oplus \mathbf{1}.
\]
\end{lemma}

\begin{proof}
Consider the following morphisms of $\cH'(q)$.
\begin{equation} \label{eq:basic-heisenberg}
\begin{tikzpicture}[baseline=0pt]
  \draw (0,3) node {$Q_{-+}$};
  \draw (0,-3) node {$Q_{-+}$};
  \draw (-3,0) node {$Q_{+-}$};
  \draw (3,0) node {$\mathbf{1}$};
  \draw[->] (-.5,-2.5) to (-2.5,-.5);
  \draw (-1.2,-1.2) node {$\rho_1$};
  \draw[->] (.5,-2.5) to (2.5,-.5);
  \draw (1.2,-1.2) node {$\rho_2$};
  \draw[->] (-2.5,.5) to (-.5,2.5);
  \draw (-1.2,1.2) node {$\iota_1$};
  \draw[->] (2.5,.5) to (.5,2.5);
  \draw (1.2,1.2) node {$\iota_2$};
  \draw[>=stealth] (-3.3,-2) node {$q^{-1}$};
  \draw[>=stealth] [shift={+(-2.5,-2)}] [->] (.4,-.4) to (-.4,.4);
  \draw[>=stealth] [shift={+(-2.5,-2)}] [<-] (-.4,-.4) to (.4,.4);
  \draw[>=stealth] [shift={+(-2.5,2)}] [<-] (.4,-.4) to (-.4,.4);
  \draw[>=stealth] [shift={+(-2.5,2)}] [->] (-.4,-.4) to (.4,.4);
  \draw[<-,>=stealth] (2,-2.2) arc(180:0:.5);
  \draw[->,>=stealth] (2,2.2) arc(180:360:.5);
\end{tikzpicture}
\end{equation}
It is immediate from the defining local relations of $\cH'(q)$ and isotopies that
\[
  \rho_2 \iota_1 = 0,\quad \rho_1 \iota_2 = 0,\quad \rho_1 \iota_1 = \id,\quad \rho_2 \iota_2 = \id,\quad \iota_1 \rho_1 + \iota_2 \rho_2 = \id,
\]
which proves the desired isomorphism.
\end{proof}

\begin{definition}[Grothendieck group]
  The \emph{(split) Grothendieck group} of an additive category $\mathcal{C}$ is the abelian group $K_0(\mathcal{C})$ (written additively) with generators $[X]$, $X \in \Ob \mathcal{C}$, and relations $[Z] = [X] + [Y]$ whenever $Z \cong X \oplus Y$.
\end{definition}

It follows from Lemma~\ref{lem:basic-heisenberg-relation} that in the Grothendieck group $K_0(\cH'(q))$, we have
\[
  [Q_-][Q_+] = [Q_+][Q_-] + 1,
\]
which is the Heisenberg relation.

\begin{remark}\label{rem:higher level}
 For each $r$-tuple of complex numbers $(u_1,\hdots,u_r)$, one can define a ``higher level" Heisenberg category $\cH'(q,u_1,\hdots,u_r)$, generalizing the definition of $\cH'(q)$.  The defining objects $Q_+,Q_-$ are the same in $\cH'(q,u_1,\hdots,u_r)$ as in $\cH'(q)$.  The morphisms of
$\cH'(q,u_1,\dots,u_r)$ look like the morphisms in $\cH'(q)$, with the caveat that strands are now allowed to carry a new defining dot, which satisfies a degree $r$ polynomial relation with roots $u_1,\hdots,u_r$.  (Thus any diagram containing a strand with more than $r$ dots on it can be written as a linear combination of diagrams whose strands carry fewer that $r$ dots.)  In this higher level categorification, the fundamental relationship between $Q_+$ and $Q_-$ becomes
\[
	 Q_{-+} \cong Q_{+-} \oplus \mathbf{1}^{\oplus r},
\]
which categorifies the ``higher level" Heisenberg relation
\[
	 [Q_-][Q_+] = [Q_+][Q_-] + r.
\]
As we will not have use for these higher level categories in the current paper, we have elected not to give the details of this generalization here.  We note, however, that these higher level categorifications
$\cH'(q,u_1,\hdots,u_r)$ are related to cyclotomic quotients of the degenerate affine Hecke algebra in the same way that $\cH'(q)$ is related to the Hecke algebra (which is the cyclotomic Hecke algebra when $r=1$).
\end{remark}
\subsection{Triple point moves}
\label{sec:triple-point}

We have the following equalities of triple point diagrams.
\begin{equation} \label{eq:triple1}
\begin{tikzpicture}[>=stealth,baseline=25]
  \draw (0,0) -- (2,2)[->];
  \draw (2,0) -- (0,2)[->];
  \draw (1,0) .. controls (0,1) .. (1,2)[->];
  \draw (2.5,1) node {$=$};
  \draw (3,0) -- (5,2)[->];
  \draw (5,0) -- (3,2)[->];
  \draw (4,0) .. controls (5,1) .. (4,2)[->];
\end{tikzpicture}
\end{equation}
\begin{equation} \label{eq:triple2}
\begin{tikzpicture}[>=stealth,baseline=25]
  \draw (0,0) -- (2,2)[->];
  \draw (2,0) -- (0,2)[<-];
  \draw (1,0) .. controls (0,1) .. (1,2)[->];
  \draw (2.5,1) node {$=$};
  \draw (3,0) -- (5,2)[->];
  \draw (5,0) -- (3,2)[<-];
  \draw (4,0) .. controls (5,1) .. (4,2)[->];
\end{tikzpicture}
\end{equation}
\begin{equation} \label{eq:triple3}
\begin{tikzpicture}[>=stealth,baseline=25]
  \draw (0,0) -- (2,2)[<-];
  \draw (2,0) -- (0,2)[->];
  \draw (1,0) .. controls (0,1) .. (1,2)[->];
  \draw (2.5,1) node {$=$};
  \draw (3,0) -- (5,2)[<-];
  \draw (5,0) -- (3,2)[->];
  \draw (4,0) .. controls (5,1) .. (4,2)[->];
\end{tikzpicture}
\end{equation}
\begin{equation} \label{eq:triple4}
\begin{tikzpicture}[>=stealth,baseline=25]
  \draw (0,0) -- (2,2)[->];
  \draw (2,0) -- (0,2)[->];
  \draw (1,0) .. controls (0,1) .. (1,2)[<-];
  \draw (2.5,1) node {$=$};
  \draw (3,0) -- (5,2)[->];
  \draw (5,0) -- (3,2)[->];
  \draw (4,0) .. controls (5,1) .. (4,2)[<-];
  \draw (6.5,1) node {$+\ q(q-1)$};
  \draw[->] (8,0) to (8,2);
  \draw[<-] (8.5,0) arc(180:0:.5);
  \draw[->] (8.5,2) arc(180:360:.5);
\end{tikzpicture}
\end{equation}

\begin{proof}
Equality~\eqref{eq:triple1} is one of the defining the relations.  We see~\eqref{eq:triple2} as follows.
\[
\begin{tikzpicture}[>=stealth]
  \draw[->] (0,0) to (2,2);
  \draw[<-] (2,0) to (0,2);
  \draw[->] (1,0) .. controls (0,1) .. (1,2);
  \draw (3,1) node {$= \ \ q^{-1}$};
  \draw[->] (4,0) to (6,2);
  \draw[->] (4,2) to (6,0);
  \draw[->] (5,0) .. controls (4,1) and (4.5,2.1) .. (5.3,1.3) ..
  controls (6.3,.3) and (6,1.5) .. (5,2);
  \draw (7.6,1) node {$-\ (1-q^{-1})$};
  \draw[->] (9,0) .. controls (10.5,1) and (11,1) .. (10,2);
  \draw[->] (9,2) to (11,0);
  \draw[->] (10,0) .. controls (9.5,.5) and (9,1) .. (9.5,1.5) .. controls (10,2)
  and (10.5,1) .. (11,2);
\end{tikzpicture}
\]
\[
\begin{tikzpicture}[>=stealth]
  \draw (0,0) node {};
  \draw (3,1) node {$= \ \ q^{-1}$};
  \draw[->] (4,0) to (6,2);
  \draw[->] (4,2) to (6,0);
  \draw[->] (5,0) .. controls (3.5,.5) and (4,1.7) .. (4.7,.7) ..
  controls (5.4,-.3) and (6,1) .. (5,2);
  \draw (7.6,1) node {$-\ (1-q^{-1})$};
  \draw[->] (9,0) .. controls (10.5,1) and (11,1) .. (10,2);
  \draw[->] (9,2) to (11,0);
  \draw[->] (10,0) .. controls (9.5,.5) and (9,1) .. (9.5,1.5) .. controls (10,2)
  and (10.5,1) .. (11,2);
\end{tikzpicture}
\]
\[
\begin{tikzpicture}[>=stealth]
  \draw (0,1) node {$=$};
  \draw[->] (1,0) to (3,2);
  \draw[->] (1,2) to (3,0);
  \draw[->] (2,0) .. controls (3,1) .. (2,2);
  \draw (4.6,1) node {$+\ (1-q^{-1})$};
  \draw[->] (6,0) .. controls (7.5,1) and (8,1) .. (7,2);
  \draw[->] (6,2) to (8,0);
  \draw[->] (7,0) .. controls (6.5,.5) and (6,1) .. (6.5,1.5) .. controls (7,2)
  and (7.5,1) .. (8,2);
  \draw (9.6,1) node {$-\ (1-q^{-1})$};
  \draw[->] (11,0) .. controls (12.5,1) and (13,1) .. (12,2);
  \draw[->] (11,2) to (13,0);
  \draw[->] (12,0) .. controls (11.5,.5) and (11,1) .. (11.5,1.5) .. controls (12,2)
  and (12.5,1) .. (13,2);
\end{tikzpicture}
\]
\[
\begin{tikzpicture}[>=stealth]
  \draw (0,1) node {$=$};
  \draw[->] (1,0) to (3,2);
  \draw[->] (1,2) to (3,0);
  \draw[->] (2,0) .. controls (3,1) .. (2,2);
\end{tikzpicture}
\]
In the above, the first equality follows from~\eqref{eq:local-relation-basic-Hecke} applied to the two strands at the top right.  The second equality follows from~\eqref{eq:local-relation-braid} applied to the middle three crossings (of the first diagram) viewed sideways.  The third equality follows from~\eqref{eq:local-relation-basic-Hecke} applied to the bottom left two crossings in the first diagram of the previous line.

The proof of~\eqref{eq:triple3} is analogous and will be omitted.
Finally,~\eqref{eq:triple4} is proven as follows.
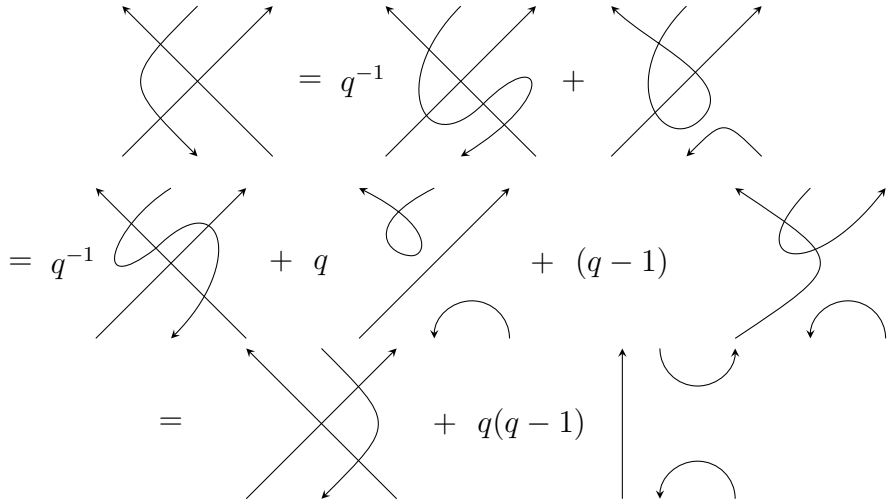
\[
\begin{tikzpicture}[>=stealth]
  \draw[->] (0,0) to (2,2);
  \draw[->] (2,0) to (0,2);
  \draw[->] (1,2) .. controls (0,1) .. (1,0);
  \draw (2.5,1) node {$=$};
  \draw (3.2,1) node {$q^{-1}$};
  \draw[->] (3.5,0) to (5.5,2);
  \draw[->] (5.5,0) to (3.5,2);
  \draw[->] (4.5,2) .. controls (3.5,1) and (4,-.1) .. (4.8,.7) ..
  controls (5.6,1.5) and (5.8,.8) .. (4.5,0);
  \draw (6,1) node {$+$};
  \draw[->] (6.5,0) to (8.5,2);
  \draw[->] (8.5,0) .. controls (8,.5) .. (7.5,0);
  \draw[->] (7.5,2) .. controls (6.5,1) and (7.2,0) .. (7.7,.5) ..
  controls (8.2,1) and (7,1.5) .. (6.5,2);
\end{tikzpicture}
\]
\[
\begin{tikzpicture}[>=stealth]
  \draw (2.5,1) node {$=$};
  \draw (3.2,1) node {$q^{-1}$};
  \draw[->] (3.5,0) to (5.5,2);
  \draw[->] (5.5,0) to (3.5,2);
  \draw[->] (4.5,2) .. controls (3.7,1.5) and (3.4,.5) .. (4.3,1.2) ..
  controls (5.2,2.1) and (5.5,1) .. (4.5,0);
  \draw (6.2,1) node {$+ \ \ q$};
  \draw[->] (7,0) to (9,2);
  \draw[->] (9,0) arc(0:180:.5);
  \draw[->] (8,2) .. controls (7,1.5) and (7.4,1.1) .. (7.7,1.1)
  .. controls (7.9,1.1) and (8,1.5) .. (7,2);
  \draw (10.2,1) node {$+\ \ (q-1)$};
  \draw[->] (12,0) .. controls (13.5,1) .. (12,2);
  \draw[->] (13,2) .. controls (12,1) and (13,.7) .. (14,2);
  \draw[->] (14,0) arc(0:180:.5);
\end{tikzpicture}
\]
\[
\begin{tikzpicture}[>=stealth]
  \draw (0,1) node {$=$};
  \draw[->] (1,0) to (3,2);
  \draw[->] (3,0) to (1,2);
  \draw[->] (2,2) .. controls (3,1) .. (2,0);
  \draw (4.5,1) node {$+ \ \ q(q-1)$};
  \draw[->] (6,0) to (6,2);
  \draw[<-] (6.5,0) arc(180:0:.5);
  \draw[->] (6.5,2) arc(180:360:.5);
\end{tikzpicture}
\]
Here the first equality follows from~\eqref{eq:local-relation-basic-Hecke} applied to the two strands at the bottom right.  The second equality follows from applying~\eqref{eq:triple3} (viewed sideways) to the three middle three crossings in the first diagram and~\eqref{eq:local-relation-basic-Hecke} to the double crossing in the middle of the second diagram.  The third equality is obtained by applying the second relation of~\eqref{eq:local-relation-up-down-double-crossing} to the first and third diagrams and the second relation in~\eqref{eq:cc-circle-and-left-curl} to the second diagram.
\end{proof}

\subsection{Right curl moves}

We will use a dot to denote a right curl and a dot labeled $d$ to
denote $d$ right curls.
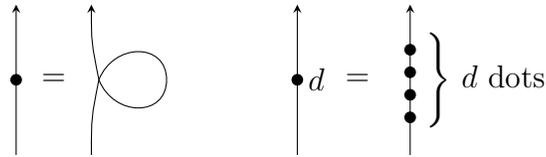
\[
\begin{tikzpicture}[>=stealth]
  \draw[->] (0,0) to (0,2);
  \filldraw (0,1) circle (2pt);
  \draw (0.5,1) node {$=$};
  \draw (2,1) .. controls (2,1.5) and (1.3,1.5) .. (1.1,1);
  \draw (2,1) .. controls (2,.5) and (1.3,.5) .. (1.1,1);
  \draw (1,0) .. controls (1,.5) .. (1.1,1);
  \draw (1.1,1) .. controls (1,1.5) .. (1,2) [->];
\end{tikzpicture}
\qquad \qquad
\begin{tikzpicture}[>=stealth]
  \draw[->] (0,0) to (0,2);
  \filldraw (0,1) circle (2pt) node[anchor=west] {$d$};
  \draw (.8,1) node {$=$};
  \draw[->] (1.5,0) to (1.5,2);
  \filldraw (1.5,.5) circle (2pt);
  \filldraw (1.5,.8) circle (2pt);
  \filldraw (1.5,1.1) circle (2pt);
  \filldraw (1.5,1.4) circle (2pt);
  \draw (2.5,1) node {$\Bigg\}\ d$ dots};
\end{tikzpicture}
\]
We will see in Section~\ref{sec:JM-elements} that dots correspond to Jucys-Murphy elements in Hecke algebras.

\begin{lemma} \label{lem:dot-crossing-moves}
We have the following equalities of diagrams.
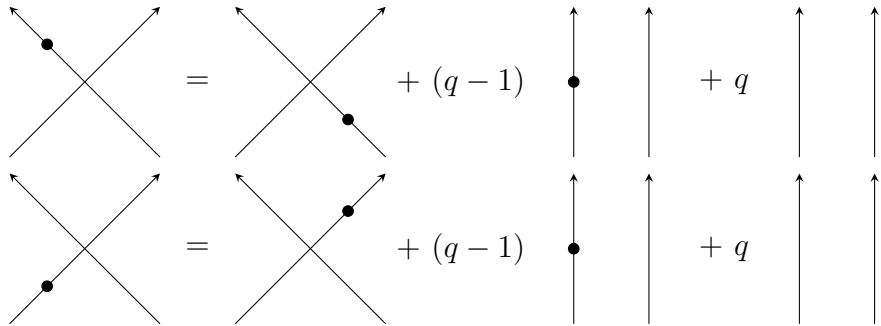
\[
\begin{tikzpicture}[>=stealth]
  \draw[->] (0,0) to (2,2);
  \draw[->] (2,0) to (0,2);
  \filldraw (0.5,1.5) circle (2pt);
  \draw (2.5,1) node {$=$};
  \draw[->] (3,0) to (5,2);
  \draw[->] (5,0) to (3,2);
  \filldraw (4.5,.5) circle (2pt);
  \draw (6,1) node {$+\ (q-1)$};
  \draw[->] (7.5,0) to (7.5,2);
  \draw[->] (8.5,0) to (8.5,2);
  \filldraw (7.5,1) circle (2pt);
  \draw (9.5,1) node {$+\ q$};
  \draw[->] (10.5,0) to (10.5,2);
  \draw[->] (11.5,0) to (11.5,2);
\end{tikzpicture}
\]
\[
\begin{tikzpicture}[>=stealth]
  \draw[->] (0,0) to (2,2);
  \draw[->] (2,0) to (0,2);
  \filldraw (0.5,0.5) circle (2pt);
  \draw (2.5,1) node {$=$};
  \draw[->] (3,0) to (5,2);
  \draw[->] (5,0) to (3,2);
  \filldraw (4.5,1.5) circle (2pt);
  \draw (6,1) node {$+\ (q-1)$};
  \draw[->] (7.5,0) to (7.5,2);
  \draw[->] (8.5,0) to (8.5,2);
  \filldraw (7.5,1) circle (2pt);
  \draw (9.5,1) node {$+\ q$};
  \draw[->] (10.5,0) to (10.5,2);
  \draw[->] (11.5,0) to (11.5,2);
\end{tikzpicture}
\]
\end{lemma}

\begin{proof}
We prove the first equality.  The second is analogous.
\[
\begin{tikzpicture}[>=stealth]
  \draw[->] (0,0) to (2,2);
  \draw[->] (2,0) to (0,2);
  \filldraw (0.5,1.5) circle (2pt);
  \draw (2.5,1) node {$=$};
  \draw[->] (3,0) to (5,2);
  \draw[->] (5,0) to (4,1) .. controls (3.5,1.5) and (3.7,1.8) ..
  (4,1.5) .. controls (4.3,1.2) and (3.7,1.2) .. (3,2);
  \draw (6,1) node {$= \ q^{-1}$};
  \draw[->] (9,0) .. controls (8,.5) and (7.2,1.5) .. (8.2,1.5) ..
  controls (9.2,1.5) and (8.5,-.25) .. (7,2);
  \draw[->] (7,0) .. controls (8,.5) and (9,.3) .. (8.2,1.3) ..
  controls (7.6,1.8) and (8.5,1.75) .. (9,2);
  \draw (9.5,1) node {$+$};
  \draw[->] (10,0) .. controls (11,1) .. (10,2);
  \draw[->] (11,0) .. controls (10,1) .. (11,2);
\end{tikzpicture}
\]
\[
\begin{tikzpicture}[>=stealth]
  \draw (6,1) node {$= \ q^{-1}$};
  \draw[->] (9,0) .. controls (8,.5) and (7.2,1.5) .. (8.2,1.5) ..
  controls (9.2,1.5) and (8.5,-.25) .. (7,2);
  \draw[->] (7,0) .. controls (7,1) and (7.5,1.6) .. (8.2,1.3) ..
  controls (8.9,1) and (9,1.5) .. (9,2);
  \draw (9.5,1) node {$+$};
  \draw[->] (10,0) .. controls (11,1) .. (10,2);
  \draw[->] (11,0) .. controls (10,1) .. (11,2);
  \draw (11.5,1) node {$=$};
  \draw[->] (14,0) .. controls (13.5,.5) and (13.7,.8) ..
  (14,.5) .. controls (14.3,.2) and (13.7,.2) .. (13,1) to (12,2);
  \draw[->] (12,0) to (14,2);
  \draw (15.5,1) node {$+ \ \ (1-q^{-1})$};
  \draw[->] (17,0) .. controls (17.5,1) and (18.5,2) .. (18.5,1) .. controls (18.5,0)
  and (17.5,1) .. (17,2);
  \draw[->] (19,0) .. controls (18,1) .. (19,2);
  \draw (19.5,1) node {$+$};
  \draw[->] (20,0) .. controls (21,1) .. (20,2);
  \draw[->] (21,0) .. controls (20,1) .. (21,2);
\end{tikzpicture}
\]
\[
\begin{tikzpicture}[>=stealth]
  \draw (0.5,1) node {$=$};
  \draw[->] (1,0) to (3,2);
  \draw[->] (3,0) to (1,2);
  \filldraw (2.5,.5) circle (2pt);
  \draw (4.5,1) node {$+ \ (q-1)$};
  \draw[->] (6,0) .. controls (6,1) and (6.5,1.5) .. (6.5,1) ..
  controls (6.5,.5) and (6,1) .. (6,2);
  \draw[->] (7,0) to (7,2);
  \draw (8.5,1) node {$-\ \ (q-1)$};
  \draw[->] (10,0) to (11,2);
  \draw[->] (11,0) to (10,2);
  \draw (11.5,1) node {$+$};
  \draw[->] (12,0) .. controls (13,1) .. (12,2);
  \draw[->] (13,0) .. controls (12,1) .. (13,2);
\end{tikzpicture}
\]
\[
\begin{tikzpicture}[>=stealth]
  \draw (0.5,1) node {$=$};
  \draw[->] (1,0) to (3,2);
  \draw[->] (3,0) to (1,2);
  \filldraw (2.5,.5) circle (2pt);
  \draw (4.5,1) node {$+ \ (q-1)$};
  \draw[->] (6,0) to (6,2);
  \filldraw (6,1) circle (2pt);
  \draw[->] (7,0) to (7,2);
  \draw (8.5,1) node {$+\ \ q$};
  \draw[->] (9.5,0) to (9.5,2);
  \draw[->] (10.5,0) to (10.5,2);
\end{tikzpicture}
\]
\end{proof}

It follows by induction that we have the following equalities.
\[
\begin{tikzpicture}[>=stealth]
  \draw[->] (0,0) to (2,2);
  \draw[->] (2,0) to (0,2);
  \filldraw (.5,1.5) circle (2pt) node[anchor=west] {$d$};
  \draw (2.5,1) node {$=$};
  \draw[->] (3,0) to (5,2);
  \draw[->] (5,0) to (3,2);
  \filldraw (4.5,.5) circle (2pt) node[anchor=west] {$d$};
  \draw (7,1) node {$+(q-1) \sum_{a=1}^{d}$};
  \draw[->] (8.5,0) to (8.5,2);
  \draw[->] (9.5,0) to (9.5,2);
  \filldraw (8.5,1) circle (2pt) node[anchor=west] {$a$};
  \filldraw (9.5,1) circle (2pt) node[anchor=west] {$(d-a)$};
  \draw (12,1) node {$+q \sum_{b=0}^{d-1}$};
  \draw[->] (13,0) to (13,2);
  \draw[->] (14,0) to (14,2);
  \filldraw (13,1) circle (2pt) node[anchor=west] {$b$};
  \filldraw (14,1) circle (2pt) node[anchor=west]  {$(d-1-b)$};
\end{tikzpicture}
\]
\[
\begin{tikzpicture}[>=stealth]
  \draw[->] (0,0) to (2,2);
  \draw[->] (2,0) to (0,2);
  \filldraw (0.5,0.5) circle (2pt) node[anchor=west] {$d$};
  \draw (2.5,1) node {$=$};
  \draw[->] (3,0) to (5,2);
  \draw[->] (5,0) to (3,2);
  \filldraw (4.5,1.5) circle (2pt) node[anchor=west] {$d$};
  \draw (7,1) node {$+(q-1) \sum_{a=1}^{d}$};
  \draw[->] (8.5,0) to (8.5,2);
  \draw[->] (9.5,0) to (9.5,2);
  \filldraw (8.5,1) circle (2pt) node[anchor=west] {$a$};
  \filldraw (9.5,1) circle (2pt) node[anchor=west] {$(d-a)$};
  \draw (12,1) node {$+q \sum_{b=0}^{d-1}$};
  \draw[->] (13,0) to (13,2);
  \draw[->] (14,0) to (14,2);
  \filldraw (13,1) circle (2pt) node[anchor=west] {$b$};
  \filldraw (14,1) circle (2pt) node[anchor=west]  {$(d-1-b)$};
\end{tikzpicture}
\]

Define $\tilde{c}_d$ to be a counterclockwise oriented circle with
$d$ right curls on it, and $c_d$ to be a clockwise oriented circle
with $d$ right curls.
\[
\begin{tikzpicture}[>=stealth]
  \draw (-2,0) node {$\tilde c_d = $};
  \draw (0,0) arc (0:360:.5) [->];
  \filldraw (-1,0) circle (2pt) node[anchor=east] {$d$};
\end{tikzpicture}
\qquad \qquad
\begin{tikzpicture}[>=stealth]
  \draw (-2,0) node {$c_d=$};
  \draw (0,0) arc (0:360:.5) [<-];
  \filldraw (-1,0) circle (2pt) node[anchor=east] {$d$};
\end{tikzpicture}
\]

\begin{proposition} \label{prop:circle-conversion}
For $d \ge 0$ we have
\[
  \tilde{c}_{d+1} = (q-1)\sum_{a=1}^{d} \tilde{c}_a c_{d-a} +
  q\sum_{b=0}^{d-1} \tilde{c}_b c_{d-1-b}.
\]
Note, in particular, that this implies $\tilde c_1=0$.
\end{proposition}

\begin{proof}
\[
\begin{tikzpicture}[>=stealth]
  \draw (-.5,1) node {$\tilde c_{d+1}=$};
  \draw[-<] (1,1) .. controls (1,2) and (1.5,2) .. (2,1) .. controls (2.5,0) and (3,0) ..
  (3,1);
  \draw (3,1) .. controls (3,2) and (2.5,2) .. (2,1) .. controls (1.5,0) and (1,0) .. (1,1);
  \filldraw (1,1) circle (2pt) node[anchor=east] {$d$};
\end{tikzpicture}
\]
\[
\begin{tikzpicture}[>=stealth]
  \draw (3.5,1) node {$=$};
  \draw[-<] (4,1) .. controls (4,2) and (4.5,2) .. (5,1) .. controls (5.5,0) and (6,0) ..
  (6,1);
  \draw (6,1) .. controls (6,2) and (5.5,2) .. (5,1) .. controls (4.5,0) and (4,0) .. (4,1);
  \filldraw (5.3,1.5) circle (2pt) node[anchor=south] {$d$};
  \draw (8,1) node {$+ \ (q-1) \sum_{a=1}^d$};
  \draw[->] (9.7,1) arc(-180:180:.5);
  \filldraw (10.2,1.5) circle (2pt) node[anchor=south] {$a$};
  \draw[-<] (11,1) arc(-180:180:.5);
  \filldraw (11.5,1.5) circle (2pt) node[anchor=south] {$d-a$};
  \draw (13,1) node {$+ \ q \sum_{b=0}^{d-1}$};
  \draw[->] (14.2,1) arc(-180:180:.5);
  \filldraw (14.7,1.5) circle (2pt) node[anchor=south] {$b$};
  \draw[-<] (15.5,1) arc(-180:180:.5);
  \filldraw (16,1.5) circle (2pt) node[anchor=south] {$d-1-b$};
\end{tikzpicture}
\]
The result then follows from the fact that the left curl is equal to
zero.
\end{proof}

\subsection{Bubble moves}

We can move clockwise circles past lines using so called ``bubble moves''.

\begin{lemma}
We have the following equality.
\[
\begin{tikzpicture}[>=stealth]
  \draw (0,0) arc (0:360:0.5) [-<];
  \draw[->] (0.5,-1) to (0.5,1);
  \draw (1,0) node {$=$};
  \draw[->] (1.5,-1) to (1.5,1);
  \draw (3,0) arc (0:360:.5) [-<];
  \draw (4.5,0) node {$+\ (1-q^{-1})$};
  \draw[->] (6,-1) to (6,1);
  \filldraw (6,0) circle (2pt);
  \draw (7,0) node {$+$};
  \draw[->] (8,-1) to (8,1);
\end{tikzpicture}
\]
\end{lemma}

\begin{proof}
We have
\[
\begin{tikzpicture}[>=stealth]
  \draw (0,0) arc (0:360:0.5) [-<];
  \draw[->] (0.5,-1) to (0.5,1);
  \draw (1.5,0) node {$=\ q^{-1}$};
  \draw (3.3,0) arc (0:360:0.5) [-<];
  \draw[->] (3,-1) to (3,1);
  \draw (3.8,0) node {$+$};
  \draw[->] (4.8,-1) .. controls (4.8,.5) and (4.5,-.5) .. (4.3,-.5) .. controls (4.1,-.5) and (4.1,.5) .. (4.3,.5) .. controls (4.5,.5) and (4.8,-.5) .. (4.8,1);
  \draw (5.3,0) node {$=$};
  \draw[->] (5.8,-1) to (5.8,1);
  \draw (7.3,0) arc(0:360:0.5) [-<];
  \draw (8.6,0) node {$+\ (1-q^{-1})$};
  \draw[shift={+(.7,-1)}] (10.3,1) .. controls (10.3,1.5) and (9.6,1.5) .. (9.4,1);
  \draw[shift={+(.7,-1)}] (10.3,1) .. controls (10.3,.5) and (9.6,.5) .. (9.4,1);
  \draw[shift={+(.7,-1)}] (9.3,0) .. controls (9.3,.5) .. (9.4,1);
  \draw[shift={+(.7,-1)}] (9.4,1) .. controls (9.3,1.5) .. (9.3,2) [->];
  \draw (11.5,0) node {$+$};
  \draw[->] (12,-1) to (12,1);
\end{tikzpicture}
\]
where in the first equality we used~\eqref{eq:local-relation-up-down-double-crossing} and in the second we used~\eqref{eq:local-relation-basic-Hecke}.
\end{proof}

More generally, we can move clockwise circles with dots past lines
as follows.
\begin{gather} \label{eq:dotted-bubble-move}
\begin{tikzpicture}[>=stealth,baseline=0pt]
  \draw (-2,0) arc (0:360:0.5) [-<];
  \draw[->] (-1.5,-1) to (-1.5,1);
  \draw (-1,0) node {$=$};
  \filldraw (-2.5,0.5) circle (2pt) node[anchor=south] {$d$};
  \draw[->] (-.5,-1) to (-.5,1);
  \draw (1,0) arc (0:360:.5) [-<];
  \filldraw (.5,.5) circle (2pt) node[anchor=south] {$d$};
  \draw (3,0) node {$+\ (d+1)(1-q^{-1})$};
  \draw[->] (5,-1) to (5,1);
  \filldraw (5,0) circle (2pt) node[anchor=west] {$(d+1)$};
  \draw (7.5,0) node {$+\ (d+1)$};
  \draw[->] (9,-1) to (9,1);
  \filldraw (9,0) circle (2pt) node[anchor=west] {$d$};
\end{tikzpicture} \\ \nonumber
\begin{tikzpicture}[>=stealth]
  \draw (0,0) node {$-(q-1)(1-q^{-1}) \sum_{a=1}^d a$};
  \draw[->] (3,-1) to (3,1);
  \filldraw (3,0) circle (2pt) node[anchor=west] {$a$};
  \draw (4.5,0) arc (0:360:0.5) [-<];
  \filldraw (4,.5) circle (2pt) node[anchor=south] {$(d-a)$};
\end{tikzpicture}
\begin{tikzpicture}[>=stealth]
  \draw (0,0) node {$-\ (q-1) \sum_{a=0}^{d-1} (2a+1)$};
  \draw[->] (3,-1) to (3,1);
  \filldraw (3,0) circle (2pt) node[anchor=west] {$a$};
  \draw (3.5,0) arc (-180:180:0.5) [-<];
  \filldraw (4.5,0) circle (2pt) node[anchor=west] {$(d-1-a)$};
\end{tikzpicture} \\ \nonumber
\begin{tikzpicture}[>=stealth]
  \draw (0,0) node {$-\ q \sum_{a=0}^{d-2} (a+1)$};
  \draw[->] (2,-1) to (2,1);
  \filldraw (2,0) circle (2pt) node[anchor=west] {$a$};
  \draw[-<] (2.5,0) arc (-180:180:0.5);
  \filldraw (3.5,0) circle (2pt) node[anchor=west] {$(d-2-a)$};
\end{tikzpicture}
\end{gather}

\subsection{Endomorphism algebras and the affine Hecke algebra}
\label{sec:endo-affine-Hecke}

\begin{theorem}\label{thm:endidentity}
The natural map
\[
  \psi_0: \kk[q,q^{-1}][c_0,c_1,c_2,\dots] \to \End_{\cH'(q)}(\mathbf{1})
\]
is an isomorphism.
\end{theorem}

\begin{proof}
Every element of $\End_{\cH'(q)}(\mathbf{1})$ is a linear combination
of closed diagrams.  Using the local relations, any closed diagram
can be expressed as a linear combination of crossingless diagrams
consisting of nested dotted circles.  The bubble moves imply that
nested dotted circles can be written as linear combinations of
dotted circles with no nesting.  Finally, counterclockwise circles
can be expressed as linear combinations of clockwise ones by
Proposition~\ref{prop:circle-conversion}.  Therefore, any element of
$\End_{\cH'(q)}(\mathbf{1})$ can be written as a linear combination
of products of dotted clockwise circles.  It follows that $\psi_0$
is surjective.

Assume $q$ is an indeterminate and $\kk = \Z$.  We can view any field $\kk'$ as an $\Z[q,q^{-1}]$-module via the map that sends $q$ to $1$.  Consider the composition
\[
  \kk'[c_0,c_1,\dots] \cong \Z[q,q^{-1}][c_0,c_1,\dots] \otimes_{\Z[q,q^{-1}]} \kk'
  \xrightarrow{\psi_0 \otimes \id} \End_{\cH'(q)} (\mathbf{1})
  \otimes_{\Z[q,q^{-1}]} \kk' = \End_{\cH'} (\mathbf{1}),
\]
where $\cH'$ is the category defined in~\cite{Kho10} (for the field $\kk'$).  This composition is precisely the map $\psi_0$ of~\cite{Kho10}, which is injective by~\cite[Proposition~3]{Kho10}.
It follows that our map $\psi_0$ is also injective when we work over the ring $\kk=\Z$ and $q$ is an indeterminate.  Since, for an arbitrary ring $\kk$, $\kk[q,q^{-1}]$ (where $q$ is either an indeterminate or an invertible element of $\kk$) is a $\Z[q,q^{-1}]$-module in the obvious way, we can tensor with $\kk[q,q^{-1}]$ and see that $\psi_0$ is an isomorphism in the general case.
\end{proof}

Let $H_n^\aff$ denote the affine Hecke algebra of type A,
\[
  H_n^\aff = H_n\otimes_{\kk[q,q^{-1}]} \kk[q,q^{-1}][x_1^{\pm 1},\hdots x_n^{\pm 1}].
\]
The Hecke algebra $H_n$ and $\kk[x_1^{\pm 1},\hdots x_n^{\pm 1}]$ are subalgebras of $H_n^\aff$, and the defining relations between these subalgebras are
\[
  t_ix_k = x_kt_i, \text{ for } i \neq k,k+1,\quad \text{and} \quad t_ix_it_i = qx_{i+1}.
\]

\begin{lemma}\label{lem:subalg}
Assume $q \in \kk^\times$, $q \ne 1$.  If, for $i=1,\hdots n$, we define new elements of the affine Hecke algebra $y_i$ by
\[
  y_i = (q-1)x_i - \frac{q}{q-1},
\]
then
\begin{align}
  y_it_k &= t_ky_i,\quad i\neq k, k+1, \label{eq:affine-Hecke1} \\
  t_i y_{i+1} &= y_it_i + (q-1)y_{i+1} + q, \label{eq:affine-Hecke2} \\
  y_{i+1}t_i &= t_iy_i + (q-1)y_{i+1} + q. \label{eq:affine-Hecke3}
\end{align}
\end{lemma}

\begin{proof}
The first relation is obvious.  For the second relation above, we have
\[
  t_iy_{i+1} = (q-1)t_ix_{i+1} - \frac{q}{q-1}t_i,
\]
while
\[
  y_it_i = (q-1)x_it_i -  \frac{q}{q-1}t_i.
\]
Subtracting, we get
\begin{equation} \label{eq:sjfky827}
  t_iy_{i+1}-y_it_i=(q-1)(t_ix_{i+1}-x_it_i).
\end{equation}
Since $t_i^{-1} = q^{-1}t_i + (q^{-1}-1)$, after multiplying both sides of the relation $t_ix_it_i = qx_{i+1}$ on the left by $t_i^{-1}$ we get
\[
  x_it_i = t_ix_{i+1} + (1-q)x_{i+1},
\]
or
\[
  t_ix_{i+1} - x_it_i = (q-1)x_{i+1}.
\]
Substituting into \eqref{eq:sjfky827}, we obtain
\[
  t_iy_{i+1}-y_it_i = (q-1)^2x_{i+1}.
\]
Since $x_{i+1}=(q-1)^{-1}y_{i+1} + \frac{q}{(q-1)^2}$, we get
\[
   t_iy_{i+1}-y_it_i= (q-1)y_{i+1} + q,
\]
as desired.  The last relation is similar.
\end{proof}

Let $H_n^+$ be the $\kk[q,q^{-1}]$-algebra with generators $t_i, y_i$, $1 \le i \le n$, and defining relations \eqref{eq:affine-Hecke1}--\eqref{eq:affine-Hecke3}.  By Lemma~\ref{lem:subalg}, if $q \in \kk^\times$, $q \ne 1$, we have
\[
  H_n^+ \cong H_n \otimes_\kk \kk[x_1,\dots,x_n] \subseteq H_n^\aff.
\]
It follows from Lemma~\ref{lem:dot-crossing-moves} that there is a natural morphism
\[
  \phi_n:H_n^+\to \End_{\cH'(q)}(Q_{+^n})
\]
taking $t_k$ to the crossing of the $k$ and $(k+1)$-st strands and
taking $y_k$ to a right curl (or dot) on the $k$-th strand.  More
generally, there is a natural morphism
\[
  \psi_n = \phi_n\otimes \psi_0 : H_n^+ \otimes_{\kk[q,q^{-1}]} \kk[q,q^{-1}][c_0,c_1,\dots]
  \to \End_{\cH'(q)}(Q_{+^n}),
\]
where the dotted clockwise circles corresponding to elements of $\kk[c_0,c_1,\dots]$ are placed to the right of the diagrams corresponding to elements of $H_n^+$.

\begin{theorem}\label{thm:endQ}
The morphism $\psi_n$ is an isomorphism of algebras.
\end{theorem}

\begin{proof}
Any diagram representing an element of $\End_{\cH'(q)}(Q_{+^n})$ can
be inductively simplified to a linear combination of standard
diagrams consisting of a element of $y \in H_n$ (written as a strand
diagram), some number (possibly zero) of dots on each strand above
the crossings, and a product of dotted clockwise circles to the
right:
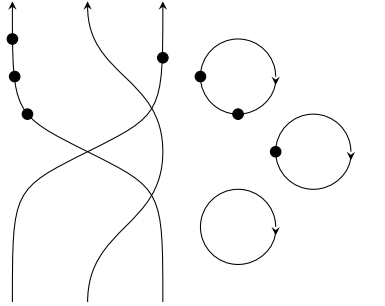
\[
\begin{tikzpicture}[>=stealth,baseline=50pt]
  \draw[->] (0,0) .. controls (0,1.5) .. (1,2) .. controls (2,2.5) .. (2,4);
  \draw[->] (2,0) .. controls (2,1.5) .. (1,2) .. controls (0,2.5) .. (0,4);
  \draw[->] (1,0) .. controls (1,1) and (2,1) .. (2,2) .. controls (2,3) and (1,3) .. (1,4);
  \filldraw (0.03,3) circle (2pt);
  \filldraw (0,3.5) circle (2pt);
  \filldraw (2,3.25) circle (2pt);
  \filldraw (0.2,2.5) circle (2pt);
  \draw[-<] (3.5,1) arc(0:360:.5);
  \draw[-<] (3.5,3) arc(0:360:.5);
  \filldraw (2.5,3) circle (2pt);
  \filldraw (3,2.5) circle (2pt);
  \draw[-<] (4.5,2) arc(0:360:.5);
  \filldraw (3.5,2) circle (2pt);
\end{tikzpicture} \quad .
\]
The surjectivity of $\psi_n$ then follows immediately from that of $\psi_0$.

Assume $q$ is an indeterminate and $\kk=\Z$.  We can view any field $\kk'$ as an $\Z[q,q^{-1}]$-module via the map that sends $q$ to $1$.  Consider the composition
\begin{gather*}
   H_n^+ \otimes_{\Z[q,q^{-1}]} \kk'[c_0,c_1,\dots]\cong \left( H_n^+ \otimes_{\Z[q,q^{-1}]} \Z[q,q^{-1}][c_0,c_1,\dots] \right) \otimes_{\Z[q,q^{-1}]} \kk' \\
   \xrightarrow{\psi_n \otimes \id} \End_{\cH'(q)} (Q_{+^n}) \otimes_{\Z[q,q^{-1}]} \kk' = \End_{\cH'} (Q_{+^n}),
\end{gather*}
where $\cH'$ is the category defined in~\cite{Kho10} (for the field $\kk'$).  This composition is precisely the map $\psi_n$ of~\cite{Kho10}, which is injective by~\cite[Proposition~4]{Kho10}.
It follows that our map $\psi_n$ is also injective when we work over the ring $\kk=\Z$ and $q$ is an indeterminate.  Since, for an arbitrary ring $\kk$, $\kk[q,q^{-1}]$ (where $q$ is either an indeterminate or an invertible element of $\kk$) is a $\Z[q,q^{-1}]$-module in the obvious way, we can tensor with $\kk[q,q^{-1}]$ and see that $\psi_n$ is an isomorphism in the general case.
\end{proof}

Having explicitly described $\End_{\cH'(q)}(Q_{+^n})$, we now turn to the more general problem of giving an explicit basis for $\Hom_{\cH'(q)}(Q_\epsilon, Q_{\epsilon'})$ for any sequences $\epsilon, \epsilon'$.  Let $k$ denote the total number of $+$s in $\epsilon$ and $-$s in $\epsilon'$.  We clearly have $\Hom_{\cH'(q)}(Q_\epsilon, Q_{\epsilon'})=0$ if the total number of $-$s in $\epsilon$ and $+$s in $\epsilon'$ is not also equal to $k$.  Thus, we assume from now on that $k$ is also the total number of $-$s in $\epsilon$ and $+$s in $\epsilon'$.

\begin{definition}
For two sign sequences $\epsilon, \epsilon'$, let $B(\epsilon, \epsilon')$ be the set of planar diagrams obtained in the following manner:
\begin{itemize}
  \item The sequences $\epsilon$ and $\epsilon'$ are written at the bottom and top (respectively) of the plane strip $\R \times [0,1]$.

  \item The elements of $\epsilon$ and $\epsilon'$ are matched by oriented segments embedded in the strip in such a way that their orientations match the signs (that is, they start at either a $+$ of $\epsilon$ or a $-$ of $\epsilon'$, and end at either a $-$ of $\epsilon$ or a $+$ of $\epsilon'$), each two segments intersect at most once, and no self-intersections or triple intersections are allowed.

  \item Any number of dots may be placed on each interval near its out endpoint (i.e. between its out endpoint and any intersections with other intervals).

  \item In the rightmost region of the diagram, a finite number of clockwise disjoint nonnested circles with any number of dots may be drawn.
\end{itemize}
The set of diagrams $B(\epsilon,\epsilon')$ is parameterized by $k!$ possible matchings of the $2k$ oriented endpoints, a sequence of $k$ nonnegative integers determining the number of dots on each interval, and by a finite sequence of nonnegative integers determining the number of clockwise circles with various numbers of dots.
\end{definition}

An example of an element of $B(--+-+,+-+-+--)$ is drawn below.
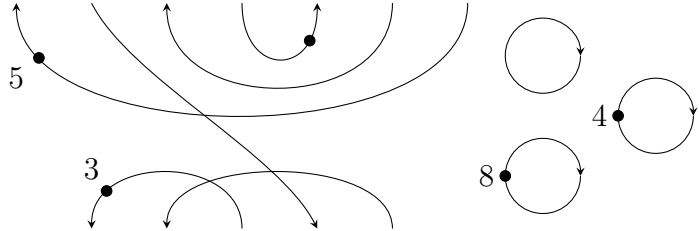
\[
\begin{tikzpicture}[>=stealth]
  \draw[<-] (0,0) .. controls (0,1) and (2,1) .. (2,0);
  \draw[<-] (1,0) .. controls (1,1) and (4,1) .. (4,0);
  \draw[->] (0,3) .. controls (.5,2) and (2.5,1) .. (3,0);
  \draw[<-] (-1,3) .. controls (-1,1) and (5,1) .. (5,3);
  \draw[<-] (1,3) .. controls (1,1.5) and (4,1.5) .. (4,3);
  \draw[->] (2,3) .. controls (2,2) and (3,2) .. (3,3);
  \filldraw (-.7,2.27) circle (2pt);
  \draw (-1,2) node {$5$};
  \filldraw (.2,.5) circle (2pt);
  \draw (0,.8) node {$3$};
  \filldraw (2.9,2.5) circle (2pt);
  \draw[<-] (6.5,2.3) arc(0:360:.5);
  \draw[<-] (6.5,.7) arc(0:360:.5);
  \filldraw (5.5,.7) circle (2pt) node[anchor=east] {$8$};
  \draw[<-] (8,1.5) arc(0:360:.5);
  \filldraw (7,1.5) circle (2pt) node[anchor=east] {$4$};
\end{tikzpicture}
\]

\begin{proposition} \label{prop:hom-space-basis}
For any sign sequences $\epsilon, \epsilon'$, the set $B(\epsilon,\epsilon')$ is a basis of the $\kk[q,q^{-1}]$-module $\Hom_{\cH'(q)} (Q_\epsilon,Q_{\epsilon'})$.
\end{proposition}

The proof, which we include for the sake of completeness, is almost identical to that of \cite[Proposition~5]{Kho10}.

\begin{proof}
It is straightforward to check that using the defining local relations of $\cH'(q)$, any element of $\Hom_{\cH'(q)} (Q_\epsilon, Q_{\epsilon'})$ can be reduced to a direct sum of elements of $B(\epsilon,\epsilon')$.  One uses \eqref{eq:local-relation-basic-Hecke} and \eqref{eq:local-relation-up-down-double-crossing} to remove double crossings, Lemma~\ref{lem:dot-crossing-moves} to move dots to the ends of intervals, \eqref{eq:dotted-bubble-move} to move circles to the rightmost region, etc.

It remains to show that $B(\epsilon, \epsilon')$ is linearly independent.  Moving the lower endpoints of a diagram up using cup diagrams, or moving the upper endpoints of a diagram down using cap diagrams, yields canonical isomorphisms
\[
  \Hom_{\cH'(q)} (Q_\epsilon, Q_{\epsilon'}) \cong \Hom_{\cH'(q)}(\mathbf{1}, Q_{\bar \epsilon \epsilon'}) \cong \Hom_{\cH'(q)} (\mathbf{1},Q_{\epsilon' \bar \epsilon}),
\]
where $\bar \epsilon$ is the sequence $\epsilon$ with the order and all signs reversed.  It thus suffices to show that $B(\varnothing, \epsilon)$ is linearly independent for any sequence $\epsilon$ with $k$ pluses and $k$ minuses.  We prove this by induction on $k$ and (for each $k$) by induction on the lexicographic order (where $+ < -$) among length $2k$ sequences.  Theorem~\ref{thm:endQ} implies the base cases $k=0,1$ and $\epsilon = +^k-^k$ for any $k$.  Now assume $\epsilon = \epsilon_1 - +\, \epsilon_2$ for some sequences $\epsilon_1$ and $\epsilon_2$. By the inductive hypothesis, $B(\varnothing, \epsilon_1 \epsilon_2)$ and $B(\varnothing, \epsilon_1 + -\, \epsilon_2)$ are linearly independent.  Lemma~\ref{lem:basic-heisenberg-relation} (more precisely, the two upper morphisms in \eqref{eq:basic-heisenberg}) gives a canonical isomorphism
\[
  Q_{\epsilon_1 + -\, \epsilon_2} \oplus Q_{\epsilon_1 \epsilon_2} \cong Q_{\epsilon_1 - +\, \epsilon_2}.
\]
This isomorphism maps the sets $B(\varnothing, \epsilon_1 \epsilon_2)$ and $B(\varnothing, \epsilon_1 + -\, \epsilon_2)$ to subsets $B_1$ and $B_2$ of $\Hom_{\cH'(q)}(Q_{\epsilon_1 - +\, \epsilon_2})$.  Let $B = B_1 \cup B_2$.  It is straightforward to check that $B(\varnothing, \epsilon_1 - +\, \epsilon_2)$ is linearly independent if and only if $B$ is.  Since we know $B$ is linearly independent by induction, we are done.
\end{proof}

\subsection{Symmetries} \label{sec:symmetries}

There are some obvious symmetries of the category $\cH'(q)$.  Let $\xi_2$ be the symmetry of $\cH'(q)$ given on diagrams by reflecting in the horizontal axis and reversing orientation of strands.  This is an involutive monoidal contravariant autoequivalence of $\cH'(q)$.

Let $\xi_3$ be the symmetry of $\cH'(q)$ given on diagrams by reflecting in the vertical axis and reversing orientation.  This is an involutive antimonoidal covariant autoequivalence of $\cH'(q)$.  By antimonoidal, we mean that $\xi_3 (M \otimes N) \cong \xi_3(N) \otimes \xi_3(M)$.

The functors $\xi_2$ and $\xi_3$ commute and hence define an action of $(\Z/2\Z)^2$.  When $q=1$, these symmetries reduce to ones defined in \cite{Kho10}.  There is a third symmetry, $\xi_1$, defined in \cite{Kho10}.  This also has a $q$-deformation but one needs to pass to an appropriate completion of the category $\cH'(q)$.  It follows from Proposition~\ref{prop:hom-space-basis} and Lemma~\ref{lem:dot-crossing-moves} that the morphism spaces of $\cH'(q)$ are filtered by numbers of dots. Let $\tilde \cH'(q)$ be the category whose objects are those of $\tilde \cH'(q)$, but where $\Hom_{\tilde \cH'(q)} (Q_\epsilon, Q_{\epsilon'})$ is the space of formal infinite linear combinations of elements of $\Hom_{\cH'(q)}(Q_\epsilon, Q_{\epsilon'})$ which are locally finite with respect to the filtration by number of dots; thus an element of $\Hom_{\tilde \cH'(q)} (Q_\epsilon, Q_{\epsilon'})$ is an infinite linear combination of diagrams such that for all $n\geq 0$ the number of summands with fewer than $n$ dots is finite.  Note that composition of such infinite sums is well-defined.  Then let $\xi_1$ be the endofunctor of $\tilde \cH'(q)$ defined locally by
\[
\begin{tikzpicture}[>=stealth,baseline=25pt]
  \draw[->] (-3,0) to (-2,2);
  \draw[->] (-2,0) to (-3,2);
  \draw (-1,1) node {${\mapsto} \ \ -q$};
  \draw[->] (0,0) to (1,2);
  \draw[->] (1,0) to (0,2);
  \filldraw[fill=white,draw=black] (.5,1) circle (2pt);
  \draw (2.2,1) node {$= \ \ -$};
  \draw[->] (3,0) to (4,2);
  \draw[->] (4,0) to (3,2);
  \draw (5.5,1) node {$-\ (1-q)$};
  \draw[->] (7,0) to (7,2);
  \draw[->] (8,0) to (8,2);
\end{tikzpicture}\quad ,
\]
$\xi_1$ is the identity on right-oriented caps and cups, and on left-oriented caps and cups, $\xi_1$ acts as
\[
\begin{tikzpicture}[>=stealth,baseline=6pt]
  \draw[<-] (1,0) arc(180:0:.5);
  \draw (3,.25) node {${\mapsto}$};
  \draw[<-] (4,0) arc(180:0:.5);
  \draw (7,.25) node {$-\ \ (q^{-1}-1)$};
  \draw[<-] (9,0) arc(180:0:.5);
  \filldraw (9.5,.5) circle (2pt);
\end{tikzpicture}
\]
and
\[
\begin{tikzpicture}[>=stealth,baseline=6pt]
  \draw[<-] (-.5,0.5) arc(180:360:.5);
  \draw (1.5,.25) node {${\mapsto}$};
  \draw (4,.25) node {$\sum_{n=0}^\infty \ \ (q^{-1}-1)^n$};
  \draw[<-] (6,0.5) arc(180:360:.5);
  \filldraw (6.5,0) circle (2pt) node[anchor=north] {$n$};
\end{tikzpicture}\quad .
\]
One can compute directly that the action of $\xi_1$ on left, right and downward-oriented crossings is given by
\[
\begin{tikzpicture}[>=stealth,baseline=25pt]
  \draw[->] (0,0) to (1,2);
  \draw[<-] (1,0) to (0,2);
  \draw (2,1.1) node {${\mapsto} \ \ -$};
  \draw[->] (3,0) to (4,2);
  \draw[<-] (4,0) to (3,2);
  \draw (5.5,1) node {$-\ (1-q)$};
  \draw[->] (7,2) to (7,1.75) arc(180:360:.5) to (8,2);
  \draw[->] (7,0) to (7,0.25) arc(180:0:.5) to (8,0);
\end{tikzpicture}\quad ,
\]
\smallskip
\[
\begin{tikzpicture}[>=stealth,baseline=25pt]
  \draw[<-] (0,0) to (1,2);
  \draw[<-] (1,0) to (0,2);
  \draw (2,1.1) node {${\mapsto} \ \ -$};
  \draw[<-] (3,0) to (4,2);
  \draw[<-] (4,0) to (3,2);
  \draw (5.5,1) node {$-\ (1-q)$};
  \draw[<-] (7,0) to (7,2);
  \draw[<-] (8,0) to (8,2);
\end{tikzpicture}\quad ,
\]
\smallskip
\[
  \begin{tikzpicture}[>=stealth,baseline=25pt]
  \draw[<-] (0,0) to (1,2);
  \draw[->] (1,0) to (0,2);
  \draw (2,1.1) node {${\mapsto} \ \ -$};
  \draw[<-] (3,0) to (4,2);
  \draw[->] (4,0) to (3,2);
\end{tikzpicture}\quad ,
\]
and the action on dots is
\[
  \begin{tikzpicture}[>=stealth,baseline=25pt]
    \draw[->] (-1,0) to (-1,2);
    \filldraw (-1,1) circle (2pt);
    \draw (1.5,1) node {$\mapsto \ \ -\sum_{n=0}^\infty \left(q^{-1}-1\right)^n$};
    \draw[->] (4,0) to (4,2);
    \filldraw (4,1) circle (2pt) node[anchor=west] {$n+1$};
  \end{tikzpicture}\quad .
\]
A straightforward computation shows that $\xi_1$ is an involutive monoidal covariant autoequivalence of $\tilde \cH'(q)$.  One sees immediately that, when $q=1$, the infinite sums in the definition of $\xi_1$ become finite, there is no need to pass to the completion $\tilde \cH'(q)$, and this autoequivalence reduces to the one defined in \cite{Kho10}.

\section{Categorification of the Heisenberg algebra} \label{sec:categorification}

In this section, we assume $\kk$ is a field of characteristic zero and $q \in \kk^*$ is not a nontrivial root of unity.

\subsection{Projectors}

Let $\cH(q)$ be the Karoubi envelope of $\cH'(q)$.  More precisely,
the objects of $\cH(q)$ are pairs $(Q_\varepsilon,e)$ where $e:
Q_\varepsilon \to Q_\varepsilon$ is an idempotent endomorphism,
$e^2=e$.  Morphisms $(Q_\varepsilon, e) \to (Q_{\varepsilon'},e')$
are morphisms $f : Q_\varepsilon \to Q_{\varepsilon'}$ in $\cH'(q)$
such that the following diagram is commutative.
\[
  \xymatrix{
    Q_\varepsilon \ar[r]^f \ar[dr]^f \ar[d]_e & Q_{\varepsilon'} \ar[d]^{e'} \\
    Q_\varepsilon \ar[r]_f & Q_{\varepsilon'}
  }
\]
In the case $q=1$, the category $\cH(1)$ is the (conjectural) categorification of the
Heisenberg algebra defined by Khovanov \cite{Kho10}.

It follows from the local relations \eqref{eq:local-relation-basic-Hecke} and \eqref{eq:local-relation-braid} that upward oriented crossings satisfy the Hecke algebra relations and so we have a canonical homomorphism
\begin{equation} \label{eq:Hn-to-Q+}
  H_n \to \End_{\cH'(q)} (Q_{+^n}).
\end{equation}
Similarly, since each space of morphisms in $\cH'(q)$ consists of diagrams up to isotopy, downward
oriented crossings also satisfy the Hecke algebra relations and give us a canonical homomorphism
\begin{equation} \label{eq:Hn-to-Q-}
  H_n \to \End_{\cH'(q)} (Q_{-^n}).
\end{equation}
Introduce the complete $q$-symmetrizer and $q$-antisymmetrizer
\begin{equation} \label{eq:q-symm-def}
  e(n) = \frac{1}{[n]_q!} \sum_{w \in S_n} t_w,\quad e'(n) = \frac{1}{[n]_{q^{-1}}!} \sum_{w \in S_n} (-q)^{-l(w)} t_w, \quad \text{where} \quad [n]_q = \sum_{i=0}^{n-1} q^i.
\end{equation}
Both $e(n)$ and $e'(n)$ are idempotents in $H_n$ (see
\cite[\S1]{Gyo86}).  We will use the notation $e(n)$ and $e'(n)$ to
also denote the image of these idempotents in $\End_{\cH'(q)}
(Q_{+^n})$ and $\End_{\cH'(q)} (Q_{-^n})$ under the canonical
homomorphisms~\eqref{eq:Hn-to-Q+} and~\eqref{eq:Hn-to-Q-}.  We then define the following objects in $\cH(q)$:
\[
  S_+^n = (Q_{+^n}, e(n)),\quad S_-^n = (Q_{-^n}, e(n)),\quad \Lambda_+^n = (Q_{+^n}, e'(n)),\quad \Lambda_-^n = (Q_{-^n}, e'(n)).
\]
Following \cite[\S6.1, 6.2]{Cvi08}, which contains diagrammatics for Young symmetrizers and antisymmetrizers for the symmetric group, we depict $S_+^n$ as a white box labeled $n$.  The inclusion morphism $S_+^n \to Q_{+^n}$ is depicted by a white box with $n$ upward oriented lines leaving from the top.  The projection $Q_{+^n} \to S_+^n$ is depicted by a white box with $n$ upward oriented lines entering the bottom.  The composition $Q_+^n \to S_+^n \to Q_{+^n}$ is depicted by a white box with $n$ upward oriented lines leaving the top and $n$ upwards oriented lines entering the bottom.
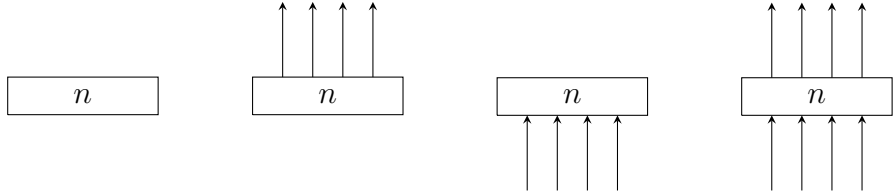
\[
  \begin{tikzpicture}[baseline=-29pt]
    \draw (0,0) to (0,.5) to (2,.5) to (2,0) to (0,0);
    \draw (1,.25) node {$n$};
  \end{tikzpicture}
  \qquad \quad
  \begin{tikzpicture}[>=stealth,baseline=-29pt]
    \draw (0,0) to (0,.5) to (2,.5) to (2,0) to (0,0);
    \draw (1,.25) node {$n$};
    \draw[->] (.4,.5) to (.4,1.5);
    \draw[->] (.8,.5) to (.8,1.5);
    \draw[->] (1.2,.5) to (1.2,1.5);
    \draw[->] (1.6,.5) to (1.6,1.5);
  \end{tikzpicture}
  \qquad \quad
  \begin{tikzpicture}[>=stealth]
    \draw (0,0) to (0,.5) to (2,.5) to (2,0) to (0,0);
    \draw (1,.25) node {$n$};
    \draw[<-] (.4,0) to (.4,-1);
    \draw[<-] (.8,0) to (.8,-1);
    \draw[<-] (1.2,0) to (1.2,-1);
    \draw[<-] (1.6,0) to (1.6,-1);
  \end{tikzpicture}
  \qquad \quad
  \begin{tikzpicture}[>=stealth]
    \draw (0,0) to (0,.5) to (2,.5) to (2,0) to (0,0);
    \draw (1,.25) node {$n$};
    \draw[<-] (.4,0) to (.4,-1);
    \draw[<-] (.8,0) to (.8,-1);
    \draw[<-] (1.2,0) to (1.2,-1);
    \draw[<-] (1.6,0) to (1.6,-1);
    \draw[->] (.4,.5) to (.4,1.5);
    \draw[->] (.8,.5) to (.8,1.5);
    \draw[->] (1.2,.5) to (1.2,1.5);
    \draw[->] (1.6,.5) to (1.6,1.5);
  \end{tikzpicture}
\]

We depict the object $\Lambda_+^n$ and its related inclusions and projections by the same diagrams but with white boxes replaced by black boxes.
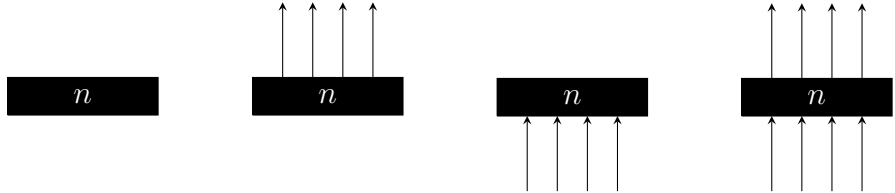
\[
  \begin{tikzpicture}[baseline=-29pt]
    \filldraw (0,0) to (0,.5) to (2,.5) to (2,0) to (0,0);
    \draw[color=white] (1,.25) node {$n$};
  \end{tikzpicture}
  \qquad \quad
  \begin{tikzpicture}[>=stealth,baseline=-29pt]
    \filldraw (0,0) to (0,.5) to (2,.5) to (2,0) to (0,0);
    \draw[color=white] (1,.25) node {$n$};
    \draw[->] (.4,.5) to (.4,1.5);
    \draw[->] (.8,.5) to (.8,1.5);
    \draw[->] (1.2,.5) to (1.2,1.5);
    \draw[->] (1.6,.5) to (1.6,1.5);
  \end{tikzpicture}
  \qquad \quad
  \begin{tikzpicture}[>=stealth]
    \filldraw (0,0) to (0,.5) to (2,.5) to (2,0) to (0,0);
    \draw[color=white] (1,.25) node {$n$};
    \draw[<-] (.4,0) to (.4,-1);
    \draw[<-] (.8,0) to (.8,-1);
    \draw[<-] (1.2,0) to (1.2,-1);
    \draw[<-] (1.6,0) to (1.6,-1);
  \end{tikzpicture}
  \qquad \quad
  \begin{tikzpicture}[>=stealth]
    \filldraw (0,0) to (0,.5) to (2,.5) to (2,0) to (0,0);
    \draw[color=white] (1,.25) node {$n$};
    \draw[<-] (.4,0) to (.4,-1);
    \draw[<-] (.8,0) to (.8,-1);
    \draw[<-] (1.2,0) to (1.2,-1);
    \draw[<-] (1.6,0) to (1.6,-1);
    \draw[->] (.4,.5) to (.4,1.5);
    \draw[->] (.8,.5) to (.8,1.5);
    \draw[->] (1.2,.5) to (1.2,1.5);
    \draw[->] (1.6,.5) to (1.6,1.5);
  \end{tikzpicture}
\]
The objects $S_-^n$ and $\Lambda_-^n$, together with their related inclusions and projections, are depicted by the same diagrams but with downward oriented lines instead of upward oriented lines.

\begin{lemma} \label{lem:absorb-symmetrizers}
Crossings are absorbed into $q$-symmetrizers at the cost of a factor of $q$ and into $q$-antisymmetrizers at the cost of a factor of $-1$.  More precisely, we have the following equalities.
\[
  \begin{tikzpicture}
    \draw (0,0) to (0,.5) to (2.4,.5) to (2.4,0) to (0,0);
    \draw (.4,.5) to (.4,1.5);
    \draw (.8,.5) to (.8,1.5);
    \draw (1.2,.5) to (1.6,1.5);
    \draw (1.6,.5) to (1.2,1.5);
    \draw (2,.5) to (2,1.5);
    \draw (3.25,.75) node {$=\ q$};
    \draw (4,0) to (4,.5) to (6.4,.5) to (6.4,0) to (4,0);
    \draw (4.4,.5) to (4.4,1.5);
    \draw (4.8,.5) to (4.8,1.5);
    \draw (5.2,.5) to (5.2,1.5);
    \draw (5.6,.5) to (5.6,1.5);
    \draw (6,.5) to (6,1.5);
  \end{tikzpicture}
  \qquad \qquad
  \begin{tikzpicture}
    \filldraw (0,0) to (0,.5) to (2.4,.5) to (2.4,0) to (0,0);
    \draw (.4,.5) to (.4,1.5);
    \draw (.8,.5) to (.8,1.5);
    \draw (1.2,.5) to (1.6,1.5);
    \draw (1.6,.5) to (1.2,1.5);
    \draw (2,.5) to (2,1.5);
    \draw (3.25,.75) node {$=\ -$};
    \filldraw (4,0) to (4,.5) to (6.4,.5) to (6.4,0) to (4,0);
    \draw (4.4,.5) to (4.4,1.5);
    \draw (4.8,.5) to (4.8,1.5);
    \draw (5.2,.5) to (5.2,1.5);
    \draw (5.6,.5) to (5.6,1.5);
    \draw (6,.5) to (6,1.5);
  \end{tikzpicture}
\]
Here the arrows can either be oriented up or down.  We also have analogous relations with the lines emanating from the bottom of the boxes instead of the top.
\end{lemma}

\begin{proof}
This follows immediately from \eqref{eq:q-symm-def} (see \cite[p.~843]{Gyo86}).
\end{proof}

\begin{lemma} \label{lem:symmetrizer-pullthrough-relation}
We have following relation.
\[
  \begin{tikzpicture}[>=stealth]
    \draw (0,0) to (0,.5) to (2,.5) to (2,0) to (0,0);
    \draw (0,2.5) to (0,3) to (2,3) to (2,2.5) to (0,2.5);
    \draw (1,.25) node {$n$};
    \draw (1,2.75) node {$n$};
    \draw[<-] (.4,.5) to (.4,2.5);
    \draw[<-] (.8,.5) to (.8,2.5);
    \draw[<-] (1.2,.5) to (1.2,2.5);
    \draw[<-] (1.6,.5) to (1.6,2.5);
    \draw[->] (2.5,0) .. controls (2.5,1) and (0,1) .. (0,1.5) .. controls (0,2) and (2.5,2) .. (2.5,3);
    \draw (3.6,1.5) node {$=\ \ q^n$};
    \draw (4.5,0) to (4.5,.5) to (6.5,.5) to (6.5,0) to (4.5,0);
    \draw (4.5,2.5) to (4.5,3) to (6.5,3) to (6.5,2.5) to (4.5,2.5);
    \draw (5.5,.25) node {$n$};
    \draw (5.5,2.75) node {$n$};
    \draw[<-] (4.9,.5) to (4.9,2.5);
    \draw[<-] (5.3,.5) to (5.3,2.5);
    \draw[<-] (5.7,.5) to (5.7,2.5);
    \draw[<-] (6.1,.5) to (6.1,2.5);
    \draw[->] (7,0) to (7,3);
    \draw (8.5,1.5) node {$-\ \ q^n[n]_q$};
    \draw (9.7,0) to (9.7,.5) to (11.7,.5) to (11.7,0) to (9.7,0);
    \draw (9.7,2.5) to (9.7,3) to (11.7,3) to (11.7,2.5) to (9.7,2.5);
    \draw (10.7,.25) node {$n$};
    \draw (10.7,2.75) node {$n$};
    \draw[<-] (10.1,.5) to (10.1,2.5);
    \draw[<-] (10.5,.5) to (10.5,2.5);
    \draw[<-] (10.9,.5) to (10.9,2.5);
    \draw[<-] (11.3,.5) .. controls (11.3,1.5) and (12.2,1.5) .. (12.2,0);
    \draw[->] (11.3,2.5) .. controls (11.3,1.5) and (12.2,1.5) .. (12.2,3);
  \end{tikzpicture}
\]
\end{lemma}

\begin{proof}
We prove this result by induction.  The case $n=1$ is simply the left hand relation in~\eqref{eq:local-relation-up-down-double-crossing}.  Now assume the result holds for $n-1$.  Then
\begin{align*}
  \begin{tikzpicture}[>=stealth]
    \draw (0,0) to (0,.5) to (2,.5) to (2,0) to (0,0);
    \draw (0,2.5) to (0,3) to (2,3) to (2,2.5) to (0,2.5);
    \draw (1,.25) node {$n$};
    \draw (1,2.75) node {$n$};
    \draw[<-] (.4,.5) to (.4,2.5);
    \draw[<-] (.8,.5) to (.8,2.5);
    \draw[<-] (1.2,.5) to (1.2,2.5);
    \draw[<-] (1.6,.5) to (1.6,2.5);
    \draw[->] (2.5,0) .. controls (2.5,1) and (0,1) .. (0,1.5) .. controls (0,2) and (2.5,2) .. (2.5,3);
  \end{tikzpicture}
  &
  \begin{tikzpicture}[>=stealth]
    \draw (-.5,1.5) node {$=\ q$};
    \draw (0,0) to (0,.5) to (2,.5) to (2,0) to (0,0);
    \draw (0,2.5) to (0,3) to (2,3) to (2,2.5) to (0,2.5);
    \draw (1,.25) node {$n$};
    \draw (1,2.75) node {$n$};
    \draw[<-] (.4,.5) to (.4,2.5);
    \draw[<-] (.8,.5) to (.8,2.5);
    \draw[<-] (1.2,.5) to (1.2,2.5);
    \draw[<-] (1.6,.5) to (1.6,2.5);
    \draw[->] (2.5,0) .. controls (2.5,1) and (.6,1) .. (.6,1.5) .. controls (.6,2) and (2.5,2) .. (2.5,3);
  \end{tikzpicture}
  \begin{tikzpicture}[>=stealth]
    \draw (-.5,1.5) node {$-\ q$};
    \draw (0,0) to (0,.5) to (2,.5) to (2,0) to (0,0);
    \draw (0,2.5) to (0,3) to (2,3) to (2,2.5) to (0,2.5);
    \draw (1,.25) node {$n$};
    \draw (1,2.75) node {$n$};
    \draw[<-] (.8,.5) to (.8,2.5);
    \draw[<-] (1.2,.5) to (1.2,2.5);
    \draw[<-] (1.6,.5) to (1.6,2.5);
    \draw[->] (.4,2.5) .. controls (.4,1.5) and (2.5,1.5) .. (2.5,3);
    \draw[<-] (.4,.5) .. controls (.4,1.5) and (2.5,1.5) .. (2.5,0);
  \end{tikzpicture}
  \\ &
  \begin{tikzpicture}[>=stealth]
    \draw (-.5,1.5) node {$=\ q$};
    \draw (0,0) to (0,.5) to (2,.5) to (2,0) to (0,0);
    \draw (0,2.5) to (0,3) to (2,3) to (2,2.5) to (0,2.5);
    \draw (1,.25) node {$n$};
    \draw (1,2.75) node {$n$};
    \draw[<-] (.4,.5) to (.4,2.5);
    \draw[<-] (.8,.5) to (.8,2.5);
    \draw[<-] (1.2,.5) to (1.2,2.5);
    \draw[<-] (1.6,.5) to (1.6,2.5);
    \draw[->] (2.5,0) .. controls (2.5,1) and (.6,1) .. (.6,1.5) .. controls (.6,2) and (2.5,2) .. (2.5,3);
  \end{tikzpicture}
  \begin{tikzpicture}[>=stealth]
    \draw (-1,1.5) node {$-\ q^{2n-1}$};
    \draw (0,0) to (0,.5) to (2,.5) to (2,0) to (0,0);
    \draw (0,2.5) to (0,3) to (2,3) to (2,2.5) to (0,2.5);
    \draw (1,.25) node {$n$};
    \draw (1,2.75) node {$n$};
    \draw[<-] (.4,.5) to (.4,2.5);
    \draw[<-] (.8,.5) to (.8,2.5);
    \draw[<-] (1.2,.5) to (1.2,2.5);
    \draw[->] (1.6,2.5) .. controls (1.6,1.5) and (2.5,1.5) .. (2.5,3);
    \draw[<-] (1.6,.5) .. controls (1.6,1.5) and (2.5,1.5) .. (2.5,0);
  \end{tikzpicture}
  \\ &
  \begin{tikzpicture}[>=stealth]
    \draw (3.6,1.5) node {$=\ q^n$};
    \draw (4.5,0) to (4.5,.5) to (6.5,.5) to (6.5,0) to (4.5,0);
    \draw (4.5,2.5) to (4.5,3) to (6.5,3) to (6.5,2.5) to (4.5,2.5);
    \draw (5.5,.25) node {$n$};
    \draw (5.5,2.75) node {$n$};
    \draw[<-] (4.9,.5) to (4.9,2.5);
    \draw[<-] (5.3,.5) to (5.3,2.5);
    \draw[<-] (5.7,.5) to (5.7,2.5);
    \draw[<-] (6.1,.5) to (6.1,2.5);
    \draw[->] (7,0) to (7,3);
    \draw (8.5,1.5) node {$-\ \ q^n[n-1]_q$};
    \draw (9.7,0) to (9.7,.5) to (11.7,.5) to (11.7,0) to (9.7,0);
    \draw (9.7,2.5) to (9.7,3) to (11.7,3) to (11.7,2.5) to (9.7,2.5);
    \draw (10.7,.25) node {$n$};
    \draw (10.7,2.75) node {$n$};
    \draw[<-] (10.1,.5) to (10.1,2.5);
    \draw[<-] (10.5,.5) to (10.5,2.5);
    \draw[<-] (10.9,.5) to (10.9,2.5);
    \draw[<-] (11.3,.5) .. controls (11.3,1.5) and (12.2,1.5) .. (12.2,0);
    \draw[->] (11.3,2.5) .. controls (11.3,1.5) and (12.2,1.5) .. (12.2,3);
  \end{tikzpicture}
  \begin{tikzpicture}[>=stealth]
    \draw (-1,1.5) node {$-\ q^{2n-1}$};
    \draw (0,0) to (0,.5) to (2,.5) to (2,0) to (0,0);
    \draw (0,2.5) to (0,3) to (2,3) to (2,2.5) to (0,2.5);
    \draw (1,.25) node {$n$};
    \draw (1,2.75) node {$n$};
    \draw[<-] (.4,.5) to (.4,2.5);
    \draw[<-] (.8,.5) to (.8,2.5);
    \draw[<-] (1.2,.5) to (1.2,2.5);
    \draw[->] (1.6,2.5) .. controls (1.6,1.5) and (2.5,1.5) .. (2.5,3);
    \draw[<-] (1.6,.5) .. controls (1.6,1.5) and (2.5,1.5) .. (2.5,0);
  \end{tikzpicture}
  \\ &
  \begin{tikzpicture}[>=stealth]
    \draw (3.6,1.5) node {$=\ \ q^n$};
    \draw (4.5,0) to (4.5,.5) to (6.5,.5) to (6.5,0) to (4.5,0);
    \draw (4.5,2.5) to (4.5,3) to (6.5,3) to (6.5,2.5) to (4.5,2.5);
    \draw (5.5,.25) node {$n$};
    \draw (5.5,2.75) node {$n$};
    \draw[<-] (4.9,.5) to (4.9,2.5);
    \draw[<-] (5.3,.5) to (5.3,2.5);
    \draw[<-] (5.7,.5) to (5.7,2.5);
    \draw[<-] (6.1,.5) to (6.1,2.5);
    \draw[->] (7,0) to (7,3);
    \draw (8.5,1.5) node {$-\ \ q^n[n]_q$};
    \draw (9.7,0) to (9.7,.5) to (11.7,.5) to (11.7,0) to (9.7,0);
    \draw (9.7,2.5) to (9.7,3) to (11.7,3) to (11.7,2.5) to (9.7,2.5);
    \draw (10.7,.25) node {$n$};
    \draw (10.7,2.75) node {$n$};
    \draw[<-] (10.1,.5) to (10.1,2.5);
    \draw[<-] (10.5,.5) to (10.5,2.5);
    \draw[<-] (10.9,.5) to (10.9,2.5);
    \draw[<-] (11.3,.5) .. controls (11.3,1.5) and (12.2,1.5) .. (12.2,0);
    \draw[->] (11.3,2.5) .. controls (11.3,1.5) and (12.2,1.5) .. (12.2,3);
  \end{tikzpicture}
\end{align*}
where in the third equality we used the inductive hypothesis and the fact that a symmetrizer of size $n-1$ on top of (or below) a symmetrizer of size $n$ is equal to a symmetrizer of size $n$.
\end{proof}

\begin{lemma} \label{lem:antisymmetrizer-pullthrough}
We have the following relation.
\[
  \begin{tikzpicture}[>=stealth]
    \filldraw (0,0) to (0,.5) to (2,.5) to (2,0) to (0,0);
    \filldraw (0,2.5) to (0,3) to (2,3) to (2,2.5) to (0,2.5);
    \draw[color=white] (1,.25) node {$m$};
    \draw[color=white] (1,2.75) node {$m$};
    \draw[->] (.4,.5) to (.4,2.5);
    \draw[->] (.8,.5) to (.8,2.5);
    \draw[->] (1.2,.5) to (1.2,2.5);
    \draw[->] (1.6,.5) to (1.6,2.5);
    \draw[<-] (-.5,0) .. controls (-.5,1) and (2,1) .. (2,1.5) .. controls (2,2) and (-.5,2) .. (-.5,3);
    \draw (3.5,1.5) node {$=\ \ q^m$};
    \draw[<-] (4.5,0) to (4.5,3);
    \filldraw (5,0) to (5,.5) to (7,.5) to (7,0) to (5,0);
    \filldraw (5,2.5) to (5,3) to (7,3) to (7,2.5) to (5,2.5);
    \draw[color=white] (6,.25) node {$m$};
    \draw[color=white] (6,2.75) node {$m$};
    \draw[->] (5.4,.5) to (5.4,2.5);
    \draw[->] (5.8,.5) to (5.8,2.5);
    \draw[->] (6.2,.5) to (6.2,2.5);
    \draw[->] (6.6,.5) to (6.6,2.5);
    \draw (8.5,1.5) node {$-\ \ q[m]_q$};
    \filldraw (10,0) to (10,.5) to (12,.5) to (12,0) to (10,0);
    \filldraw (10,2.5) to (10,3) to (12,3) to (12,2.5) to (10,2.5);
    \draw[color=white] (11,.25) node {$m$};
    \draw[color=white] (11,2.75) node {$m$};
    \draw[->] (10.4,.5) .. controls (10.4,1.5) and (9.6,1.5) .. (9.6,0);
    \draw[<-] (10.4,2.5) .. controls (10.4,1.5) and (9.6,1.5) .. (9.6,3);
    \draw[->] (10.8,.5) to (10.8,2.5);
    \draw[->] (11.2,.5) to (11.2,2.5);
    \draw[->] (11.6,.5) to (11.6,2.5);
  \end{tikzpicture}
\]
\end{lemma}

\begin{proof}
The proof is similar to that of Lemma~\ref{lem:symmetrizer-pullthrough-relation} and is therefore omitted.
\end{proof}

\begin{lemma} \label{lem:pull-strand-off-symmetrizer}
We have the following relations, where the strands can be oriented up or down.
\[
  \begin{tikzpicture}
    \draw (0,0) to (0,.5) to (2,.5) to (2,0) to (0,0);
    \draw (.4,.5) to (.4,1.5);
    \draw (.8,.5) to (.8,1.5);
    \draw (1.2,.5) to (1.2,1.5);
    \draw (1.6,.5) to (1.6,1.5);
    \draw (.4,0) to (.4,-1);
    \draw (.8,0) to (.8,-1);
    \draw (1.2,0) to (1.2,-1);
    \draw (1.6,0) to (1.6,-1);
    \draw (1,.25) node {$n$};
    \draw (3,.25) node {$=\ \frac{1}{[n]_q}$};
    \draw (4,0) to (4,.5) to (5.6,.5) to (5.6,0) to (4,0);
    \draw (4.4,.5) to (4.4,1.5);
    \draw (4.8,.5) to (4.8,1.5);
    \draw (5.2,.5) to (5.2,1.5);
    \draw (6,-1) to (6,1.5);
    \draw (4.4,0) to (4.4,-1);
    \draw (4.8,0) to (4.8,-1);
    \draw (5.2,0) to (5.2,-1);
    \draw (4.8,.25) node {$n-1$};
    \draw (7,.25) node {$+\ \frac{[n-1]_q}{[n]_q}$};
    \draw (8,1) to (8,1.5) to (9.6,1.5) to (9.6,1) to (8,1);
    \draw (8,-0.5) to (9.6,-0.5) to (9.6,-1) to (8,-1) to (8,-0.5);
    \draw (8.4,1.5) to (8.4,2.25);
    \draw (8.8,1.5) to (8.8,2.25);
    \draw (9.2,1.5) to (9.2,2.25);
    \draw (8.4,-1) to (8.4,-1.75);
    \draw (8.8,-1) to (8.8,-1.75);
    \draw (9.2,-1) to (9.2,-1.75);
    \draw (8.4,1) to (8.4,-0.5);
    \draw (8.8,1) to (8.8,-0.5);
    \draw (9.2,1) .. controls (9.2,0) and (10,0) .. (10,-1.75);
    \draw (9.2,-0.5) .. controls (9.2,0.5) and (10,.5) .. (10,2.25);
    \draw (8.8,1.25) node {$n-1$};
    \draw (8.8,-.75) node {$n-1$};
  \end{tikzpicture}
\]

\[
  \begin{tikzpicture}
    \filldraw (0,0) to (0,.5) to (2,.5) to (2,0) to (0,0);
    \draw (.4,.5) to (.4,1.5);
    \draw (.8,.5) to (.8,1.5);
    \draw (1.2,.5) to (1.2,1.5);
    \draw (1.6,.5) to (1.6,1.5);
    \draw (.4,0) to (.4,-1);
    \draw (.8,0) to (.8,-1);
    \draw (1.2,0) to (1.2,-1);
    \draw (1.6,0) to (1.6,-1);
    \draw[color=white] (1,.25) node {$m$};
    \draw (3,.25) node {$=\ \frac{1}{[m]_{q^{-1}}}$};
    \filldraw (4.4,0) to (4.4,.5) to (6,.5) to (6,0) to (4.4,0);
    \draw (4.8,.5) to (4.8,1.5);
    \draw (5.2,.5) to (5.2,1.5);
    \draw (5.6,.5) to (5.6,1.5);
    \draw (4,-1) to (4,1.5);
    \draw (4.8,0) to (4.8,-1);
    \draw (5.2,0) to (5.2,-1);
    \draw (5.6,0) to (5.6,-1);
    \draw[color=white] (5.2,.25) node {$m-1$};
    \draw (7,.25) node {$-\ \frac{[m-1]_q}{[m]_q}$};
    \filldraw (8,1) to (8,1.5) to (9.6,1.5) to (9.6,1) to (8,1);
    \filldraw (8,-0.5) to (9.6,-0.5) to (9.6,-1) to (8,-1) to (8,-0.5);
    \draw (8.4,1.5) to (8.4,2.25);
    \draw (8.8,1.5) to (8.8,2.25);
    \draw (9.2,1.5) to (9.2,2.25);
    \draw (8.4,-1) to (8.4,-1.75);
    \draw (8.8,-1) to (8.8,-1.75);
    \draw (9.2,-1) to (9.2,-1.75);
    \draw (9.2,1) to (9.2,-0.5);
    \draw (8.8,1) to (8.8,-0.5);
    \draw (8.4,1) .. controls (8.4,0) and (7.6,0) .. (7.6,-1.75);
    \draw (8.4,-0.5) .. controls (8.4,0.5) and (7.6,.5) .. (7.6,2.25);
    \draw[color=white] (8.8,1.25) node {$m-1$};
    \draw[color=white] (8.8,-.75) node {$m-1$};
  \end{tikzpicture}
\]
\end{lemma}

\begin{proof}
These statements are $q$-analogues of Equations (6.10) and (6.19) of \cite{Cvi08}.  We sketch the proof of the second.  The proof of the first is analogous and will be omitted.  First note that
\begin{align*}
  &\begin{tikzpicture}
    \filldraw[shift={+(-.3,0)}] (0,0) to (0,.5) to (2,.5) to (2,0) to (0,0);
    \draw[shift={+(-.3,0)}] (.4,.5) to (.4,1.5);
    \draw[shift={+(-.3,0)}] (.8,.5) to (.8,1.5);
    \draw[shift={+(-.3,0)}] (1.2,.5) to (1.2,1.5);
    \draw[shift={+(-.3,0)}] (1.6,.5) to (1.6,1.5);
    \draw[shift={+(-.3,0)}] (.4,0) to (.4,-1);
    \draw[shift={+(-.3,0)}] (.8,0) to (.8,-1);
    \draw[shift={+(-.3,0)}] (1.2,0) to (1.2,-1);
    \draw[shift={+(-.3,0)}] (1.6,0) to (1.6,-1);
    \draw[color=white] [shift={+(-.3,0)}] (1,.25) node {$m$};
    \draw (2.8,.25) node {$= \frac{1}{[m]_{q^{-1}}} \Bigg($};
    \filldraw (4.4,0) to (4.4,.5) to (6,.5) to (6,0) to (4.4,0);
    \draw (4.8,.5) to (4.8,1.5);
    \draw (5.2,.5) to (5.2,1.5);
    \draw (5.6,.5) to (5.6,1.5);
    \draw (4,-1) to (4,1.5);
    \draw (4.8,0) to (4.8,-1);
    \draw (5.2,0) to (5.2,-1);
    \draw (5.6,0) to (5.6,-1);
    \draw[color=white] (5.2,.25) node {$m-1$};
    \draw (7,.25) node {$+\ (-q)^{-1}$};
    \filldraw[shift={+(4,0)}] (4.4,0) to (4.4,.5) to (6,.5) to (6,0) to (4.4,0);
    \draw[shift={+(4,0)}] (4.8,.5) .. controls (4.8,1) and (4.4,1) .. (4.4,1.5);
    \draw[shift={+(4,0)}] (5.2,.5) to (5.2,1.5);
    \draw[shift={+(4,0)}] (5.6,.5) to (5.6,1.5);
    \draw[shift={+(4,0)}] (4,-1) to (4,.5) .. controls (4,1) and (4.8,1) .. (4.8,1.5);
    \draw[shift={+(4,0)}] (4.8,0) to (4.8,-1);
    \draw[shift={+(4,0)}] (5.2,0) to (5.2,-1);
    \draw[shift={+(4,0)}] (5.6,0) to (5.6,-1);
    \draw[color=white] [shift={+(4,0)}] (5.2,.25) node {$m-1$};
    \draw (11,.25) node {$+\ (-q)^{-2}$};
    \filldraw[shift={+(8,0)}] (4.4,0) to (4.4,.5) to (6,.5) to (6,0) to (4.4,0);
    \draw[shift={+(8,0)}] (4.8,.5) .. controls (4.8,1) and (4.4,1) .. (4.4,1.5);
    \draw[shift={+(8,0)}] (5.2,.5) .. controls (5.2,1) and (4.8,1) .. (4.8,1.5);
    \draw[shift={+(8,0)}] (5.6,.5) to (5.6,1.5);
    \draw[shift={+(8,0)}] (4,-1) to (4,.5) .. controls (4,1) and (5.2,1) .. (5.2,1.5);
    \draw[shift={+(8,0)}] (4.8,0) to (4.8,-1);
    \draw[shift={+(8,0)}] (5.2,0) to (5.2,-1);
    \draw[shift={+(8,0)}] (5.6,0) to (5.6,-1);
    \draw[color=white] [shift={+(8,0)}] (5.2,.25) node {$m-1$};
    \draw (15,.25) node {$+\ \cdots$};
  \end{tikzpicture}
  \\
& \qquad \qquad \qquad \qquad \qquad \qquad \qquad \qquad \qquad \qquad \qquad
  \begin{tikzpicture}
    \draw (-.4,.25) node {$+ \ (-q)^{-(m-1)}$};
    \filldraw[shift={+(-3,0)}] (4.4,0) to (4.4,.5) to (6,.5) to (6,0) to (4.4,0);
    \draw[shift={+(-3,0)}] (4.8,.5) .. controls (4.8,1) and (4.4,1) .. (4.4,1.5);
    \draw[shift={+(-3,0)}] (5.2,.5) .. controls (5.2,1) and (4.8,1) .. (4.8,1.5);
    \draw[shift={+(-3,0)}] (5.6,.5) .. controls (5.6,1) and (5.2,1) .. (5.2,1.5);
    \draw[shift={+(-3,0)}] (4,-1) to (4,.5) .. controls (4,1) and (5.6,1) .. (5.6,1.5);
    \draw[shift={+(-3,0)}] (4.8,0) to (4.8,-1);
    \draw[shift={+(-3,0)}] (5.2,0) to (5.2,-1);
    \draw[shift={+(-3,0)}] (5.6,0) to (5.6,-1);
    \draw[color=white] [shift={+(-3,0)}] (5.2,.25) node {$m-1$};
    \draw (3.5,.25) node {$\Bigg).$};
  \end{tikzpicture}
\end{align*}
We now apply a $q$-antisymmetrizer of size $m-1$ to the rightmost $m-1$ strands.  On the left hand side of the equation, we use the fact that a $q$-antisymmetrizer of size $m$ followed by a $q$-symmetrizer of size $m-1$ equals a $q$-antisymmetrizer of size $m$, to see that this side remains unchanged.  We then use Lemma~\ref{lem:absorb-symmetrizers} to simplify all the terms on the right hand side of the equation.
\end{proof}

Define the following morphisms in $\cH(q)$.
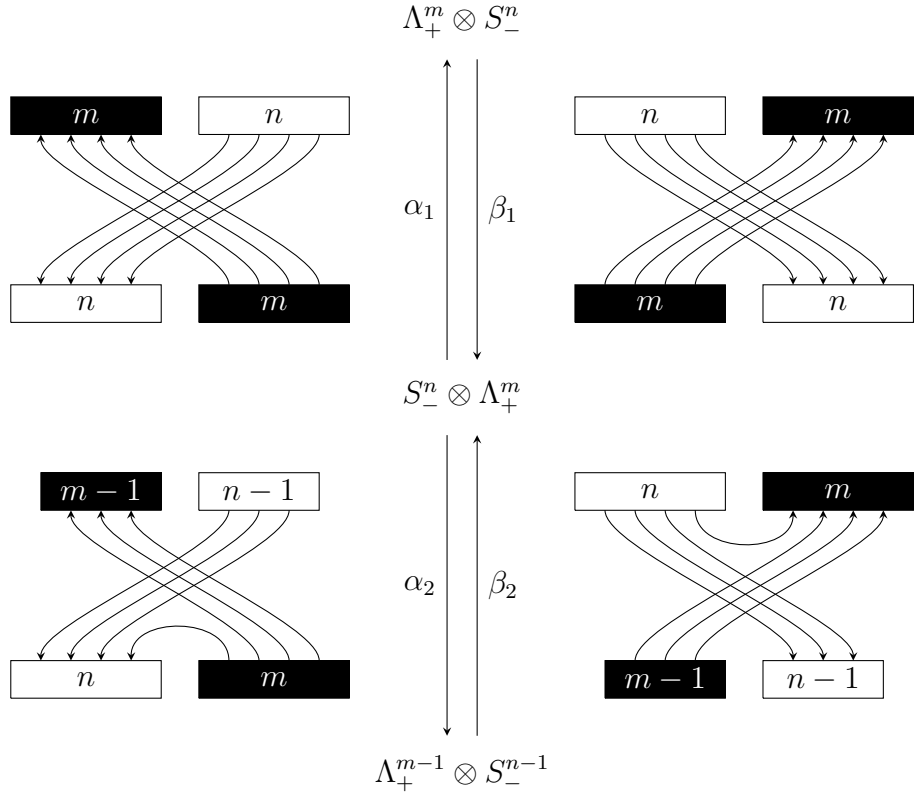
\[
  \begin{tikzpicture}[>=stealth]
    \draw (6,5) node {$\Lambda_+^m \otimes S_-^n$};
    \draw (6,0) node {$S_-^n \otimes \Lambda_+^m$};
    \draw (6,-5) node {$\Lambda_+^{m-1} \otimes S_-^{n-1}$};
    \draw[->] (5.8,.5) to (5.8,2.5) node[anchor=east] {$\alpha_1$} to (5.8,4.5);
    \draw[<-] (6.2,.5) to (6.2,2.5) node[anchor=west] {$\beta_1$} to (6.2,4.5);
    \draw[->] (5.8,-.5) to (5.8,-2.5) node[anchor=east] {$\alpha_2$} to (5.8,-4.5);
    \draw[<-] (6.2,-.5) to (6.2,-2.5) node[anchor=west] {$\beta_2$} to (6.2,-4.5);
    \draw (0,1) to (0,1.5) to (2,1.5) to (2,1) to (0,1);
    \filldraw (0,3.5) to (0,4) to (2,4) to (2,3.5) to (0,3.5);
    \filldraw (2.5,1) to (2.5,1.5) to (4.5,1.5) to (4.5,1) to (2.5,1);
    \draw (2.5,3.5) to (2.5,4) to (4.5,4) to (4.5,3.5) to (2.5,3.5);
    \draw[<-] (.4,3.5) .. controls (.4,3) and (2.9,2) .. (2.9,1.5);
    \draw[<-] (.8,3.5) .. controls (.8,3) and (3.3,2) .. (3.3,1.5);
    \draw[<-] (1.2,3.5) .. controls (1.2,3) and (3.7,2) .. (3.7,1.5);
    \draw[<-] (1.6,3.5) .. controls (1.6,3) and (4.1,2) .. (4.1,1.5);
    \draw[<-] (.4,1.5) .. controls (.4,2) and (2.9,3) .. (2.9,3.5);
    \draw[<-] (.8,1.5) .. controls (.8,2) and (3.3,3) .. (3.3,3.5);
    \draw[<-] (1.2,1.5) .. controls (1.2,2) and (3.7,3) .. (3.7,3.5);
    \draw[<-] (1.6,1.5) .. controls (1.6,2) and (4.1,3) .. (4.1,3.5);
    \draw (1,1.25) node {$n$};
    \draw[color=white] (1,3.75) node {$m$};
    \draw[color=white] (3.5,1.25) node {$m$};
    \draw (3.5,3.75) node {$n$};
    \filldraw (7.5,1) to (7.5,1.5) to (9.5,1.5) to (9.5,1) to (7.5,1);
    \draw (7.5,3.5) to (7.5,4) to (9.5,4) to (9.5,3.5) to (7.5,3.5);
    \draw (10,1) to (10,1.5) to (12,1.5) to (12,1) to (10,1);
    \filldraw (10,3.5) to (10,4) to (12,4) to (12,3.5) to (10,3.5);
    \draw[->] (7.9,3.5) .. controls (7.9,3) and (10.4,2) .. (10.4,1.5);
    \draw[->] (8.3,3.5) .. controls (8.3,3) and (10.8,2) .. (10.8,1.5);
    \draw[->] (8.7,3.5) .. controls (8.7,3) and (11.2,2) .. (11.2,1.5);
    \draw[->] (9.1,3.5) .. controls (9.1,3) and (11.6,2) .. (11.6,1.5);
    \draw[->] (7.9,1.5) .. controls (7.9,2) and (10.4,3) .. (10.4,3.5);
    \draw[->] (8.3,1.5) .. controls (8.3,2) and (10.8,3) .. (10.8,3.5);
    \draw[->] (8.7,1.5) .. controls (8.7,2) and (11.2,3) .. (11.2,3.5);
    \draw[->] (9.1,1.5) .. controls (9.1,2) and (11.6,3) .. (11.6,3.5);
    \draw[color=white] (8.5,1.25) node {$m$};
    \draw (8.5,3.75) node {$n$};
    \draw (11,1.25) node {$n$};
    \draw[color=white] (11,3.75) node {$m$};
    \draw (0,-4) to (0,-3.5) to (2,-3.5) to (2,-4) to (0,-4);
    \filldraw (.4,-1.5) to (.4,-1) to (2,-1) to (2,-1.5) to (.4,-1.5);
    \filldraw (2.5,-4) to (2.5,-3.5) to (4.5,-3.5) to (4.5,-4) to (2.5,-4);
    \draw (2.5,-1.5) to (2.5,-1) to (4.1,-1) to (4.1,-1.5) to (2.5,-1.5);
    \draw[<-] (1.6,-3.5) .. controls (1.6,-2.9) and (2.9,-2.9) .. (2.9,-3.5);
    \draw[<-] (.8,-1.5) .. controls (.8,-2) and (3.3,-3) .. (3.3,-3.5);
    \draw[<-] (1.2,-1.5) .. controls (1.2,-2) and (3.7,-3) .. (3.7,-3.5);
    \draw[<-] (1.6,-1.5) .. controls (1.6,-2) and (4.1,-3) .. (4.1,-3.5);
    \draw[<-] (.4,-3.5) .. controls (.4,-3) and (2.9,-2) .. (2.9,-1.5);
    \draw[<-] (.8,-3.5) .. controls (.8,-3) and (3.3,-2) .. (3.3,-1.5);
    \draw[<-] (1.2,-3.5) .. controls (1.2,-3) and (3.7,-2) .. (3.7,-1.5);
    \draw (1,-3.75) node {$n$};
    \draw[color=white] (1.2,-1.25) node {$m-1$};
    \draw[color=white] (3.5,-3.75) node {$m$};
    \draw (3.3,-1.25) node {$n-1$};
    \filldraw (7.9,-4) to (7.9,-3.5) to (9.5,-3.5) to (9.5,-4) to (7.9,-4);
    \draw (7.5,-1.5) to (7.5,-1) to (9.5,-1) to (9.5,-1.5) to (7.5,-1.5);
    \draw (10,-4) to (10,-3.5) to (11.6,-3.5) to (11.6,-4) to (10,-4);
    \filldraw (10,-1.5) to (10,-1) to (12,-1) to (12,-1.5) to (10,-1.5);
    \draw[->] (7.9,-1.5) .. controls (7.9,-2) and (10.4,-3) .. (10.4,-3.5);
    \draw[->] (8.3,-1.5) .. controls (8.3,-2) and (10.8,-3) .. (10.8,-3.5);
    \draw[->] (8.7,-1.5) .. controls (8.7,-2) and (11.2,-3) .. (11.2,-3.5);
    \draw[->] (9.1,-1.5) .. controls (9.1,-2.1) and (10.4,-2.1) .. (10.4,-1.5);
    \draw[->] (8.3,-3.5) .. controls (8.3,-3) and (10.8,-2) .. (10.8,-1.5);
    \draw[->] (8.7,-3.5) .. controls (8.7,-3) and (11.2,-2) .. (11.2,-1.5);
    \draw[->] (9.1,-3.5) .. controls (9.1,-3) and (11.6,-2) .. (11.6,-1.5);
    \draw[color=white] (8.7,-3.75) node {$m-1$};
    \draw (8.5,-1.25) node {$n$};
    \draw (10.8,-3.75) node {$n-1$};
    \draw[color=white] (11,-1.25) node {$m$};
  \end{tikzpicture}
\]

\begin{proposition} \label{prop:map-compositions}
We have the following:
\begin{align*}
  \alpha_1 \beta_2 &= 0, \\
  \alpha_2 \beta_1 &= 0, \\
  \alpha_1 \beta_1 &= q^{mn} \id, \\
  \alpha_2 \beta_2 &= \frac{q^{(m-1)(n-1)}}{[m]_{q^{-1}}[n]_q} \id.
\end{align*}
\end{proposition}

\begin{proof}
We have
\[
  \begin{tikzpicture}[>=stealth]
    \draw (6.5,-1.25) node {$\alpha_1 \beta_2 =$};
    \filldraw (7.9,-4) to (7.9,-3.5) to (9.5,-3.5) to (9.5,-4) to (7.9,-4);
    \draw (7.5,-1.5) to (7.5,-1) to (9.5,-1) to (9.5,-1.5) to (7.5,-1.5);
    \draw (10,-4) to (10,-3.5) to (11.6,-3.5) to (11.6,-4) to (10,-4);
    \filldraw (10,-1.5) to (10,-1) to (12,-1) to (12,-1.5) to (10,-1.5);
    \draw[->] (7.9,-1.5) .. controls (7.9,-2) and (10.4,-3) .. (10.4,-3.5);
    \draw[->] (8.3,-1.5) .. controls (8.3,-2) and (10.8,-3) .. (10.8,-3.5);
    \draw[->] (8.7,-1.5) .. controls (8.7,-2) and (11.2,-3) .. (11.2,-3.5);
    \draw[->] (9.1,-1.5) .. controls (9.1,-2.1) and (10.4,-2.1) .. (10.4,-1.5);
    \draw[->] (8.3,-3.5) .. controls (8.3,-3) and (10.8,-2) .. (10.8,-1.5);
    \draw[->] (8.7,-3.5) .. controls (8.7,-3) and (11.2,-2) .. (11.2,-1.5);
    \draw[->] (9.1,-3.5) .. controls (9.1,-3) and (11.6,-2) .. (11.6,-1.5);
    \draw[color=white] (8.7,-3.75) node {$m-1$};
    \draw (8.5,-1.25) node {$n$};
    \draw (10.8,-3.75) node {$n-1$};
    \draw[color=white] (11,-1.25) node {$m$};
    \filldraw[shift={+(7.5,-2.5)}] (0,3.5) to (0,4) to (2,4) to (2,3.5) to (0,3.5);
    \draw[shift={+(7.5,-2.5)}] (2.5,3.5) to (2.5,4) to (4.5,4) to (4.5,3.5) to (2.5,3.5);
    \draw[<-][shift={+(7.5,-2.5)}] (.4,3.5) .. controls (.4,3) and (2.9,2) .. (2.9,1.5);
    \draw[<-][shift={+(7.5,-2.5)}] (.8,3.5) .. controls (.8,3) and (3.3,2) .. (3.3,1.5);
    \draw[<-][shift={+(7.5,-2.5)}] (1.2,3.5) .. controls (1.2,3) and (3.7,2) .. (3.7,1.5);
    \draw[<-][shift={+(7.5,-2.5)}] (1.6,3.5) .. controls (1.6,3) and (4.1,2) .. (4.1,1.5);
    \draw[<-][shift={+(7.5,-2.5)}] (.4,1.5) .. controls (.4,2) and (2.9,3) .. (2.9,3.5);
    \draw[<-][shift={+(7.5,-2.5)}] (.8,1.5) .. controls (.8,2) and (3.3,3) .. (3.3,3.5);
    \draw[<-][shift={+(7.5,-2.5)}] (1.2,1.5) .. controls (1.2,2) and (3.7,3) .. (3.7,3.5);
    \draw[<-][shift={+(7.5,-2.5)}] (1.6,1.5) .. controls (1.6,2) and (4.1,3) .. (4.1,3.5);
    \draw[color=white][shift={+(7.5,-2.5)}] (1,3.75) node {$m$};
    \draw[shift={+(7.5,-2.5)}] (3.5,3.75) node {$n$};
  \end{tikzpicture}
  \quad
  \begin{tikzpicture}[>=stealth,baseline=-120pt]
    \draw (7,-1.5) node {$=$};
    \filldraw (7.9,-4) to (7.9,-3.5) to (9.5,-3.5) to (9.5,-4) to (7.9,-4);
    \draw (10,-4) to (10,-3.5) to (11.6,-3.5) to (11.6,-4) to (10,-4);
    \draw[->] (7.9,-1.5) .. controls (7.9,-2) and (10.4,-3) .. (10.4,-3.5);
    \draw[->] (8.3,-1.5) .. controls (8.3,-2) and (10.8,-3) .. (10.8,-3.5);
    \draw[->] (8.7,-1.5) .. controls (8.7,-2) and (11.2,-3) .. (11.2,-3.5);
    \draw (9.1,-1.5) .. controls (9.1,-2.1) and (10.4,-2.1) .. (10.4,-1.5);
    \draw (8.3,-3.5) .. controls (8.3,-3) and (10.8,-2) .. (10.8,-1.5);
    \draw (8.7,-3.5) .. controls (8.7,-3) and (11.2,-2) .. (11.2,-1.5);
    \draw (9.1,-3.5) .. controls (9.1,-3) and (11.6,-2) .. (11.6,-1.5);
    \draw[color=white] (8.7,-3.75) node {$m-1$};
    \draw (10.8,-3.75) node {$n-1$};
    \filldraw[shift={+(7.5,-3)}] (0,3.5) to (0,4) to (2,4) to (2,3.5) to (0,3.5);
    \draw[shift={+(7.5,-3)}] (2.5,3.5) to (2.5,4) to (4.5,4) to (4.5,3.5) to (2.5,3.5);
    \draw[<-][shift={+(7.5,-3)}] (.4,3.5) .. controls (.4,3) and (2.9,2) .. (2.9,1.5);
    \draw[<-][shift={+(7.5,-3)}] (.8,3.5) .. controls (.8,3) and (3.3,2) .. (3.3,1.5);
    \draw[<-][shift={+(7.5,-3)}] (1.2,3.5) .. controls (1.2,3) and (3.7,2) .. (3.7,1.5);
    \draw[<-][shift={+(7.5,-3)}] (1.6,3.5) .. controls (1.6,3) and (4.1,2) .. (4.1,1.5);
    \draw[shift={+(7.5,-3)}] (.4,1.5) .. controls (.4,2) and (2.9,3) .. (2.9,3.5);
    \draw[shift={+(7.5,-3)}] (.8,1.5) .. controls (.8,2) and (3.3,3) .. (3.3,3.5);
    \draw[shift={+(7.5,-3)}] (1.2,1.5) .. controls (1.2,2) and (3.7,3) .. (3.7,3.5);
    \draw[shift={+(7.5,-3)}] (1.6,1.5) .. controls (1.6,2) and (4.1,3) .. (4.1,3.5);
    \draw[color=white][shift={+(7.5,-3)}] (1,3.75) node {$m$};
    \draw[shift={+(7.5,-3)}] (3.5,3.75) node {$n$};
    \draw (12.5,-1.5) node {$=0.$};
  \end{tikzpicture}
\]
In the second equality, we use the fact that since the middle $q$-symmetrizer and $q$-antisym\-metrizer are linear combinations of various crossings, the triple point moves~\eqref{eq:triple2} and~\eqref{eq:triple3} allow us to pull them through through the lines above them and absorb them into the upper $q$-symmetrizer and $q$-antisymmetrizer.  In the third equality, we used the fact that a left curl is zero.  The proof that $\alpha_2 \beta_1 = 0$ is analogous.  The relation $\alpha_1 \beta_1 = q^{mn} \id$ follows immediately from the right hand relation in~\eqref{eq:local-relation-up-down-double-crossing}.

The relation involving $\alpha_2 \beta_2$ is proved as follows.  By applying Lemma~\ref{lem:pull-strand-off-symmetrizer} to the middle two boxes of $\alpha_2 \beta_2$, one gets four summands.  Two of these contain a left curl and are therefore equal to zero.  Another contains a left curl after applying the right hand relation in~\eqref{eq:local-relation-up-down-double-crossing}.  Thus only one summand is nonzero.  This summand contains a counterclockwise circle, which is the identity by~\eqref{eq:cc-circle-and-left-curl}.  Proceeding in the same way as for $\alpha_1 \beta_1$, one gets the factor of $q^{(m-1)(n-1)}$.
\end{proof}

Define
\[
  \beta_1' = q^{-mn} \beta_1,\quad \beta_2' = \frac{[m]_{q^{-1}} [n]_q}{q^{(m-1)(n-1)}} \beta_2.
\]
Then, by Proposition~\ref{prop:map-compositions}, we have
\begin{equation} \label{eq:modified-map-compositions}
  \alpha_1 \beta_2' = 0,\quad \alpha_2 \beta_1'=0,\quad \alpha_1 \beta_1' = \id,\quad \alpha_2 \beta_2' = \id.
\end{equation}

\begin{proposition} \label{prop:map-composition-sum}
We have
\[
  \beta_1' \alpha_1 + \beta_2' \alpha_2 = \id.
\]
\end{proposition}

\begin{proof}
We give only a sketch of the proof, which is a straightforward computation.  First, one computes $\beta_1' \alpha_1$.  As in the proof of Proposition~\ref{prop:map-compositions}, the middle $q$-symmetrizer and $q$-antisymmetrizer can be pulled through the crossings and absorbed into the upper ones.  Then one uses Lemmas~\ref{lem:absorb-symmetrizers} and~\ref{lem:symmetrizer-pullthrough-relation} (or~\ref{lem:antisymmetrizer-pullthrough}) to show that
\[
  \beta_1' \alpha_1 = \id -[n]_q [m]_{q^{-1}}
  \begin{tikzpicture}[>=stealth,baseline=70]
    \draw (0,1) to (0,1.5) to (2,1.5) to (2,1) to (0,1);
    \draw (0,3.5) to (0,4) to (2,4) to (2,3.5) to (0,3.5);
    \filldraw (2.5,1) to (2.5,1.5) to (4.5,1.5) to (4.5,1) to (2.5,1);
    \filldraw (2.5,3.5) to (2.5,4) to (4.5,4) to (4.5,3.5) to (2.5,3.5);
    \draw[->] (.4,3.5) to (.4,1.5);
    \draw[->] (.8,3.5) to (.8,1.5);
    \draw[->] (1.2,3.5) to (1.2,1.5);
    \draw[->] (1.6,3.5) .. controls (1.6,3) and (2.9,3) .. (2.9,3.5);
    \draw[<-] (1.6,1.5) .. controls (1.6,2) and (2.9,2) .. (2.9,1.5);
    \draw[->] (3.3,1.5) to (3.3,3.5);
    \draw[->] (3.7,1.5) to (3.7,3.5);
    \draw[->] (4.1,1.5) to (4.1,3.5);
    \draw (1,1.25) node {$n$};
    \draw (1,3.75) node {$n$};
    \draw[color=white] (3.5,1.25) node {$m$};
    \draw[color=white] (3.5,3.75) node {$m$};
  \end{tikzpicture}
  \ .
\]
In showing this relation (and the next), one notes that any diagram containing a $q$-symmetrizer and $q$-antisymmetrizer connected by two or more strands is zero (this follows immediately from Lemma~\ref{lem:absorb-symmetrizers} since $q \ne -1$).  In a similar fashion, one shows that
\[
  \beta_2' \alpha_2 = [n]_q [m]_{q^{-1}}
  \begin{tikzpicture}[>=stealth,baseline=70]
    \draw (0,1) to (0,1.5) to (2,1.5) to (2,1) to (0,1);
    \draw (0,3.5) to (0,4) to (2,4) to (2,3.5) to (0,3.5);
    \filldraw (2.5,1) to (2.5,1.5) to (4.5,1.5) to (4.5,1) to (2.5,1);
    \filldraw (2.5,3.5) to (2.5,4) to (4.5,4) to (4.5,3.5) to (2.5,3.5);
    \draw[->] (.4,3.5) to (.4,1.5);
    \draw[->] (.8,3.5) to (.8,1.5);
    \draw[->] (1.2,3.5) to (1.2,1.5);
    \draw[->] (1.6,3.5) .. controls (1.6,3) and (2.9,3) .. (2.9,3.5);
    \draw[<-] (1.6,1.5) .. controls (1.6,2) and (2.9,2) .. (2.9,1.5);
    \draw[->] (3.3,1.5) to (3.3,3.5);
    \draw[->] (3.7,1.5) to (3.7,3.5);
    \draw[->] (4.1,1.5) to (4.1,3.5);
    \draw (1,1.25) node {$n$};
    \draw (1,3.75) node {$n$};
    \draw[color=white] (3.5,1.25) node {$m$};
    \draw[color=white] (3.5,3.75) node {$m$};
  \end{tikzpicture}
  \ .
\]
Adding these two expressions then gives the desired relation.
\end{proof}

\begin{theorem} \label{thm:main-object-isom}
In the category $\cH(q)$, we have
\begin{align*}
  S_-^n \otimes S_-^m &\cong S_-^m \otimes S_-^n, \\
  \Lambda_+^n \otimes \Lambda_+^m &\cong \Lambda_+^m \otimes \Lambda_+^n, \\
  S_-^n \otimes \Lambda_+^m &\cong \left( \Lambda_+^m \otimes S_-^n \right) \oplus \left( \Lambda_+^{m-1} \otimes S_-^{n-1} \right).
\end{align*}
\end{theorem}

\begin{proof}
To see the second isomorphism, consider the morphisms
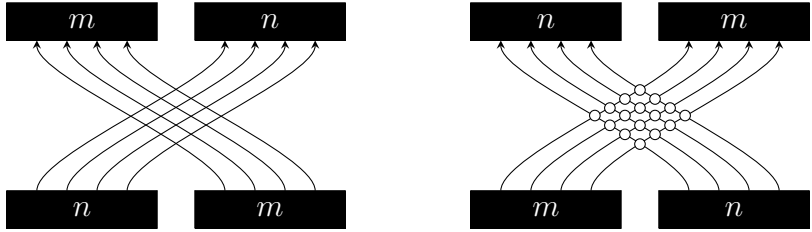
\[
  \begin{tikzpicture}[>=stealth]
    \filldraw (0,1) to (0,1.5) to (2,1.5) to (2,1) to (0,1);
    \filldraw (0,3.5) to (0,4) to (2,4) to (2,3.5) to (0,3.5);
    \filldraw (2.5,1) to (2.5,1.5) to (4.5,1.5) to (4.5,1) to (2.5,1);
    \filldraw (2.5,3.5) to (2.5,4) to (4.5,4) to (4.5,3.5) to (2.5,3.5);
    \draw[<-] (.4,3.5) .. controls (.4,3) and (2.9,2) .. (2.9,1.5);
    \draw[<-] (.8,3.5) .. controls (.8,3) and (3.3,2) .. (3.3,1.5);
    \draw[<-] (1.2,3.5) .. controls (1.2,3) and (3.7,2) .. (3.7,1.5);
    \draw[<-] (1.6,3.5) .. controls (1.6,3) and (4.1,2) .. (4.1,1.5);
    \draw[->] (.4,1.5) .. controls (.4,2) and (2.9,3) .. (2.9,3.5);
    \draw[->] (.8,1.5) .. controls (.8,2) and (3.3,3) .. (3.3,3.5);
    \draw[->] (1.2,1.5) .. controls (1.2,2) and (3.7,3) .. (3.7,3.5);
    \draw[->] (1.6,1.5) .. controls (1.6,2) and (4.1,3) .. (4.1,3.5);
    \draw[color=white] (1,1.25) node {$n$};
    \draw[color=white] (1,3.75) node {$m$};
    \draw[color=white] (3.5,1.25) node {$m$};
    \draw[color=white] (3.5,3.75) node {$n$};
  \end{tikzpicture}
  \qquad \qquad
  \begin{tikzpicture}[>=stealth]
    \filldraw (0,1) to (0,1.5) to (2,1.5) to (2,1) to (0,1);
    \filldraw (0,3.5) to (0,4) to (2,4) to (2,3.5) to (0,3.5);
    \filldraw (2.5,1) to (2.5,1.5) to (4.5,1.5) to (4.5,1) to (2.5,1);
    \filldraw (2.5,3.5) to (2.5,4) to (4.5,4) to (4.5,3.5) to (2.5,3.5);
    \draw[<-] (.4,3.5) .. controls (.4,3) and (2.9,2) .. (2.9,1.5);
    \draw[<-] (.8,3.5) .. controls (.8,3) and (3.3,2) .. (3.3,1.5);
    \draw[<-] (1.2,3.5) .. controls (1.2,3) and (3.7,2) .. (3.7,1.5);
    \draw[<-] (1.6,3.5) .. controls (1.6,3) and (4.1,2) .. (4.1,1.5);
    \draw[->] (.4,1.5) .. controls (.4,2) and (2.9,3) .. (2.9,3.5);
    \draw[->] (.8,1.5) .. controls (.8,2) and (3.3,3) .. (3.3,3.5);
    \draw[->] (1.2,1.5) .. controls (1.2,2) and (3.7,3) .. (3.7,3.5);
    \draw[->] (1.6,1.5) .. controls (1.6,2) and (4.1,3) .. (4.1,3.5);
    \filldraw[fill=white,draw=black] (1.65,2.5) circle(2pt);
    \filldraw[fill=white,draw=black] (2.05,2.5) circle(2pt);
    \filldraw[fill=white,draw=black] (2.45,2.5) circle(2pt);
    \filldraw[fill=white,draw=black] (2.85,2.5) circle(2pt);
    \filldraw[fill=white,draw=black] (1.85,2.6) circle(2pt);
    \filldraw[fill=white,draw=black] (2.25,2.6) circle(2pt);
    \filldraw[fill=white,draw=black] (2.65,2.6) circle(2pt);
    \filldraw[fill=white,draw=black] (1.85,2.37) circle(2pt);
    \filldraw[fill=white,draw=black] (2.25,2.37) circle(2pt);
    \filldraw[fill=white,draw=black] (2.65,2.37) circle(2pt);
    \filldraw[fill=white,draw=black] (2.05,2.72) circle(2pt);
    \filldraw[fill=white,draw=black] (2.45,2.72) circle(2pt);
    \filldraw[fill=white,draw=black] (2.05,2.25) circle(2pt);
    \filldraw[fill=white,draw=black] (2.45,2.25) circle(2pt);
    \filldraw[fill=white,draw=black] (2.25,2.84) circle(2pt);
    \filldraw[fill=white,draw=black] (2.25,2.12) circle(2pt);
    \draw[color=white] (1,1.25) node {$m$};
    \draw[color=white] (1,3.75) node {$n$};
    \draw[color=white] (3.5,1.25) node {$n$};
    \draw[color=white] (3.5,3.75) node {$m$};
  \end{tikzpicture}
\]
where we recall the definition \eqref{eq:inverse-crossing} of the inverse crossing, denoted by an open circle.  One easily verifies that these two morphisms compose in either order to yield the identity.  The proof of the first isomorphism in the statement of the theorem is analogous.  The third isomorphism follows immediately from \eqref{eq:modified-map-compositions} and Proposition~\ref{prop:map-composition-sum}.
\end{proof}

\subsection{The Heisenberg 2-category}

In the sequel, we will define actions of our graphical categories on categories of modules for Hecke algebras and general linear groups over finite fields.  These actions are most naturally described in the language of 2-categories.  Therefore, we define here a 2-category built from the graphical category $\cH(q)$.

\begin{definition}[Heisenberg 2-category]
We define the \emph{Heisenberg 2-category} $\fH'(q)$ as follows.  The objects of $\fH'(q)$ are the integers.  For
$n,m \in \Z = \Ob \fH'(q)$, $\Hom_{\fH'(q)}(n,m)$ is the full
subcategory of $\cH'(q)$ containing the objects $Q_\varepsilon$,
$\varepsilon = \varepsilon_1 \dots \varepsilon_l$, for which
\[
  m-n = \# \{i\ |\ \varepsilon_i = +\} - \# \{i\ |\ \varepsilon_i =
  -\}.
\]
We then define $\fH(q)$ to be the 2-category whose objects are the
integers and for $n,m \in \Z$, $\Hom_{\fH(q)}(n,m)$ is the Karoubi
envelope of $\Hom_{\fH'(q)}(n,m)$.
\end{definition}

\begin{definition}[Representation of $\fH(q)$] \label{def:2-cat-rep}
Let $\mathcal{X} = \bigoplus_{n \in \Z} \mathcal{X}_n$ be a $\Z$-graded additive category.  Then the endofunctors of $\mathcal{X}$ naturally form a 2-category whose objects are the integers, whose 1-morphisms from $n$ to $m$, $n,m \in \Z$, are the functors from $\mathcal{X}_n$ to $\mathcal{X}_m$, and whose 2-morphisms are natural transformations.   A \emph{representation} of $\fH(q)$ is a functor from $\fH(q)$ to such a 2-category of endofunctors.
\end{definition}

\subsection{A modified Heisenberg algebra}

The algebra $\fh_\Z$ is $\Z$-graded after setting
\[
  \deg b_m = m,\quad \deg a_m = -m, \quad m \in \N_+.
\]
For $n \in \Z$, let $\fh_\Z(n)$ be the subspace of $\fh_\Z$ consisting of
homogeneous elements of degree $n$.

Any graded ring gives rise to a preadditive category with an object for each graded piece. In this way, we obtain the following category.

\begin{definition}
Let $\dot \fh_\Z$ be the preadditive category defined as follows:
\begin{itemize}
  \item $\Ob \dot \fh_\Z = \Z$.
  \item For $m,n \in \Z$, the morphisms from $n$ to $m$ are $\fh_\Z(m-n)$.
\end{itemize}
Composition of morphisms is simply given by multiplication in $\fh_\Z$.
\end{definition}

\subsection{A categorification of the Heisenberg algebra}

There is a natural bijection between the set $\cP(n)$ of partitions
of $n$ and the set of isomorphism classes of representations of
$H_n$ (see, for example, \cite[Chapter~3]{Mat99}).  To a partition $\lambda = (\lambda_1,
\dots, \lambda_k)$ of $n$, there corresponds the irreducible
representation $L_\lambda$. This is the unique common irreducible
summand of the representation induced from the trivial
representation of the parabolic subgroup $H_\lambda= H_{\lambda_1} \times \dots \times
H_{\lambda_k}$ of $H_n$, and the representation induced from the sign representation of the parabolic
subgroup $H_{\lambda^*} $, where $\lambda^*$ is the dual partition. Let
$e_\lambda \in H_n$ be the corresponding Young idempotent, so that
$e_\lambda^2 = e_\lambda$ and $L_\lambda = H_n e_\lambda$.

For $\lambda \in \cP(n)$, let $Q_{+,\lambda} = (Q_{+^n}, e_\lambda)
\in \fH(q)$.  Here we are viewing $e_\lambda$ as an idempotent in
$\End_{\cH(q)}(Q_{+^n})$ via the map~\eqref{eq:Hn-to-Q+}.  Similarly,
define $Q_{-,\lambda} = (Q_{-^n}, e_\lambda)$, where we view
$e_\lambda$ as an idempotent in $\End_{\cH(q)}(Q_{-^n})$
via~\eqref{eq:Hn-to-Q-}.  In particular,
\[
  S^n_+ = Q_{+,(n)},\quad \Lambda^n_+ = Q_{+,(1^n)},\quad S^n_- = Q_{-,(n)},
  \quad \Lambda^n_- = Q_{-,(1^n)}.
\]

Recall that if $\mathfrak{C}$ is a 2-category, then the \emph{(split) Grothendieck group} $K_0(\mathfrak{C})$ of $\mathfrak{C}$ is the category whose objects are the objects of $\mathfrak{C}$ and whose sets of morphisms are the (split) Grothendieck groups of the corresponding morphism categories in $\mathfrak{C}$.  The Grothendieck group $K_0(\fH(q))$ of $\fH(q)$ is naturally a preadditive category.

\begin{definition}
Define a functor $\bF : \dot \fh_\Z \to K_0(\fH(q))$ as follows.  On
objects, we define $\bF(n) = n$ for all $n \in Z$.  We define $\bF$
on morphisms by
\[
  \bF(a_n) = [S_-^n],\quad \bF(b_n) = [\Lambda_+^n],
\]
and requiring $\bF$ to be monoidal. The functor $\bF$ is well
defined by Theorem~\ref{thm:main-object-isom}.
\end{definition}

We can identify the subring of $\fh_\Z$ generated by the $a_n$, $n \ge
1$, with the ring of symmetric functions so that $a_n$ corresponds
to the $n$-th complete symmetric function $h_n$.  Let $a_\lambda$
denote the polynomial in the $a_n$'s associated to the Schur
function corresponding to the partition $\lambda$.  Then
\[
  \bF(a_\lambda) = [Q_{-,\lambda}].
\]
Similarly, we identity the subring of $\fh_\Z$ generated by the $b_n$,
$n \ge 1$, with the ring of symmetric functions so that $b_n$
corresponds to the $n$-th elementary symmetric function $e_n$.  Let
$b_\lambda$ denote the polynomial in the $b_n$'s associated to the
Schur function corresponding to the partition $\lambda$.  Then
\[
  \bF(b_\lambda) = [Q_{+,\lambda^*}].
\]
The ring $\fh_\Z$ has a basis $\{b_\lambda a_\mu\}_{\lambda,\mu}$,
where $\lambda$ and $\mu$ run over all partitions.  It follows that
$\{[Q_{+,\mu}] [Q_{-,\lambda}]\}_{\lambda,\mu}$ spans the subring
$\bF ( \dot \fh_\Z )$ of $K_0(\fH(q))$.

\begin{theorem} \label{thm:heisenberg-injection}
For $q$ not a root of unity, the functor $\bF$ is faithful.
\end{theorem}

The proof of Theorem~\ref{thm:heisenberg-injection}  is given in Section~\ref{sec:injection-proof}.

\begin{conjecture}
For $q$ not a root of unity, the functor $\bF$ is an equivalence.
\end{conjecture}

\begin{remark}
When $q$ is a root of unity, the Hecke algebra and its representation theory changes considerably.  It would be interesting to study the category $\cH(q)$ in this case.  When $q$ is not a root of unity, we conjecture that the categories $\cH(q)$ for various $q$ are equivalent to one another.
\end{remark}

\subsection{Symmetries}

The symmetries of $\cH'(q)$ described in Section~\ref{sec:symmetries} naturally induce symmetries of $\cH(q)$, $\fH(q)$, and $K_0(\fH(q))$.  The induced involutions of $K_0(\fH(q))$ preserve the image of $\bF$ (which we identify with $\dot \fh_\Z$).

Since $\xi_2$ preserves the $q$-symmetrizer $e(n)$ and the $q$-antisymmetrizer $e'(n)$ (see \eqref{eq:q-symm-def}), it follows that it preserves the objects $S_+^n$, $S_-^n$, $\Lambda_+^n$, and $\Lambda_-^n$ of $\cH(q)$.  The symmetry $\xi_2$ then naturally induces an involution of $\fH(q)$ that is the identity on objects.  The induced functor on $\dot \fh_\Z$ is the identity.

It is easy to check that the induced action of $\xi_3$ on $\cH(q)$ interchanges $S_+^n$ with $S_-^n$ and $\Lambda_+^n$ with $\Lambda_-^n$.  It induces a symmetry on $\fH(q)$ that, on objects, sends $n$ to $-n$, for $n \in \Z$.  The induced involutive functor on $\dot \fh_\Z$ is the contravariant functor that sends the object $n$, $n \in \Z$, to $-n$, and interchanges $a_n$ and $b_n$.

\section{Action on modules for Hecke algebras}
\label{sec:Hecke-action}

In this section, we will describe how our graphical category acts on the category of modules for Hecke algebras of type $A$.  We refer the reader to \cite{Kho10b} for an overview of the type of diagrammatic presentation of functors, natural transformations, and biadjointness we use here.

\subsection{Bimodules for Hecke algebras}

For $1 \le k \le n$, we can view $H_k$ as a subalgebra of $H_n$ via the embedding $t_i \mapsto t_i$.  We introduce here some notation for bimodules.  First note that $H_n$ is naturally an $(H_n,H_n)$-bimodule. Via our identification of $H_k$, $1 \le k \le n$, with a subalgebra of $H_n$, we can naturally view $H_n$ as an $(H_k,H_l)$-bimodule for $1 \le k,l \le n$.  We will write $_k (n)_l$ to denote this bimodule. If $k$ or $l$ is equal to $n$, we will often omit the subscript. Thus, for instance,
\begin{itemize}
  \item $(n)$ denotes $H_n$, considered as an $(H_n,H_n)$-bimodule,
  \item $(n)_{n-1}$ denotes $H_n$, considered as an
  $(H_n,H_{n-1})$-bimodule, and
  \item $_{n-1}(n)$ denotes $H_n$, considered as an
  $(H_{n-1},H_n)$-bimodule.
\end{itemize}

\subsection{Decompositions}

We collect here various results that will be used in the sequel.

\begin{lemma} \label{lem:H_n+1 over H_n}
The algebra $H_{n+1}$ is free of rank $n+1$ as both a right and a
left $H_n$-module.  In particular, we have the following.
\begin{enumerate}
  \item The set $\{t_n t_{n-1} \dots t_j\ |\ 1 \le j \le n+1\}$ is a
  basis of $H_{n+1}$ as a left $H_n$-module.  Here we interpret $t_n
  t_{n-1} \dots t_j$ as being $1_{n+1}$ when $j=n+1$.

  \item The set $\{t_j \dots t_{n-1} t_n\ |\ 1 \le j \le n+1\}$ is a
  basis of $H_{n+1}$ as a right $H_n$-module.
\end{enumerate}
\end{lemma}

\begin{proof}
We prove the second statement since the first is analogous.  To
prove that our set generates $H_{n+1}$ as a right $H_n$-module, it
suffices to show that each $t_w$, $w \in S_{n+1}$, can be written as
a right $H_n$-multiple of $t_j \cdots t_n$ for some $1 \le j \le
n+1$. Let $w \in S_{n+1}$ and set $j=w(n+1)$.  Then
\[
  s_n s_{n-1} \cdots s_j w(n+1) = n+1.
\]
Thus $s_n s_{n-1} \cdots s_j w = \tilde w$ for some $\tilde w \in
S_n$ (viewed as the subgroup of $S_{n+1}$ fixing $n+1$).  Then
\begin{equation} \label{eq:n n+1 reduced expression}
  w = s_j \cdots s_{n-1} s_n \tilde w
\end{equation}
and any reduced expression of $\tilde w$ gives a reduced expression
of $w$ via substitution in \eqref{eq:n n+1 reduced expression}.
Therefore,
\[
  t_w = t_j \cdots t_{n-1} t_n t_{\tilde w}, \quad t_{\tilde w} \in
  H_n,
\]
as desired.

It remains to show that the $t_j \cdots t_{n-1} t_n$ are linearly
independent.  But this is clear since $t_j \cdots t_{n-1} t_n H_n$
is the right $H_n$-submodule of $H_{n+1}$ spanned by the $t_w$ for
$w \in S_{n+1}$ with $w(n+1) = j$.
\end{proof}

The following lemma, which is a Mackey formula for $H_n$-modules, is well known.  We include a proof for the sake of completeness.

\begin{lemma} \label{lem:key-bimodule-decomp}
We have the following isomorphism of $(H_n,H_n)$-bimodules:
\[
  _n(n+1)_n \cong (n) \oplus \big( (n)_{n-1}(n) \big).
\]
\end{lemma}

\begin{proof}
Let ${\hat H}_{n+1}$ be the subspace of $H_{n+1}$ spanned by $\{t_w\
|\ w \in S_{n+1} \setminus S_n\}$.  That is, ${\hat H}_{n+1}$ is
spanned by those $t_w$ for which $w(n+1) \ne n+1$.  Then it is clear
that
\[
  H_{n+1} \cong H_n \oplus {\hat H}_{n+1} \quad \text{as $(H_n,H_n)$-bimodules}.
\]
It remains to show that ${\hat H}_{n+1}$ is isomorphic to
$(n)_{n-1}(n)$ as an $(H_n,H_n)$-bimodule.  Define a map
\begin{equation} \label{eq:restriction-isom}
  (n)_{n-1}(n) \to {\hat H}_{n+1},\quad t_{w} \otimes t_{w'} \mapsto
  t_w t_n t_{w'}.
\end{equation}
Since elements of $H_{n-1}$ commute with $t_n$, this map is well
defined.  It is also clearly a surjective homomorphism of
$(H_n,H_n)$-bimodules.  Now, the dimension of ${\hat  H}_{n+1}$ is
\[
  \dim {\hat H}_{n+1} = |S_{n+1}| - |S_n| = (n+1)! - n! = n \cdot n!.
\]
On the other hand, the dimension of $(n)_{n-1}(n)$ is the dimension
of $H_n$ times the rank of $H_n$ considered as a right
$H_{n-1}$-module.  Therefore, by Lemma~\ref{lem:H_n+1 over H_n}, we
have
\[
   \dim (n)_{n-1}(n) = n \cdot n! = \dim {\hat H}_{n+1}.
\]
Therefore, the map \eqref{eq:restriction-isom} is an isomorphism as
desired.
\end{proof}

\subsection{Adjunctions} \label{sec:Hecke-adjunctions}

Let $\res$ denote the functor of restriction from the category of $H_n$-modules to the category of $H_{n-1}$-modules and let $\ind$ denote the functor of induction from $H_{n-1}$-modules to $H_n$-modules.  In a slight abuse of notation, we use the notation $\res, \ind$ for different values of $n$.  The goal of this section is to show that these functors are biadjoint.  Note that restriction and induction are realized by tensoring with appropriate bimodules.  In particular:
\begin{align*}
  \res M &= {_{n-1}(n)} \otimes_{H_n} M,\quad M \in H_n\text{-mod},\quad \text{and} \\
  \ind N &= (n)_{n-1} \otimes_{H_{n-1}} N,\quad N \in H_{n-1}\text{-mod}.
\end{align*}
Similarly, compositions of induction and restriction functors are given by tensoring by appropriate bimodules.  Then natural transformations between such functors are simply bimodule homomorphisms.  We can thus either work in the 2-category of categories, functors, and natural transformations or in the 2-category of modules, bimodules, and bimodule homomorphisms.  Most often, it will be convenient for us to work in the language of bimodules.  In any 2-category, one can talk of adjoint 1-morphisms and so we will often talk of adjoint bimodules.

We now define some important bimodule maps.
\begin{enumerate}
\item
Let $\rcap$ denote the bimodule map
\[
  \rcap: (n+1)_n(n+1) \to (n+1),\quad
  x \otimes y \mapsto xy,
\]
given by multiplication.

\item
Let $\rcup$ denote the bimodule map
\[
  \rcup : (n) \hookrightarrow {_n}(n+1)_n,\quad z \mapsto z.
\]

\item
Let $\lcap$ denote the bimodule map
\[
  \lcap: {_n}(n+1)_n \to (n)
\]
given by declaring
\[
  \lcap|_{H_n} = \id, \ \ \lcap(t_n) = 0.
\]

\item
Let $\lcup$ denote the bimodule map
\[
  \lcup: (n+1) \to (n+1)_n(n+1)
\]
determined by
\[
  1_{n+1} \mapsto \sum_{i=1}^{n+1} q^{i-(n+1)} t_{i}\hdots t_{n-1} t_n \otimes t_n t_{n-1}\hdots t_{i}.
\]
We set $t_{n+1}=1$ in the above formula, so that the $i=n+1$ term in the sum is $1\otimes 1$.
\end{enumerate}

It is clear that $\rcap$, $\rcup$, and $\lcap$ are indeed maps of bimodules.
\begin{lemma}
The map $\lcup$ is a map of bimodules.
\end{lemma}
\begin{proof}
It suffices to show that the sum appearing in the definition
commutes with $t_j$ for all $1 \le j \le n$.  Now, if $j< i-1$, we
clearly have
\[
  t_j (t_i \cdots t_n \otimes t_n \cdots t_i) = (t_i \cdots t_n
  \otimes t_n \cdots t_i) t_j.
\]
If $j > i$, then
\begin{align*}
  t_j (t_i \cdots t_n \otimes t_n \cdots t_i) &= t_i t_{i+1} \cdots
  t_{j-2} (t_j t_{j-1} t_j) t_{j+1} \cdots t_n \otimes t_n \cdots t_i
  \\
  &= t_i t_{i+1} \cdots t_{j-2} (t_{j-1} t_j t_{j-1}) t_{j+1} \cdots t_n
  \otimes t_n \cdots t_i
  \\
  &= t_i t_{i+1} \cdots t_n t_{j-1} \otimes t_n \cdots t_i \\
  &= t_i t_{i+1} \cdots t_n \otimes t_{j-1} t_n \cdots t_i \\
  &= t_i t_{i+1} \cdots t_n \otimes t_n \cdots t_{j+1} (t_{j-1} t_j t_{j-1}) t_{j-2} \cdots t_i \\
  &= t_i t_{i+1} \cdots t_n \otimes t_n \cdots t_{j+1} (t_j t_{j-1} t_j) t_{j-2} \cdots t_i \\
  &= (t_i \cdots t_n \otimes t_n \cdots t_i) t_j.
\end{align*}
It remains to consider the terms with $j=i-1,i$.  We have
\begin{gather*}
  t_j \left( q^{j-(n+1)} t_j \cdots t_n \otimes t_n \cdots t_j +
  q^{j-n} t_{j+1} \cdots t_n \otimes t_n \cdots t_{j+1} \right) \\
  = q^{j-(n+1)} (q+(q-1)t_j) t_{j+1} \cdots t_n \otimes t_n \cdots t_j +
  q^{j-n} t_j t_{j+1} \cdots t_n \otimes t_n \cdots t_{j+1} \\
  = q^{j-n} t_{j+1} \cdots t_n \otimes t_n \cdots t_j +
  (q-1)q^{j-(n+1)} t_j \cdots t_n \otimes t_n \cdots t_j + q^{j-n}
  t_j \cdots t_n \otimes t_n \cdots t_{j+1} \\
  = \left( q^{j-(n+1)} t_j \cdots t_n \otimes t_n \cdots t_j +
  q^{j-n} t_{j+1} \cdots t_n \otimes t_n \cdots t_{j+1} \right) t_j.
\end{gather*}
The result follows.
\end{proof}

We will now start using string diagram notation for 2-categories.  In particular, the above bimodule homomorphisms will correspond to diagrams as follows:
\begin{gather*}
\begin{tikzpicture}[>=stealth,baseline=6pt]
  \draw (-1,.25) node {$\rcap = $};
  \draw[->] (0,0) arc(180:0:.5);
  \draw (1.6,.25) node {$n+1$};
\end{tikzpicture}\ , \quad
\begin{tikzpicture}[>=stealth,baseline=6pt]
  \draw (-1,.25) node {$\rcup = $};
  \draw[->] (0,0.5) arc(180:360:.5);
  \draw (1.25,.25) node {$n$};
\end{tikzpicture}\ , \\
\begin{tikzpicture}[>=stealth,baseline=6pt]
  \draw (-1,.25) node {$\lcap = $};
  \draw[<-] (0,0) arc(180:0:.5);
  \draw (1.25,.25) node {$n$};
\end{tikzpicture}\ , \quad
\begin{tikzpicture}[>=stealth,baseline=6pt]
  \draw (-1,.25) node {$\lcup = $};
  \draw[<-] (0,0.5) arc(180:360:.5);
  \draw (1.6,.25) node {$n+1$};
\end{tikzpicture}\ .
\end{gather*}
We refer the reader to \cite{Kho10b} for an overview of this notation for 2-categories.  Note that the labels of the regions of a string diagram are uniquely determined by the label of one region and the fact that, as we move from right to left, labels increase by one as we cross upward pointing strands and decrease by one as we cross downward pointing strands.  An example of a diagram with all regions labeled is as follows.
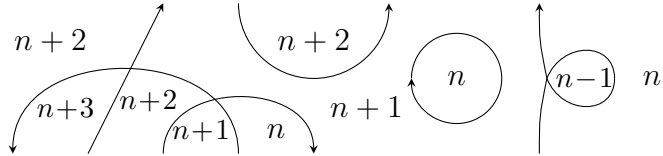
\[
  \begin{tikzpicture}[>=stealth]
    \draw[->] (0,0) to (1,2);
    \draw[->] (1,0) .. controls (1,1) and (3,1) .. (3,0);
    \draw[->] (2,0) .. controls (2,1.5) and (-1,1.5) .. (-1,0);
    \draw[->] (2,2) arc(180:360:1);
    \draw (7,1) .. controls (7,1.5) and (6.3,1.5) .. (6.1,1);
    \draw (7,1) .. controls (7,.5) and (6.3,.5) .. (6.1,1);
    \draw (6,0) .. controls (6,.5) .. (6.1,1);
    \draw (6.1,1) .. controls (6,1.5) .. (6,2) [->];
    \draw[->] (4.3,1) arc(180:-180:.6);
    \draw (7.5,1) node {$n$};
    \draw (6.6,1) node {\small{$n\!-\!1$}};
    \draw (3.7,.6) node {$n+1$};
    \draw (4.9,1) node {$n$};
    \draw (3,1.5) node {$n+2$};
    \draw (-.5,1.5) node {$n+2$};
    \draw (-.3,.6) node {\small$n\!+\!3$};
    \draw (.8,.7) node {\small$n\!+\!2$};
    \draw (1.5,.3) node {\small$n\!+\!1$};
    \draw (2.5,.3) node {$n$};
  \end{tikzpicture}
\]
The above diagram corresponds to a bimodule map
\[
  {_{n+2}(n+3)}_{n+3}(n+3)_{n+2}(n+2)_{n+1}(n+1)_n(n+1)_{n+1}(n+1)_n \to (n+2)_{n+1}(n+2)_{n+2}(n+2)_{n+1}(n+1)_n.
\]
In what follows, we will use various natural bimodule isomorphisms, such as $(n+1)_n(n) \cong (n+1)_n$, without mention.  When we draw a diagram without the regions labeled, we assert that that the relation in question holds for all possible labelings.

\begin{proposition}
The adjunction maps $\lcap$, $\lcup$, $\rcap$, $\rcup$ defined above make
$(\res,\ind)$ into a biadjoint pair.
\end{proposition}

\begin{proof}
This amounts to proving the following four equalities:
\[
  \begin{tikzpicture}[>=stealth,baseline=25pt]
    \draw[->] (0,0) .. controls (0,2) and (1,2) .. (1,1) .. controls (1,0) and (2,0) .. (2,2);
    \draw (2.5,1) node {$=$};
    \draw[->] (3,0) to (3,2);
  \end{tikzpicture}
  \quad,\quad
  \begin{tikzpicture}[>=stealth,baseline=25pt]
    \draw[<-] (0,2) .. controls (0,0) and (1,0) .. (1,1) .. controls (1,2) and (2,2) .. (2,0);
    \draw (2.5,1) node {$=$};
    \draw[->] (3,0) to (3,2);
  \end{tikzpicture}
  \quad,\quad
  \begin{tikzpicture}[>=stealth,baseline=25pt]
    \draw[<-] (0,0) .. controls (0,2) and (1,2) .. (1,1) .. controls (1,0) and (2,0) .. (2,2);
    \draw (2.5,1) node {$=$};
    \draw[<-] (3,0) to (3,2);
  \end{tikzpicture}
  \quad,\quad
  \begin{tikzpicture}[>=stealth,baseline=25pt]
    \draw[->] (0,2) .. controls (0,0) and (1,0) .. (1,1) .. controls (1,2) and (2,2) .. (2,0);
    \draw (2.5,1) node {$=$};
    \draw[<-] (3,0) to (3,2);
  \end{tikzpicture}
  \quad .
\]
We will prove the third equality.  The rest are similar.  If the rightmost region is labeled $n+1$, the left hand side of the third equality is the bimodule homomorphism ${_n(n+1)} \to {_n(n+1)}$ given by the composition
\[
  {_n(n+1)} \xrightarrow{\lcup} {_n(n+1)}_n(n+1) \xrightarrow{\lcap \otimes \id} (n)_n(n+1) \xrightarrow{\cong} {_n(n+1)}.
\]
Since $_n(n+1)$ is generated by $1_{n+1}$, it suffices to determine the image of this element.  We have
\[
  1_{n+1} \mapsto \sum_{i=1}^{n+1} q^{i-(n+1)} t_i \cdots t_{n-1}t_n \otimes t_n t_{n-1} \dots t_i \mapsto 1_n \otimes 1_{n+1} \mapsto 1_{n+1}.
\]
Hence the composition is the identity map, as desired.
\end{proof}

As a result of the above proposition, any endomorphism of $\ind$ defines an endomorphism of $\res$ in two possible ways (one using the adjunctions $\rcap$ and $\rcup$, and one using $\lcap$ and $\lcup$.)  When the two endomorphisms of $\res$ defined in this way are equal to one another (for every endomorphism of $\ind$), the adjunction data $(\rcap, \rcup, \lcap, \lcup)$ is said to be \emph{cyclic}.  We refer the reader to \cite{Kho10b} for a further description of cyclic biadjointness and its relationship to planar diagrammatics for bimodules.

\begin{proposition}\label{prop:cyclicity}
The adjunction data above is cyclic when $q$ is an indeterminate or a prime power.
\end{proposition}

\begin{proof}
We prove this in Section~\ref{sec:parabolic}, where we work in the setting of parabolic induction and restriction for finite groups (see Corollary~\ref{cor:Hecke-cyclicity}).
\end{proof}

\subsection{Crossings}

We define the following bimodule map
\[
  \begin{tikzpicture}[>=stealth,baseline=10pt]
    \draw (-1.2,.5) node {$\ucross =$};
    \draw[->] (0,0) to (1,1);
    \draw[->] (1,0) to (0,1);
    \draw (1.1,.5) node {$n$};
  \end{tikzpicture}
  \quad : \quad (n+2)_n \to (n+2)_n,\quad z \mapsto
  z t_{n+1}.
\]

It follows from Proposition~\ref{prop:cyclicity} that any two isotopic diagrams involving the 2-morphism
\[
  \begin{tikzpicture}[>=stealth]
    \draw[->] (0,0) to (1,1);
    \draw[->] (1,0) to (0,1);
    \draw (1.1,.5) node {$n$};
  \end{tikzpicture}
\]
as well as cups and caps are equal, when $q$ is an indeterminate.  Furthermore, specializing $q$ to any element of $k^\times$ implies that this is true in general.  We
define left, right, and downward crossings to be equal to any
composition of cups and caps and an upward crossing yielding a
bimodule map between the appropriate bimodules.  By the above, any
two such definitions are equivalent.  For instance, we can define
the left, right and downward crossings as follows.
\begin{equation} \label{eq:leftcross-def}
\begin{tikzpicture}[>=stealth,baseline=10pt]
  \draw[<-] (0,0) to (1,1);
  \draw[->] (1,0) to (0,1);
  \draw (1.1,.5) node {$n$};
  \draw (1.7,.5) node {$=$};
  \draw (5,1) .. controls (5,0) and (4,0) .. (3.5,0.5);
  \draw[->] (3.5,.5) .. controls (3,1) and (2,1) .. (2,0);
  \draw[->] (3,0) to (4,1);
  \draw (5.5,.5) node {$n$};
\end{tikzpicture}
\end{equation}
\begin{equation} \label{eq:rightcross-def}
\begin{tikzpicture}[>=stealth,baseline=10pt]
  \draw[->] (0,0) to (1,1);
  \draw[<-] (1,0) to (0,1);
  \draw (1.1,.5) node {$n$};
  \draw (1.7,.5) node {$=$};
  \draw (2,1) .. controls (2,0) and (3,0) .. (3.5,0.5);
  \draw[->] (3.5,.5) .. controls (4,1) and (5,1) .. (5,0);
  \draw[->] (4,0) to (3,1);
  \draw (5.5,.5) node {$n$};
\end{tikzpicture}
\end{equation}
\begin{equation} \label{eq:downcross-def}
\begin{tikzpicture}[>=stealth,baseline=0pt]
  \draw (-5,0) node {$n$};
  \draw[->] (-3.5,.5) to (-4.5,-.5);
  \draw[->] (-4.5,.5) to (-3.5,-.5);
  \draw (-3,0) node {$=$};
  \draw (-2.2,0) node {$n$};
  \draw (-1.5,1.5) .. controls (-1.5,0) and (-1,-1) .. (0,0);
  \draw[->] (0,0) .. controls (1,1) and (1.5,0) .. (1.5,-1.5);
  \draw (-2,1.5) .. controls (-2,-2.5) and (1,-1) .. (0,0);
  \draw[->] (0,0) .. controls (-1,1) and (2,2.5) .. (2,-1.5);
\end{tikzpicture}
\end{equation}

\begin{lemma}
We have
\begin{align*}
  \begin{tikzpicture}[>=stealth,baseline=10pt]
    \draw (-1.5,.5) node {$\dcross=$};
    \draw[<-] (0,0) to (1,1);
    \draw[<-] (1,0) to (0,1);
    \draw (-.1,.5) node {$n$};
  \end{tikzpicture}
  \quad &: \quad {_n(n+2)} \to {_n(n+2)},\quad z \mapsto
  t_{n+1} z, \\
  \begin{tikzpicture}[>=stealth,baseline=10pt]
    \draw (-1.2,.5) node {$\rcross=$};
    \draw[->] (0,0) to (1,1);
    \draw[<-] (1,0) to (0,1);
    \draw (1.1,.5) node {$n$};
  \end{tikzpicture}
  \quad &: \quad (n)_{n-1}(n) \hookrightarrow {_n(n+1)_n},\quad w \otimes z
  \mapsto w t_n z, \\
  \begin{tikzpicture}[>=stealth,baseline=10pt]
    \draw (-1.2,.5) node {$\lcross=$};
    \draw[<-] (0,0) to (1,1);
    \draw[->] (1,0) to (0,1);
    \draw (1.1,.5) node {$n$};
  \end{tikzpicture}
  \quad &: \quad {_n(n+1)_n} \twoheadrightarrow{(n)_{n-1}(n)},\quad
  1_{n+1} \mapsto 0,\quad t_n \mapsto q 1_n \otimes 1_n.
\end{align*}
\end{lemma}

\begin{proof}
We prove the third statement.  The proofs of the others are similar.  We need to compute the following map of bimodules.
\[
\begin{tikzpicture}
  \draw (5,1) .. controls (5,0) and (4,0) .. (3.5,0.5);
  \draw[->] (3.5,.5) .. controls (3,1) and (2,1) .. (2,0);
  \draw[->] (3,0) to (4,1);
  \draw (5.5,.5) node {$n$};
\end{tikzpicture}
\]
This is the bimodule map $_n(n+1)_n \to (n)_{n-1}(n)$ given by the composition
\begin{gather*}
  _n(n+1)_n \xrightarrow{\cong} {_n(n+1)}_n(n)\xrightarrow{\id \otimes \lcup} {_n(n+1)}_n(n)_{n-1}(n) \xrightarrow{\cong} {_n(n+1)_{n-1}(n)} \\
  \xrightarrow{\ucross \otimes \id} {_n(n+1)_{n-1}(n)} \xrightarrow{\cong} {_n(n+1)_n(n)_{n-1}(n)} \xrightarrow{\lcap \otimes \id \otimes \id} (n)_n(n)_{n-1}(n) \xrightarrow{\cong} (n)_{n-1}(n).
\end{gather*}
We compute
\begin{gather*}
  1_{n+1} \mapsto 1_{n+1} \otimes 1_n \mapsto \sum_{i=1}^n 1_{n+1} \otimes q^{i-n} t_i \dots t_{n-1} \otimes t_{n-1} \dots t_i \mapsto \sum_{i=1}^n q^{i-n} t_i \dots t_{n-1} \otimes t_{n-1} \dots t_i \\
  \mapsto \sum_{i=1}^n q^{i-n} t_i \dots t_{n-1} t_n \otimes t_{n-1} \dots t_i \mapsto \sum_{i=1}^n q^{i-n} t_i \dots t_{n-1} t_n \otimes 1_n \otimes t_{n-1} \dots t_i \mapsto 0,
\end{gather*}
and
\begin{gather*}
  t_n \mapsto t_n \otimes 1_n \mapsto \sum_{i=1}^n t_n \otimes q^{i-n} t_i \dots t_{n-1} \otimes t_{n-1} \dots t_i \mapsto \sum_{i=1}^n q^{i-n} t_n t_i \dots t_{n-1} \otimes t_{n-1} \dots t_i \\
  \mapsto \sum_{i=1}^n q^{i-n} t_n t_i \dots t_{n-1} t_n \otimes t_{n-1} \dots t_i \mapsto \sum_{i=1}^n q^{i-n} t_n t_i \dots t_{n-1} t_n \otimes 1_n \otimes  t_{n-1} \dots t_i.
\end{gather*}
Now, all terms except the $i=n$ term are mapped to zero.  The $i=n$ term is equal to
\[
  t_n^2 \otimes 1_n \otimes 1_n = (q+(q-1)t_n) \otimes 1_n \otimes 1_n \mapsto q 1_n \otimes 1_n \otimes 1_n \mapsto q 1_n \otimes 1_n.
\]
\end{proof}

\subsection{Categorification of bosonic Fock space} \label{sec:BFS-cat}

\begin{proposition}\label{prop:relations}
The relations \eqref{eq:local-relation-basic-Hecke}--\eqref{eq:cc-circle-and-left-curl} hold when these diagrams are interpreted as maps of bimodules.
\end{proposition}
\begin{proof}
Relations~\eqref{eq:local-relation-basic-Hecke} and~\eqref{eq:local-relation-braid} follow immediately from the definition of $\ucross$ and the relations
\[
  t_i^2 = q + (q-1)t_i,\quad t_it_{i+1}t_i = t_{i+1} t_i t_{i+1},
\]
in the Hecke algebra.  The remaining relations encode the bimodule decomposition
\[
  _n(n+1)_n \cong  (n) \oplus \big( (n)_{n-1}(n) \big)
\]
of Lemma~\ref{lem:key-bimodule-decomp}.
\end{proof}

\begin{definition}
Let $\fA$ be the 2-category defined as follows.
\begin{itemize}
  \item $\Ob \fA = \N \cup \{\nabla\}$.

  \item The 1-morphisms from $n$ to $m$ for $n,m \in \N$, are functors from $H_n$-mod to $H_m$-mod that are direct summands of compositions of induction and restriction functors.  The only 1-morphism from $\nabla$ to $\nabla$ is the identity.  For $n \in \N$, there are no 1-morphisms from $n$ to $\nabla$ or from $\nabla$ to $n$.

  \item The 2-morphisms are natural transformations of functors.
\end{itemize}
\end{definition}

Note that any $(H_m,H_n)$-bimodule $M$ yields a functor
\[
  H_n\text{-mod} \to H_m\text{-mod},\quad V \mapsto M \otimes V,
\]
and any homomorphism $M_1 \to M_2$ of $(H_m,H_n)$-bimodules gives rise to a natural transformation between the corresponding functors.  In particular, the bimodules $(n)_{n-1}$ and $_{n-1}(n)$ correspond to induction and restriction, respectively.

\begin{definition} \label{def:2-rep-Hecke}
It follows from Proposition~\ref{prop:relations} that we can define a 2-functor $\bA : \fH'(q) \to \fA$ as follows.
\begin{itemize}
  \item For $n \in \Z = \Ob \fH'(q)$, $\bA(n) = n$ if $n \ge 0$, and $\bA(n)=\nabla$ if $n < 0$.

  \item On 1-morphisms, $\bA$ maps $Q_\epsilon \in \Hom_{\fH'(q)}(n,m)$ for a sequence $\epsilon = \epsilon_1 \epsilon_2 \dots \epsilon_k$, to the tensor product of induction and restriction bimodules, where each $+$ corresponds to the induction bimodule and each $-$ corresponds to the restriction bimodule.  For instance, for $Q_{++--+--} \in \Hom_{\fH'(q)}(n,n-1)$,
      \[
        \bA(Q_{++--+--}) = (n-1)_{n-2} (n-2)_{n-3} (n-2)_{n-2} (n-1)_{n-1} (n-1)_{n-2} (n-1)_{n-1} (n).
      \]
      If, for some $k$, the last $k$ terms of $\epsilon$ contain at least $(n+1)$ more $-$'s than $+$'s, then $\bA$ maps $Q_\epsilon \in \Hom_{\fH'(q)}(n,m)$ to 0.

  \item On 2-morphisms, $\bA$ maps a planar diagram (a 1-morphism of $\fH'(q)$) to the corresponding bimodule map (or, more precisely, to the natural transformation corresponding to this bimodule map) according to the definitions given in this section.
\end{itemize}
Since $\fA$ is idempotent complete, the functor $\bA$ induces a functor $\bA : \fH(q) \to \fA$ (which we denote by the same symbol) on the Karoubi envelope $\fH(q)$ of $\fH'(q)$, and hence a functor
\[
  [\bA] : K_0(\fH(q)) \to K_0(\fA).
\]
The functor $\bA$ is a representation of $\fH(q)$ in the sense of Definition~\ref{def:2-cat-rep}.
\end{definition}

\begin{remark}
The reader should compare the functor $\bA$ to the analogous functor defined in \cite{Kho10}.  In \cite{Kho10}, which is in the language of 1-categories, the category in question is monoidal but the functor is not.  This is one of the main motivations for the 2-category point of view.
\end{remark}

Elements of $\Hom_{\fA}(n,m)$ are direct summands of finite compositions of induction and restriction functors from $H_n$-mod to $H_m$-mod.   Therefore, descending to Grothendieck groups, we obtain a functor
\[
  \theta : K_0(\fA) \to \bigoplus_{n,m \in \N} \Hom_\Z (K_0(H_n\text{-mod}),K_0(H_m\text{-mod})),
\]
where we view the bigraded ring $\bigoplus_{n,m \in \N} \Hom_\Z (K_0(H_n\text{-mod}), K_0(H_m\text{-mod}))$ as a category in the natural way.  Namely, the objects are nonnegative integers, and the set of morphisms from $n$ to $m$ is $\Hom_\Z (K_0(H_n\text{-mod}), K_0(H_m\text{-mod}))$.

Consider the composition
\[
  \theta [\bA] : K_0(\fH(q)) \to \bigoplus_{n,m \in \N} \Hom_\Z(K_0(H_n\text{-mod}),K_0(H_m\text{-mod})).
\]
For $[M] \in K_0(H_n\text{-mod})$, we have
\begin{align*}
  (\theta [\bA])([Q_{+,\mu}])([M]) &= \left[\Ind_{H_{|\mu|} \times H_n}^{H_{|\mu|+k}} (L_\mu \otimes M)\right], \\
  (\theta [\bA])([Q_{-,\lambda}])([M]) &=
  \begin{cases}
    0 & \text{if } |\lambda| > n, \\
    \Hom_{H_{|\lambda|}}(L_\lambda,M) \in H_{n-|\lambda|}\text{-mod}, & \text{if } |\lambda| \le n,
  \end{cases}
\end{align*}
where, for a partition $\lambda = (\lambda_1,\lambda_2,\dots,\lambda_k)$, we define $|\lambda| = \sum_{i=1}^k \lambda_k$.
In the expression $\Hom_{H_{|\lambda|}}(L_\lambda,M)$, we restrict $M$ to an $H_{|\lambda|} \times H_{n-|\lambda|}$-module and take homomorphisms from the irreducible module $L_\lambda$ for $H_{|\lambda|}$.  This hom-space is naturally an $H_{n-|\lambda|}$-module.

\begin{remark}
When $\kk$ is a field of characteristic zero and $q \in \kk^\times$ is not a nontrivial root of unity, the functor $\bA : \fH(q) \to \fA$ is a categorification of the (bosonic) Fock space representation of the Heisenberg algebra.
\end{remark}

\subsection{Proof of Theorem~\ref{thm:heisenberg-injection}}
\label{sec:injection-proof}

Since $\Hom_{\dot \fh_\Z} (n,m) = \fh_\Z(m-n) = \Hom_{\dot \fh_\Z}(n+k,m+k)$ for any $m,n,k \in \Z$, we have a natural functor $\mathbf{S}_k : \dot \fh_\Z \to \dot \fh_\Z$ defined on objects by $\mathbf{S}_k(n) = n+k$, $n \in \Z$.  This functor is clearly an isomorphism.

Consider the composition
\[
  \theta [\bA] \bF \mathbf{S}_k : \dot \fh_\Z \to \bigoplus_{m,n \in \N} \Hom_\Z (K_0(H_n\text{-mod}),K_0(H_m\text{-mod})).
\]
We claim that the direct sum of these maps over all $k$ is injective, from which it follows that $\bF$ is faithful.  This sum is clearly injective on objects, and so it suffices to show it is injective on morphisms.  Fix $m,n \in \Z = \Ob \dot \fh_\Z$.  An arbitrary element of $\Hom_{\dot \fh_\Z}(n,m) = \fh_\Z(m-n)$ is a finite sum of the form
\[
  y = \sum_{\lambda,\mu} y_{\lambda,\mu} b_\mu a_\lambda,\quad y_{\lambda,\mu} \in \Z,
\]
where only partitions $\lambda$ and $\mu$ satisfying $|\mu| - |\lambda| = m-n$ appear in the sum.
Assume $y \ne 0$.  Let
\[
  l = \max \{|\lambda|\ |\ y_{\lambda,\mu} \ne 0 \text{ for some } \mu\},
\]
and choose a partition $\nu$ with $|\nu| = l$ such that $y_{\nu,\tau} \ne 0$ for some $\tau$.  Let $k= |\nu| - n$.  We have
\[
  \bF \mathbf{S}_k (y) = \sum_{\lambda,\mu} y_{\lambda,\mu} [Q_{+,\mu^*}] [Q_{-,\lambda}] \in \Hom_{K_0(\fH(q))} (|\nu|, |\tau|).
\]
Now, for $\lambda$ with $|\lambda| = |\nu|$,
\[
  \theta [\bA] \bF \mathbf{S}_k (a_\lambda) = \theta [\bA] ([Q_{-\lambda}]) \in \Hom_\Z (K_0(H_{|\nu|}\text{-mod}), K_0(H_0\text{-mod}))
\]
maps $[L_\rho]$ to 0 if $|\rho| = |\nu|$ and $\rho \ne \lambda$.  It also takes $[L_\lambda]$ to $[L_\varnothing]$, where $L_\varnothing$ is the irreducible module over $H_0 = \kk$.  Thus, $\theta [\bA] \bF \mathbf{S}_k (y)$ takes $[L_\nu]$ to
\[
  \sum_{\mu} y_{\nu,\mu} [L_{\mu^*}] \ne 0,
\]
and so $\theta [\bA] \bF \mathbf{S}_k (y)$ is a nonzero map.

\subsection{Jucys-Murphy elements} \label{sec:JM-elements}

For $k = 0,1,2,\hdots n$, let
\begin{align*}
  L_{k+1} &= \sum_{i=1}^k q^{i-k} t_i \cdots t_k \cdots t_i \\
  &= t_{k} + q^{-1} t_{k-1}t_{k}t_{k-1} +  q^{-2} t_{k-2}t_{k-1}t_kt_{k-1}t_{k-2}
  + \dots + q^{1-k} t_1 \cdots t_k \cdots t_1.
\end{align*}
By convention, $L_1=0$.  The $L_k$ (or, more precisely, $q^{-1}L_k$) are called \emph{Jucys-Murphy elements} of $H_{n+1}$ (see, for example, \cite[\S3.3]{Mat99}).
The $L_k$ are significant in the theory of Hecke algebras, at least in part, because of the following facts.

\begin{enumerate}
    \item The elements $L_k$, $k=1,\hdots n+1$, generate an abelian subalgebra of $H_{n+1}$.

    \item The space of symmetric polynomials in the $L_k$ is the center of
    $H_{n+1}$.

    \item The element $L_k$ in $H_{n+1}$, $1 \le k \le n+1$, commutes with $H_{k-1}$, viewed as a subalgebra of $H_{n+1}$ in the usual way.
\end{enumerate}

\begin{proposition}\label{prop:JM}
Under the definitions of crossings, caps and cups given above, the right curls are the Jucys-Murphy elements.  More precisely,
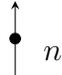
\[
\begin{tikzpicture}[>=stealth]
  \draw[->] (0,0) to (0,1);
  \filldraw (0,.5) circle (2pt);
  \draw (.5,.3) node {$n$};
\end{tikzpicture}
\]
corresponds to the bimodule map
\[
  (n+1)_n \to (n+1)_n,\quad z \mapsto zL_{n+1}.
\]
\end{proposition}
\begin{proof}
This straightforward calculation will be omitted.
\end{proof}

This interpretation of Jucys-Murphy elements as dots in our graphical calculus allows us to give a purely graphical proof of the following well-known relations (see, for example, \cite[Proposition~3.26]{Mat99}).

\begin{corollary}
We have
\begin{enumerate}
  \item $t_l L_{n+1} = L_{n+1} t_l$ for all $1 \le l \le n-1$,
  \item $L_{n+1} t_n = t_n L_n + (q-1) L_{n+1} +q$, and
  \item $t_n L_{n+1} = L_n t_n + (q-1) L_{n+1} +q$.
\end{enumerate}
\end{corollary}

\begin{proof}
This follows immediately from Proposition~\ref{prop:JM} and Lemma~\ref{lem:dot-crossing-moves}.
\end{proof}

\section{Action on modules for finite general linear groups}
\label{sec:GL-action}

In this section, we will describe how our graphical category acts on the category of modules for finite general linear groups.

\subsection{Representations of finite general linear groups} \label{sec:reps-GL}

Let $p$ be a prime number and let $\F_q$ denote a finite field with
$q=p^m$ elements.  We will write $GL_n$ for the group of invertible
$n \times n$ matrices with entries in $\F_q$.  We embed $GL_n$ into
$GL_{n+1}$ in the upper left, so that the image of $GL_n$ in
$GL_{n+1}$ consists of those invertible matrices whose last row is
$(0,\hdots,0,1)$ and whose last column is $(0,\hdots,0,1)^t$.
This embedding of groups induces an embedding of $\kk$-algebras
$\kk[GL_n] \hookrightarrow \kk[GL_{n+1}]$.  Unless explicitly
mentioned otherwise, all of the embeddings in this section will be
induced from this embedding of algebras.  For notational simplicity,
we will write $A_n = \kk[GL_n]$.

Let $U_{n,n+1}\subseteq GL_{n+1}$ denote the subgroup of upper
triangular invertible matrices whose upper left $n \times n$ block is the
$n \times n$ identity matrix and whose last column is arbitrary. For
example, $U_{2,3}$ is the subgroup of matrices of the form
\[
  \left(
  \begin{array}{ccc}
  1&0&a\\
  0&1&b\\
  0&0&c\\
  \end{array}\right), \ \ \ c\neq 0.
\]
Note that $U_{n,n+1}$ is a normal subgroup of $P_{n,n+1} = GL_n
\cdot U_{n,n+1}$.  Define an idempotent $v_{n+1}\in A_{n+1}$
by
\[
  v_{n+1} = \frac{1}{|U_{n,n+1}|} \sum_{u\in U_{n,n+1}} u.
\]
Note that the element $v_{n+1}$ satisfies
\begin{itemize}
\item $v_{n+1}^2 = v_{n+1}$,
\item for all $x\in \kk[U_{n,n+1}]$, $xv_{n+1} = v_{n+1}x = v_{n+1}$, and
\item for all $y \in A_n$, $yv_{n+1} = v_{n+1}y$.
\end{itemize}

We may consider $A_{n+1}v_{n+1}$ as an
$(A_{n+1},A_n)$-bimodule.  Similarly we may consider
$v_{n+1}A_{n+1}$ as a $(A_n,A_{n+1})$-bimodule. We will use
notation for these bimodules and their tensor products similar to the
notation used for bimodules over Hecke algebras.  In particular,
\begin{itemize}
  \item $(n)$ denotes $A_n$, considered as an $(A_n,A_n)$-bimodule,
  \item $(n)_{n-1}$ denotes $A_nv_n$, considered as a
  $(A_n,A_{n-1})$-bimodule,
  \item $_{n-1}(n)$ denotes $v_nA_n$, considered as an $(A_{n-1},A_n)$-bimodule,
  \item $_{n-1}(n)_{n-1}$ denotes $v_nA_nv_n$, considered as an $(A_{n-1},A_{n-1})$ bimodule, and
  \item tensor products of these bimodules are denoted by juxtaposition.
\end{itemize}

The bimodule $(n)_{n-1} = A_nv_n$ gives rise to an induction functor
\[
  \ind: A_{n-1}\mbox{-mod} \to A_n\mbox{-mod}, \ \ M \mapsto
  A_nv_n\otimes_{A_{n-1}} M.
\]
Similarly, the bimodule $_{n-1}(n) = v_nA_n$ gives rise to a restriction functor
\[
  \res: A_n\mbox{-mod}  \to A_{n-1}\mbox{-mod}, \ \ N \mapsto
  v_nA_n\otimes_{A_n} N.
\]
More generally, an $(A_n,A_m)$-bimodule $X$ gives rise to a functor from the category of left
$A_m$-modules to the category of left $A_n$-modules: this functor takes an $A_m$-module $M$ to the $A_n$-module $X\otimes_{A_m} M$.  Composition of functors corresponds to tensor product of bimodules, and natural transformations correspond to bimodule maps.  Therefore, in order to define natural transformations of compositions of the functors $\ind$ and $\res$, we will define bimodule maps between tensor products of the bimodules $(n)_{n-1}$ and $_{n-1}(n)$.

\subsection{Adjunctions} \label{sec:GL-adjunctions}

\begin{enumerate}
\item
Let $\rcap$ denote the bimodule map
\[
  \rcap: (n+1)_n(n+1) \to (n+1),\quad
  x v_{n+1}\otimes v_{n+1}y \mapsto xv_{n+1}y.
\]

\item
Let $\rcup$ denote the bimodule map
\[
  \rcup : (n) \hookrightarrow {_n}(n+1)_n,\quad z \mapsto v_{n+1}zv_{n+1}.
\]

\item
Let $\lcap$ denote the bimodule map
\[
  \lcap: {_n}(n+1)_n \to (n)
\]
given as follows.  For a group element $g\in GL_{n+1}$, we set
\[
  \lcap(v_{n+1}gv_{n+1}) =
  \begin{cases}
  v_{n+1}gv_{n+1}, & g\in P_{n,n+1}, \\
  0, & g\notin P_{n,n+1}.
  \end{cases}
\]
Extending by $\kk$-linearity, this defines a $(A_n,A_n)$ bimodule
map
\[
  v_{n+1}A_{n+1}v_{n+1} \to v_{n+1}\kk[P_{n,n+1}]v_{n+1} \cong A_n,
\]
as desired.  The isomorphism of the last line follows immediately from the properties of the idempotent $v_{n+1}$ listed above.

\item
Let $\lcup$ denote the bimodule map
\[
  \lcup: (n+1) \to (n+1)_n(n+1)
\]
determined as follows.  Let $GL_{n+1} = \coprod_{i=1}^s
P_{n,n+1}g_i$ be a decomposition of $GL_{n+1}$ into left $P_{n,n+1}$
cosets.  The element
\[
  \sum_{i=1}^s g_i^{-1} v_{n+1} \otimes v_{n+1}g_i \in (n+1)_n(n+1)
\]
does not depend on the choice of coset representatives
$\{g_i\}_{i=1}^s$.  Moreover, this element is a Casimir element, so
that
\[
  a \left( \sum_{i=1}^s g_i^{-1} v_{n+1} \otimes v_{n+1}g_i \right) =
  \left( \sum_{i=1}^s g_i^{-1} v_{n+1} \otimes v_{n+1}g_i \right) a
\]
for all $a\in A_{n+1}$.  Thus there is a natural bimodule map $(n+1)
\to (n+1)_n(n+1)$ determined by
\[
  1_{n+1} \mapsto \sum_{i=1}^s g_i^{-1} v_{n+1} \otimes v_{n+1}g_i.
\]
\end{enumerate}

\begin{proposition}\label{prop:cyclic-GL}
The adjunction maps $\lcap$, $\lcup$, $\rcap$, $\rcup$ above make
$(\res,\ind)$ into a cyclic, biadjoint pair.
\end{proposition}

\begin{proof}
We prove this in Section~\ref{sec:parabolic}, where we work in the more general setting of parabolic induction and restriction for arbitrary finite groups (see Corollary~\ref{cor:GL-cyclicity}).
\end{proof}

\subsection{Crossings} \label{sec:GL-crossings}

Let $s_n \in GL_{n+1}$ denote the symmetric group element
\[
  s_n= \left(
  \begin{array}{ccc}
  I_{n-1}&0&0\\
  0&0&1\\
  0&1&0\\
  \end{array}\right),
\]
where $I_{n-1}$ is the $(n-1) \times (n-1)$ identity matrix.  Let $t_n =
qv_{n+1}s_nv_{n+1} \in A_{n+1}$.

Let $(n+2)_n$ denote the bimodule $(n+2)_n = (n+2)_{n+1}(n+1)_n =
A_{n+2}v_{n+2}v_{n+1}$.  We define a bimodule map
\[
  \begin{tikzpicture}[>=stealth,baseline=10pt]
    \draw[->] (0,0) to (1,1);
    \draw[->] (1,0) to (0,1);
    \draw (1.1,.5) node {$n$};
  \end{tikzpicture}
  \quad : \quad (n+2)_n \to (n+2)_n,\quad zv_{n+2}v_{n+1} \mapsto
  z t_{n+1}v_{n+2}v_{n+1}.
\]
Since $A_n$ commutes with $v_{n+2}$, $v_{n+1}$ and $s_{n+1}$,
it follows that the above is a well-defined map of $(A_{n+2},A_n)$-bimodules.  By Proposition~\ref{prop:cyclic-GL}, this 2-morphism is cyclic.  Therefore, any two isotopic diagrams
involving this crossing as well as cups and caps are equal.  We
define left, right, and downward crossings to be equal to any
composition of cups and caps and an upward crossing yielding a
bimodule map between the appropriate bimodules.  By the above, any
two such definitions are equivalent.  For instance, we can define
the left, right and downward crossings as in \eqref{eq:leftcross-def}--\eqref{eq:downcross-def}.

\begin{lemma}
We have
\begin{align*}
  \begin{tikzpicture}[>=stealth,baseline=12pt]
    \draw[<-] (0,0) to (1,1);
    \draw[<-] (1,0) to (0,1);
    \draw (-.1,.5) node {$n$};
  \end{tikzpicture}
  \quad &: \quad _n(n+2) \to {_n}(n+2),\quad v_{n+1}v_{n+2}z \mapsto
 v_{n+1}v_{n+2} t_{n+1} z, \\
  \begin{tikzpicture}[>=stealth,baseline=12pt]
    \draw[->] (0,0) to (1,1);
    \draw[<-] (1,0) to (0,1);
    \draw (1.1,.5) node {$n$};
  \end{tikzpicture}
  \quad &: \quad (n)_{n-1}(n) \hookrightarrow {_n(n+1)_n},\quad wv_n \otimes v_nz
  \mapsto v_{n+1}w t_n zv_{n+1}, \\
  \begin{tikzpicture}[>=stealth,baseline=12pt]
    \draw[<-] (0,0) to (1,1);
    \draw[->] (1,0) to (0,1);
    \draw (1.1,.5) node {$n$};
  \end{tikzpicture}
  \quad &: \quad {_n(n+1)_n} \twoheadrightarrow{(n)_{n-1}(n)},\quad
  v_{n+1} \mapsto 0,\quad v_{n+1} t_n v_{n+1} \mapsto q v_n \otimes v_n.
\end{align*}
\end{lemma}

\begin{proof}
All three statements are straightforward computations.  In the third computation, it is useful to note that $_n(n+1)_n$ has a bimodule basis consisting
of $1_n$ and $t_n$; thus any bimodule map from $_n(n+1)_n$ is
completely determined by the image of these two elements.  We omit the details of these three computations.
\end{proof}

\subsection{Another representation of the graphical category} \label{sec:GL-Fock-rep}

The following proposition should be compared to Proposition \ref{prop:relations}.

\begin{proposition}\label{prop:action2}
The relations \eqref{eq:local-relation-basic-Hecke}--\eqref{eq:cc-circle-and-left-curl} hold when these diagrams are interpreted as maps of bimodules.
\end{proposition}

\begin{proof}
This follows by direct computation from the definitions and results of Sections~\ref{sec:reps-GL}--\ref{sec:GL-crossings}.
\end{proof}

Some comments about the above proposition are in order.  The
relations involving only upward pointing strands follow from the
classical work of Iwahori \cite{Iwa64}. The analogues of our upward pointing
strands were introduced in Chuang-Rouquier \cite{ChuRou08} in the context of $\mathfrak{sl}_2$
categorifications in the modular representation theory of $GL_n$.
The relations amongst upward pointing braid-like diagrams --- that is, diagrams with no local minima or maxima --- are essentially contained in \cite{JoyStr95}, which studies the braided monoidal structure on the category of characteristic zero representations of all $GL_n$.

Let $B_n\subseteq GL_n$ denote the Borel subgroup of upper triangular matrices.
The relations involving only upward pointing strands imply that
there is a canonical morphism from the Hecke algebra $H_n$ to
$\mbox{End}_{A_n}(\mbox{Ind}_{\kk[B_n]}^{A_n} \mathbf{1})$, the endomorphism
algebra of the induction to $A_n$ of the trivial $\kk[B_n]$-module.

\begin{definition} \label{def:GL-2-cat}
Let $\fB$ be the 2-category defined as follows.
\begin{itemize}
  \item $\Ob \fB = \N \cup \{\nabla\}$.

  \item The 1-morphisms from $n$ to $m$ for $n,m \in \N$, are functors from $A_n$-mod to $A_m$-mod that are direct summands of compositions of induction and restriction functors.  The only 1-morphism from $\nabla$ to $\nabla$ is the identity. For $n \in \N$, there are no 1-morphisms from $n$ to $\nabla$ or from $\nabla$ to $n$.

  \item The 2-morphisms are natural transformations of functors.
\end{itemize}
\end{definition}
It follows from Proposition~\ref{prop:action2} that we can define a 2-functor $\bB : \fH(q) \to \fB$ as in Definition~\ref{def:2-rep-Hecke} (with $H_n$ replaced by $A_n$).

The representations of $\fH(q)$ given here and in Section~\ref{sec:BFS-cat} are closely related.  Thus our categorical Heisenberg actions provide another perspective on the appearance of the Hecke algebra in the
representation theory of $GL_n$.  Let us explain the relationship between the two $\fH(q)$ representations in a bit more detail.  Let $\mathcal{C}_n \subseteq A_n\mbox{-mod}$
denote the full subcategory of $A_n\mbox{-mod}$ consisting of direct summands of modules of the form $\ind^r M$, $r \in \N$, $M$ an $A_0$-module.  Equivalently, the objects of $\mathcal{C}_n$ are the \emph{unipotent} $A_n$-modules, that is, the modules which occur as a direct summand of the induction of a trivial $\kk[B_n]$-module to $A_n$.  Let $\mathfrak{C}$ be the 2-category defined as in Definition~\ref{def:GL-2-cat}, but with $A_n$-mod replaced by $\mathcal{C}_n$.  It follows that the category $\mathfrak{C}$ is a representation of $\fH(q)$.    Define the idempotent
\[
  b_n = \frac{1}{|B_n|} \sum_{b\in  B_n} b = v_n v_{n-1} \cdots v_1
\]
in the Borel $B_n$.  The subalgebra $b_nA_nb_n\subseteq A_n$ is the sum of the \emph{unipotent blocks} of $A_n$.  Note that any $M \in \mathcal{C}_n$ contains a vector fixed by $\kk[B_n]$.  Thus, the inclusion of $M$ into a free module $A_n^{\oplus k}$, $k \in \N$, factors through $(A_nb_n)^{\oplus k}$.  Since $\Hom_{A_n}(A_nb_n,A_nb_n) = b_nA_nb_n$, we have a functor
\[
  \Hom_{A_n}(A_nb_n, -): \mathcal{C}_n \to b_nA_nb_n\mbox{-mod}.
\]
It is straightforward to check that this functor is an equivalence of categories, with the inverse functor given by tensoring with the $(A_n,b_nA_nb_n)$-bimodule $A_nb_n$.

Iwahori's theorem \cite{Iwa64} implies that for each $n$ there is an isomorphism
\[
  b_nA_nb_n \cong H_n.
\]
These isomorphisms for all $n$ are compatible with the embeddings
$b_nA_nb_n \subseteq b_{n+1}A_{n+1}b_{n+1}$ and $H_n\subseteq H_{n+1}$, and thus with the
actions of the 1-morphisms
$Q_+$ and $Q_-$.  It is then easy to see that the isomorphisms $b_nA_nb_n \cong H_n$ are compatible with the defining cup, cap, and crossing 2-morphisms.  Thus we have the following proposition.
\begin{proposition}\label{prop:equiv}
The equivalence of categories
\[
  \textstyle
  \bigoplus_n \Hom_{A_n}(A_nb_n, -): \bigoplus_n \mathcal{C}_n \to \bigoplus_n H_n\text{-$\mathrm{mod}$}
\]
induces an equivalence $\mathfrak{C} \to \fA$ of $\fH(q)$ representations.  In other words,
\[
  \xymatrix{ \fH(q) \ar[r] \ar[dr] & \mathfrak{C} \ar[d] \\
  & \fA}
\]
is a commutative diagram (up to isomorphism) with the functor $\mathfrak{C} \to \fA$ being an equivalence of 2-categories.
\end{proposition}

\begin{remark}
Proposition~\ref{prop:equiv} implies that the functor $\fH(q) \to \mathfrak{C}$ is another categorification of bosonic Fock space.  The functor $\bB$ to the entire category $\fB$ is a categorification of an infinite sum of bosonic Fock spaces.
\end{remark}

%
\section{Parabolic induction and restriction for finite groups} \label{sec:parabolic}
%

The purpose of this section is to describe the cyclic biadjointness of parabolic induction and restriction functors which show up naturally in the representation theory of finite groups.  We recommend the survey \cite{Kho10b} for an introduction to cyclic biadjoint functors.

\subsection{Bimodules from the group algebra of a semi-direct product}\label{sec:pardef}

Recall that $\kk$ is a field of characteristic zero.  Let $L$ and $U$ be subgroups of a finite group $G$ such that
\begin{itemize}
  \item $L$ normalizes $U$, i.e. $U$ is a normal subgroup of $P = LU$,

  \item $L\cap U = \{1\}$.
\end{itemize}
Let $\theta : U \to \C^*$ be a multiplicative character of $U$ normalized by $L$, so that
\[
  \theta (m u m^{-1}) = \theta (u) \quad \text{for all }  m \in L, \ u\in U.
\]
Let $v_\theta \in \kk[P]$ be defined by
\[
  v_\theta = \frac{1}{|U|} \sum_{u\in U} \theta(u)^{-1} u.
\]

\begin{lemma} \label{lem:v-theta}
The element $v_\theta$ satisfies the following:
\begin{itemize}
  \item $v_\theta v_\theta = v_\theta$,
  \item for $u\in U$, $uv_\theta = v_\theta u = \theta(u) v_\theta$,
  \item for $m\in L$, $mv_\theta = v_\theta m$.
\end{itemize}
\end{lemma}
\begin{proof}
All three facts follow from direct computations which are omitted.
\end{proof}

The space $\kk[P]\vt$ has the natural structure of a $(\kk[P],\kk[P])$-bimodule, and, by restriction, the structure of a $(\kk[L],\kk[L])$-bimodule.

\begin{corollary}
$\kk[P]\vt \cong \kk[L]$ as a $(\kk[L],\kk[L])$-bimodule.
\end{corollary}

\begin{proof}
It follows from Lemma~\ref{lem:v-theta} that the map
\[
  \kk[L] \to \kk[P]\vt, \ \ 1\mapsto \vt,
\]
is a $(\kk[L],\kk[L])$-bimodule isomorphism.
\end{proof}

\subsection{Parabolic induction and restriction functors}

We define functors
\[
  \ind_{U,\theta} : \kk[L]\mbox{-mod} \to \kk[G]\mbox{-mod}
\]
and
\[
  \res_{U,\theta} : \kk[G]\mbox{-mod} \to \kk[L]\mbox{-mod}
\]
depending on the subgroup $U$ and the character $\theta$, as follows.

\begin{definition}\label{def:induction}
The \emph{parabolic induction} functor
\[
  \ind_{U,\theta} : \kk[L]\mbox{-mod} \to \kk[G]\mbox{-mod}
\]
is defined by tensoring with the $(\kk[G],\kk[L])$-bimodule
$\kk[G]\otimes_{\kk[P]} \kk[P]\vt \cong \kk[G]\vt$:
\[
  \ind_{U,\theta} (M) := \big( \kk[G]\otimes_{\kk[P]} \kk[P]\vt \big) \otimes_{\kk[L]} M.
\]
\end{definition}

\begin{definition}\label{def:restriction}
The \emph{parabolic restriction} functor
\[
  \res_{U,\theta} : \kk[G]\mbox{-mod} \to \kk[L]\mbox{-mod}
\]
is defined by tensoring with the $(\kk[L],\kk[G])$-bimodule $\kk[P]\vt \otimes_{\kk[P]}\kk[G]$:
\[
  \res_{U,\theta} (M') = \big( \kk[P]\vt \otimes_{\kk[P]} \kk[G] \big) \otimes_{\kk[G]} M'.
\]
\end{definition}

Note that the functors $\ind_{U,\theta}$, $\res_{U,\theta}$, depend on both the subgroup $U$ and on the character $\theta$.  Some basic properties of the functors $\ind_{U,\theta}$ and $\res_{U,\theta}$ are listed below.  We also refer the reader to \cite{Zel81} for further details:
\begin{enumerate}
  \item $\ind_{U,\theta}$ and $\res_{U,\theta}$ are additive.
  \item Let $N$ and $V$ be subgroups of $L$ and let $\theta'$ be a character of $V$ such that the functors $\ind_{V,\theta'}: \kk[N]\mbox{-mod} \rightarrow \kk[L]\mbox{-mod}$ and $\res_{V,\theta'} : \kk[L]\mbox{-mod} \rightarrow \kk[N]\mbox{-mod}$ are defined.  Define a character $\theta^0$ of $U^0 = UV$ by
    \[
      \theta^0(uv) = \theta(u)\theta'(v), \ u\in U, \ v\in V.
    \]
    Then there are isomorphisms of functors
    \[
      \ind_{U,\theta}\circ \ind_{V,\theta'} \cong \ind_{U^0,\theta^0}, \ \
      \res_{V,\theta'}\circ \res_{U,\theta} \cong \res_{U^0,\theta^0}.
    \]
  \item  Let $H$ be another finite group.  For $M \in \kk[H]\mbox{-mod}$, there are functors
    \[
      T^l_M: \kk[L]\mbox{-mod}\rightarrow \kk[L\times H]\mbox{-mod},\ \ T_M^l(M') = M\otimes_{\kk} M'
    \]
    and
    \[
      T^r_M : \kk[L]\mbox{-mod} \rightarrow \kk[H\times L]\mbox{-mod},\ \ T_M^r(M') = M' \otimes_\kk M
    \]
    given by outer tensor product with $M$ on the left and on the right.
    The parabolic induction functor
    \[
      \ind_{U,\theta} : \kk[L]\mbox{-mod} \to \kk[G]\mbox{-mod}
    \]
    induces a parabolic induction functor
    \[
      \ind_{U\times 1,\theta\times 1}:  \kk[L\times H]\mbox{-mod} \to \kk[G\times H]\mbox{-mod}.
    \]
    Then there are isomorphisms of functors
    \[
      T^r_M\circ \ind_{U,\theta} \cong \ind_{U\times 1,\theta\times 1} \circ T^r_M.
    \]
    In other words, parabolic induction commutes with $T^r_M$ (and also with $T^l_M$).  Similarly,
    parabolic restriction commutes with $T^r_M$ and $T^l_M$.
    \end{enumerate}

\subsection{Cyclic biadjointness for parabolic induction and restriction}

We now define natural transformations $\rcap,\rcup,\lcap,\lcup$ which will serve as
biadjunction morphisms.
\begin{enumerate}
  \item Let
    \[
      \rcap: \ind_{U,\theta} \circ \res_{U,\theta} \to \id_{\kk[G]},
    \]
    where $\id_{\kk[G]}$ is the identity functor on the category of $\kk[G]$-modules,
    be the natural transformation given as follows.  The left hand side is given as a functor by tensoring with the $(\kk[G],\kk[G])$-bimodule
    \[
     \kk[G]\otimes_{\kk[P]} \kk[P]\vt \otimes_{k[L]} \kk[P]\vt \otimes_{k[P]} \kk[G]
     \cong \kk[G]\otimes_{\kk[P]} \kk[P]\vt \otimes_{k[P]} \kk[G].
    \]
    Thus the bimodule map defined on a $\kk$-basis by
    \[
     g\otimes b\vt \otimes h \mapsto gb\vt h \in \kk[G]
    \]
    gives a natural transformation to $\id_{\kk[G]}$.

  \item Let
    \[
     \rcup : \id_{\kk[L]} \to \res_{U,\theta} \circ \ind_{U,\theta}
    \]
    be the natural transformation given by the bimodule map
    \[
     \kk[L] \to \kk[P]\vt \otimes_{\kk[P]} \kk[G]\otimes_{\kk[P]} \kk[P]\vt,
    \]
    determined by
    \[
      1 \mapsto \vt\otimes 1 \otimes \vt.
    \]

  \item Let
    \[
     \lcap: \res_{U,\theta} \circ \ind_{U,\theta} \to \id_{\kk[L]}
    \]
    be the natural transformation given by the bimodule map
    \[
      \kk[P]\vt \otimes_{\kk[P]} \kk[G]\otimes_{\kk[P]} \kk[P]\vt \to \kk[L],
    \]
    defined as follows.  Let $\Proj_P: \kk[G] \to \kk[P]$ be the $\kk$-linear map defined, for $g\in G$, by
    \[
      \Proj_P(g) = \begin{cases}
      g, \ g\in P,\\
      0, \text{ otherwise.}
      \end{cases}
    \]
    Then define $\lcap$ as the composition
    \begin{gather*}
      \lcap: \kk[P]\vt \otimes_{\kk[P]} \kk[G] \otimes_{\kk[P]} \kk[P]\vt \xrightarrow{1\otimes \Proj_P \otimes 1} \kk[P]\vt \otimes_{\kk[P]} \kk[P] \otimes_{\kk[P]} \kk[P]\vt \\
      \xrightarrow{\mu} \kk[P]\vt\simeq \kk[L],
    \end{gather*}
    where $\mu$ is the multiplication map $b\vt\otimes p \otimes b'\vt \mapsto b\vt p b'\vt = bpb'\vt$.

  \item
    Let
    \[
      \lcup :  \id_{\kk[G]}\to \ind_{U,\theta} \circ \res_{U,\theta}
    \]
    be the natural transformation given by the bimodule map
    \[
      \lcup : \kk[G] \to  \kk[G]\otimes_{\kk[P]} \kk[P]\vt \otimes_{\kk[P]} \kk[G]
    \]
    defined as follows.  Let $G = \coprod_{i=1}^m Pg_i$ be a decomposition of $G$ into left $P$-cosets. The element
    \[
      \sum_{i=1}^m g_i^{-1} \otimes \vt \otimes g_i \in \kk[G]\otimes_{\kk[P]} \kk[P]\vt \otimes_{\kk[P]} \kk[G]
    \]
    does not depend on the choice of representatives $\{g_i\}_{i=1}^m$.  Indeed, if $g_i' = p_ig_i$ for $1 \le i \le m$, then
    \[
      \sum_{i=1}^m (g'_i)^{-1} \otimes \vt \otimes g'_i = \sum_{i=1}^m g_i^{-1}p_i^{-1} \otimes \vt \otimes p_ig_i=\sum_{i=1}^m g_i^{-1} \otimes \vt \otimes g_i
    \]
    (since elements of $P$ move across the tensor products and past $\vt$). Thus we may define a bimodule map $\lcup$ by declaring that
    \[
      1 \mapsto \sum_{i=1}^m g_i^{-1} \otimes \vt \otimes g_i.
    \]
    Then, for $g\in G$,
    \[
      \lcup(g) = \sum_{i=1}^m g_i^{-1} \otimes \vt \otimes g_ig = \sum_{i=1}^m gg_i^{-1} \otimes \vt \otimes g_i.
    \]
    The last equality is needed to ensure that we have a bimodule map, and is proven as follows.  Suppose that $g_ig = p_ig_i'$ for some $p_i\in P$ and some bijection  $i\mapsto i'$ of $\{1,\hdots, m\}$.  Then
    \begin{gather*}
      \sum_{i=1}^m g_i^{-1} \otimes \vt \otimes g_ig = \sum_{i=1}^m g_i^{-1} \otimes \vt \otimes p_ig'_i = \sum_{i=1}^m g_i^{-1}p_i \otimes \vt \otimes g'_i \\
      =\sum_{i=1}^m g(g'_i)^{-1} \otimes \vt \otimes g'_i
      =\sum_{i=1}^m gg_i^{-1} \otimes \vt \otimes g_i,
    \end{gather*}
    as desired.
\end{enumerate}

\begin{proposition}
The natural transformations $\rcap,\rcup,\lcap,\lcup$ make $\ind_{U,\theta}$ and
$\res_{U,\theta}$ into a biadjoint pair.
\end{proposition}

\begin{proof}
We must show four adjunction relations.
\begin{enumerate}
  \item $(\rcap \otimes \id_{\ind_{U,\theta}})(\id_{\ind_{U,\theta}}\otimes \rcup) = \id_{\ind_{U,\theta}}.$\\
    To show this we must show that the composition of bimodule maps
    \begin{gather*}
      \kk[G]\otimes_{\kk[P]}\kk[P]\vt \xrightarrow{\id\otimes\rcup} \kk[G]\otimes_{\kk[P]}\kk[P]\vt\otimes_{\kk[L]}\kk[P]\vt\otimes_{\kk[P]} \kk[G] \otimes_{\kk[P]} \kk[P]\vt \\
      \xrightarrow{\rcap \otimes \id} \kk[G]\otimes_{\kk[P]} \kk[P]\vt
    \end{gather*}
    is the identity.  We check that for $g\in G$ and $b\in P$,
    \[
      g\otimes b\vt \mapsto (g\otimes b\vt)\otimes(\vt\otimes 1 \otimes \vt) \mapsto gb\vt \otimes \vt = g\otimes b\vt,
    \]
    as desired.

    \item $(\id_{\res_{U,\theta}} \otimes \rcap)(\rcup \otimes \id_{\res_{U,\theta}}) = \id_{\res_{U,\theta}}$.\\

    The composition of bimodule maps
    \begin{gather*}
      \kk[P]\vt \otimes_{\kk[P]} \kk[G] \xrightarrow{\rcup\otimes \id} \kk[P]\vt \otimes_{\kk[P]} \kk[G] \otimes_{\kk[P]}
      \kk[P]v_\theta\otimes_{\kk[L]} \kk[P]\vt \otimes_{\kk[P]} \kk[G] \\
      \xrightarrow{\id\otimes \rcap} \kk[P]\vt \otimes_{\kk[P]} \kk[G]
    \end{gather*}
    is given, for $b\in P$ and $g\in G$, by
    \[
      b\vt\otimes g \mapsto b\vt\otimes 1\otimes \vt \otimes g \mapsto b\vt\otimes \vt g = b\vt\otimes g,
    \]
    as desired.

    \item $(\lcap \otimes \id_{\res_{U,\theta}})(\id_{\res_{U,\theta}}\otimes \lcup) = \id_{\res_{U,\theta}}$.\\

      The composition of bimodule maps
      \begin{gather*}
        \kk[P]\vt \otimes_{\kk[P]}\kk[G] \xrightarrow{\id \otimes \lcup} \kk[P]\vt \otimes_{\kk[P]}\kk[G] \otimes_{\kk[P]} \kk[P]\vt \otimes_{\kk[P]} \kk[G] \\
        \xrightarrow{\lcap\otimes \id} \kk[L]\otimes_{\kk[L]}\kk[P]\vt \otimes_{k[P]} \kk[G] \cong \kk[P]\vt \otimes_{k[P]} \kk[G]
      \end{gather*}
      is given, for $b\in P$, $g\in G$, by
      \[
        b\vt\otimes g \mapsto \sum_{i=1}^m b\vt \otimes gg_i^{-1}\otimes \vt \otimes g_i \mapsto b\vt gg_j^{-1}\vt \otimes g_j = b\vt \otimes g,
      \]
      where $j\in \{1,\hdots, m\}$ is the unique index such that $gg_j^{-1}\in P$ (the other summands are killed by $\lcap$).  Thus the composition is the identity, as desired.\\

  \item $(\id_{\ind_{U,\theta}} \otimes \lcap)(\lcup \otimes \id_{\ind_{U,\theta}}) = \id_{\ind_{U,\theta}}$.\\

    The composition of bimodule maps
    \begin{gather*}
      \kk[G]\otimes_{\kk[P]} \kk[P]\vt \xrightarrow{\lcup\otimes \id} \kk[G]\otimes_{\kk[P]} \kk[P]\vt \otimes_{\kk[P]} \otimes \kk[G]\otimes_{\kk[P]} \kk[P]\vt \\
      \xrightarrow{\id\otimes\lcap}  \kk[G]\otimes_{\kk[P]} \kk[P]\vt
    \end{gather*}
    is given, for $g\in G$ and $b\in P$, by
    \[
      g\otimes b\vt \mapsto \sum_{i=1}^m g_i^{-1} \otimes \vt \otimes g_i g \otimes b\vt \mapsto g_j^{-1} \otimes \vt g_j g b \vt = g\otimes b\vt,
    \]
    where $j\in \{1,\hdots, m\}$ is the unique index such that $g_j g \in P$. Thus the composition is the identity, as desired.
\end{enumerate}

This completes the proof of biadjointness.
\end{proof}

\subsection{Bimodule endomorphisms}

We begin by discussing the endomorphism algebra of the functors $\res_{U,\theta}$ and $\ind_{U,\theta}$.
Consider the $(\kk[L],\kk[G])$-bimodule $\kk[P]\vt\otimes_{\kk[P]} \kk[G]$.
Let $(\kk[P]\vt\otimes_{\kk[P]} \kk[G])^P$ denote the subspace of $P$-invariants under the conjugation action of $P$:
\[
  (\kk[P]\vt\otimes_{\kk[P]} \kk[G])^P
  = \{ x\in \kk[P]\vt\otimes_{\kk[P]} \kk[G] \mid mx=xm \text{ for all } m\in P\}.
\]
Any $x\in (\kk[P]\vt\otimes_{\kk[P]} \kk[G])^P$ gives rise to a $(\kk[L],\kk[G])$-bimodule map
\[
  {'x}:\kk[P]\vt\otimes_{\kk[P]} \kk[G]\to \kk[P]\vt\otimes_{\kk[P]} \kk[G],
\]
determined by setting
\[
  {'x}(\vt\otimes 1) = x,
\]
so that, for $b\in L$, $g\in G$,
\[
  {'x}(b\vt \otimes g) = bxg.
\]
Therefore, any $x\in (\kk[P]\vt\otimes_{\kk[P]} \kk[G])^P$ may be considered as a natural transformation
\[
  \res_{U,\theta} \to \res_{U,\theta}.
\]
We may define a multiplication on $(\kk[P]\vt\otimes_{\kk[P]} \kk[G])^P$ by
\[
  \left(\sum_{g\in G} \lambda_g (b_g\vt\otimes g)\right)\left(\sum_{h\in G} \mu_h (d_h\vt\otimes h)\right) = \sum_{(g,h)\in G\times G} \lambda_g\mu_h (d_hb_g\vt \otimes  gh).
\]
(Note the order of the terms in the product on the right.)  It follows from Lemma \ref{lem:natres} that this multiplication makes $(\kk[P]\vt\otimes_{\kk[P]} \kk[G])^P$ into a $\kk$-algebra, with unit element $\vt\otimes 1$.

\begin{lemma}\label{lem:natres}
The assignment $a \mapsto {'a}$ induces an isomorphism of $\kk$-algebras
\[
  (\kk[P]\vt\otimes_{\kk[P]} \kk[G])^P \simeq \Hom(\res_{U,\theta},\res_{U,\theta}),
\]
where $\Hom(\res_{U,\theta},\res_{U,\theta})$ denotes the $\kk$-algebra of natural transformations of the functor $\res_{U,\theta}$.
\end{lemma}

\begin{proof}
The $\kk$-algebra $\Hom(\res_{U,\theta},\res_{U,\theta})$ is isomorphic to the space of $(\kk[L],\kk[G])$-bimodule maps
\[
  \{f: \kk[P]\vt\otimes_{\kk[P]} \kk[G] \rightarrow \kk[P]\vt\otimes_{\kk[P]} \kk[G]\}.
\]
First we show that the map $a \mapsto {'a}$ intertwines the multiplication on the two sides.  To see this, let
$a = \sum_{g\in G} \lambda_g (b_g\vt\otimes g)$ and $c = \sum_{h\in G} \mu_h (d_h\vt\otimes h)$ be elements of $(\kk[P]\vt\otimes_{\kk[P]} \kk[G])^P$.  Then
\[
  {'(ac)}(\vt\otimes 1) = \sum_{(g,h)\in G\times G} \lambda_g\mu_h (d_hb_g\otimes gh).
\]
We also have
\[
  ({'a} {'c}) (\vt\otimes 1) = {'a}\left(\sum_{h\in G} \mu_h(d_h\vt\otimes h)\right) =
  \sum_{h\in G} \mu_h d_h {'a}(\vt\otimes 1) h =
  \sum_{(g,h)\in G\times G} \lambda_g\mu_h (d_hb_g \otimes gh),
\]
so that
\[
  {'(ac)}(\vt\otimes 1) = ({'a} {'c}) (\vt\otimes 1),\text{ and therefore } {'a}{'c} = {'(ac)}.
\]
To see that the map $a\mapsto {'a}$ is injective,
let $a,c\in (\kk[P]\vt\otimes_{\kk[P]} \kk[G])^P$ such that ${'a} = {'c}$ as bimodule maps.  Then
\[
  a=a(\vt\otimes 1) = {'a}(\vt\otimes 1) = {'c}(\vt\otimes 1) = c(\vt\otimes1)= c.
\]
Thus the map is injective.  It also now follows that $(\kk[P]\vt\otimes_{\kk[P]} \kk[G])^P$ is an algebra, since it is a subalgebra of $\Hom(\res_{U,\theta},\res_{U,\theta})$.

To see surjectivity, suppose that
\[
  f : \kk[P]\vt\otimes_{\kk[P]} \kk[G]\to \kk[P]\vt\otimes_{\kk[P]} \kk[G]
\]
is a bimodule map, and set $a = f(\vt\otimes 1)$.  Then is it easy to check that
\[
  a\in (\kk[P]\vt\otimes_{\kk[P]} \kk[G])^P
\]
and that $f = {'a}$ as bimodule maps.
\end{proof}

Thus the $\kk$-algebra $(\kk[P]\vt\otimes_{\kk[P]} \kk[G])^P$ is isomorphic to the algebra of natural transformations of $\res_{U,\theta}$.  Similarly, we may consider $( \kk[G]\otimes_{\kk[P]}\kk[P]\vt)^P$ as an algebra, with multiplication
\[
  \left(\sum_{g\in G} \lambda_g (g\otimes b_g\vt)\right)\left(\sum_{h\in G} \mu_h (h\otimes d_h\vt)\right) = \sum_{(g,h)\in G\times G} \lambda_g\mu_h (hg \otimes  b_gd_h\vt).
\]

\begin{lemma}
The map
\[
  \tau: (\kk[P]\vt\otimes_{\kk[P]} \kk[G])^P_\mathrm{op} \to (\kk[G]\otimes_{\kk[P]} \kk[P]\vt)^P,
\]
\[
  \sum_{g\in G} \lambda_g (b_g\vt\otimes g) \mapsto \sum_{g\in G} \lambda_g (g\otimes b_g\vt)
\]
is an isomorphism of algebras.
\end{lemma}

\begin{proof}
The proof of this statement is a straightforward check, which is omitted.  Note that the subscript ``op" denotes the opposite algebra.
\end{proof}

Thus, in complete analogy with Lemma \ref{lem:natres}, we have the following.

\begin{lemma}\label{lem:opp}
The algebra $( \kk[G]\otimes_{\kk[P]}\kk[P]\vt)^P$ is isomorphic to the algebra of natural transformations of the functor $\ind_{U,\theta}$, with the element $a \in (\kk[G]\otimes_{\kk[P]}\kk[P]\vt)^P$ giving a bimodule map via $a'(1\otimes \vt)= a. $
\end{lemma}

\begin{proof}
The proof is essentially the same as the proof of Lemma \ref{lem:natres}.
\end{proof}

\subsection{Cyclicity}\label{sec:cyclicity}

Given an element $a\in  (\kk[P]\vt\otimes_{\kk[P]} \kk[G])^P$, we may also consider $a$ as an element of the opposite algebra
\[
	a\in  (\kk[P]\vt\otimes_{\kk[P]} \kk[G])^P_\mathrm{op}
	\xrightarrow{\tau}
	(\kk[G]\otimes_{\kk[P]} \kk[P]\vt)^P.
\]
Thus the element $a$ of the opposite algebra is identified with the element $\tau(a)$ in
the algebra $(\kk[G]\otimes_{\kk[P]} \kk[P]\vt)^P$.
From now on, we will identify the bimodule maps $'a$ and $a'$ with the corresponding natural transformations of $\res_{U,\theta}$ and $\ind_{U,\theta}$.  So we have
\[
  'a: \res_{U,\theta} \to \res_{U,\theta},\quad \text{and} \quad
  a': \ind_{U,\theta} \to \ind_{U,\theta}.
\]

\begin{proposition}\label{prop:cyclic}
For any $a\in  (\kk[P]\vt\otimes_{\kk[P]} \kk[G])^P$, there are the following equalities of natural transformations:
\begin{enumerate}
  \item $\rcap(a'\otimes \id_{\res_{U,\theta}}) = \rcap(\id_{\ind_{U,\theta}} \otimes {'a})$,

  \item $('a \otimes \id_{\ind_{U,\theta}})\rcup = (\id_{\res_{U,\theta}} \otimes a')\rcup$,

  \item $\lcap('a \otimes \id_{\ind_{U,\theta}}) = \lcap (\id_{\res_{U,\theta}} \otimes a')$,

  \item $(a'\otimes \id_{\res_{U,\theta}})\lcup = (\id_{\ind_{U,\theta}} \otimes 'a)\lcup$.
\end{enumerate}
\end{proposition}

\begin{proof}
For $a\in (\kk[P]\vt\otimes_{\kk[P]} \kk[G])^P$, denote by
$\overline{a}$ the image of $a$ under the natural multiplication map in $\kk[G]$,
\[
  (\kk[P]\vt\otimes_{\kk[P]} \kk[G])^P\to \kk[G]^P, \ \ a\mapsto \overline{a}.
\]
Similarly, for $a\in (\kk[G]\otimes_{\kk[P]} \kk[P]\vt)^P$, we will also denote by $\overline{a}$ the image of $a$ in
$\kk[G]$ under the multiplication map
\[
  (\kk[G]\otimes_{\kk[P]} \kk[P]\vt)^P \to \kk[G]^P, \ \ a\mapsto \overline{a}.
\]
Note that
\[
  \overline{{'a}(\vt\otimes 1)} = \overline{a'(1\otimes \vt)} = \overline{a} \in \kk[G]^P,
\]
so this abuse of notation is mostly harmless.

\begin{enumerate}
  \item The left and right hand sides are given by bimodule maps
    \[
      \kk[G]\otimes_{\kk[P]}\kk[P]\vt\otimes_{\kk[P]}\kk[G] \to \kk[G].
    \]
    For $g\in G$, $b\in L$, the left hand side is the map
    \[
      g\otimes b\vt \otimes h \mapsto (g ab) \otimes h \mapsto g\overline{a}b h\in \kk[G],
    \]
    while the right hand side is the map
    \[
      g\otimes b\vt \otimes h \mapsto g\otimes (bah) \mapsto g b\overline{a} h.
    \]
    But $b\overline{a} = \overline{a}b$, since $b\in P$ and $\overline{a} \in  \kk[G]^P$. Hence the two sides agree.

    \item The left and right hand sides are given by bimodule maps
      \[
        \kk[L] \to \kk[P]\vt\otimes_{\kk[P]}\kk[G]\otimes_{\kk[G]} \kk[G] \otimes_{k[P]} \kk[P]\vt.
      \]
      Write $a\in (\kk[P]\vt\otimes_{\kk[P]} \kk[G])^P$ as $a = \sum_{g\in G}  \lambda_g ( \vt \otimes g)$, $\lambda_g\in \kk$, so that the left hand side is the map which sends
      \[
        1 \mapsto \vt\otimes 1\otimes 1\otimes \vt \mapsto a\otimes 1\otimes \vt = \sum_{g\in G} \lambda_g (\vt \otimes g \otimes 1\otimes \vt),
      \]
      while the right hand side is given by
      \[
        1 \mapsto \vt \otimes 1\otimes 1 \otimes \vt \mapsto \vt \otimes 1\otimes a = \sum_{g\in G} \lambda_g (\vt \otimes 1 \otimes g \otimes \vt),
      \]
      where now
      \[
        a = \sum_{g\in G} \lambda_g (g\otimes \vt) \in (\kk[G]\otimes_{\kk[P]} \kk[P]\vt)^P
      \]
      induces the natural transformation $a'$.  Since $g$ moves across the middle tensor product $\otimes_{\kk[G]}$, the left and right hand sides agree.

  \item The left and right hand sides are given by bimodule maps
    \[
      \kk[P]\vt\otimes_{\kk[P]} \kk[G] \otimes_{k[P]} \kk[P]\vt \to \kk[P]\vt \simeq \kk[L] .
    \]
    The left hand side is given by the bimodule map
    \[
      \vt \otimes g \otimes \vt \mapsto ag\otimes \vt \mapsto \Proj_P(\overline{a}g)\vt.
    \]
    Similarly, the right hand side is given by the bimodule map
    \[
      \vt \otimes g \otimes \vt \mapsto \vt\otimes ga\mapsto \Proj_P(g\overline{a})\vt.
    \]
    The agreement of the two sides is given by the equality $\Proj_P(g\overline{a})=\Proj_P(\overline{a}g)$, for $\overline{a}\in \kk[G]^P$, $g\in G$, which is straightforward to check.

    \item Both sides are bimodule maps
      \[
        \kk[G] \to \kk[G]\otimes_{\kk[P]} \kk[P]\vt \otimes_{\kk[P]} \kk[G].
      \]
      Writing $a = \sum_{g\in G}\lambda_g (\vt\otimes g)$, the bimodule map on the left hand side is determined by the image
      \[
        1 \mapsto \sum_{i=1}^m g_i^{-1} \otimes \vt \otimes g_i \mapsto \sum_{g\in G} \sum_{i=1}^m  \lambda_g (g_i^{-1}g \otimes \vt \otimes g_i).
      \]
      Since the element $\sum_{i=1}^m g_i^{-1} \otimes \vt \otimes g_i$ is independent of the choice of left $P$ coset representatives $g_i$, the left hand side can also be written as
      \[
        1\mapsto \sum_{g\in G} \sum_{i=1}^m\lambda_g (g_{i'}^{-1}p_i^{-1}g \otimes \vt \otimes p_ig_{i'}),
      \]
      for any $p_i\in P$ and bijections $i\mapsto i'$ of $\{1,\hdots,m\}$, one bijection for each $g\in G$.

      The right hand side, on the other hand, is given by
      \[
        1\mapsto   \sum_{i=1}^m g_i^{-1} \otimes \vt \otimes g_i \mapsto \sum_{g\in G} \sum_{i=1}^m  \lambda_g (g_i^{-1} \otimes \vt \otimes g g_i).
      \]
      Now, for each fixed $g$, we may write $gg_i = p_i g_{i'}$, for some $p_i\in P$ and some bijection $i\mapsto i'$ of $\{1,\hdots,m\}$.  Then $g_i^{-1} = g_{i'}^{-1}p_i^{-1}g$, and the right hand side may be written as
      \[
        1\mapsto \sum_{g\in G} \sum_{i=1}^m  \lambda_g (g_{i'}^{-1}p_i^{-1}g \otimes \vt \otimes p_ig_{i'}),
      \]
      which agrees with the left hand side.
\end{enumerate}
\end{proof}

Thus we have the following theorem.
\begin{theorem}\label{thm:cyclicbiadjointness}
The natural transformations $\rcap,\rcup,\lcap,\lcup$ make
$\ind_{U,\theta}$ and $\res_{U,\theta}$ into a cyclic biadjoint pair.
\end{theorem}

\begin{corollary} \label{cor:GL-cyclicity}
The functors $\ind$ and $\res$ from Section~\ref{sec:GL-adjunctions} are a cyclic biadjoint pair.
\end{corollary}

\begin{proof}
This follows immediately from Theorem~\ref{thm:cyclicbiadjointness} after setting $G = GL_n$, $L=GL_{n-1}$, $U = U_{n-1,n}$, and taking $\theta$ to be the trivial character.
\end{proof}

\begin{corollary} \label{cor:Hecke-cyclicity}
The adjunction data of Section~\ref{sec:Hecke-adjunctions} is cyclic when $q$ is an indeterminate or a prime power.
\end{corollary}

\begin{proof}
Corollary~\ref{cor:GL-cyclicity}, together with the equivalence
\[ \textstyle
  \bigoplus_n \Hom_{A_n}(A_nb_n, -) : \bigoplus_n b_nA_nb_n\mbox{-mod} \to \bigoplus_n H_n\mbox{-mod}
\]
of Section~\ref{sec:GL-Fock-rep}, gives the result for $q$ a prime power.  Now suppose that $q$ is an indeterminate.  Proving the adjunction data is cyclic amounts to proving four equalities of bimodule maps as in Proposition~\ref{prop:cyclic}.  After choosing a $\kk[q,q^{-1}]$-basis, such maps can be represented by matrices with entries in $\kk[q,q^{-1}]$.  Since we have the desired equalities when $q$ is specialized to any prime power, the equalities follow for $q$ an indeterminate (here we use the fact that if two elements of $\kk[q,q^{-1}]$ are equal at an infinite number of values of $q$, they are equal).
\end{proof}


\bibliographystyle{abbrv}
\bibliography{biblist}

\end{document}